\documentclass[11pt,reqno]{amsart}
\usepackage{amsfonts,amssymb,amsmath,a4,color,bm,url}

\newtheorem{theorem}{Theorem}
\newtheorem{problem}{Problem}

\newtheorem{proposition}{Proposition}
\newtheorem{lemma}{Lemma}
\newtheorem{corollary}{Corollary}
\theoremstyle{remark}
\newtheorem{remark}{Remark}

\newtheorem*{conjecture}{\bf Conjecture}

\newtheorem{definition}{Definition}

\newcommand\Q{\mathbb{Q}}
\newcommand\R{\mathbb{R}}
\newcommand\Z{\mathbb{Z}}
\newcommand\N{\mathbb{N}}
\newcommand{\ve}{\varepsilon}

\newcommand\II{\mathbb{I}}
\newcommand\cB{\mathcal{B}}
\newcommand\cC{\mathcal{C}}
\newcommand\cQ{\mathcal{Q}}
\newcommand\cN{\mathcal{N}}
\newcommand\cS{\mathcal S}
\newcommand\cW{\mathcal W}
\newcommand\cK{\mathcal K}
\newcommand\cM{\mathcal M}
\newcommand\cV{\mathcal V}

\newcommand{\vvv}[1]{{\mathbf{#1}}}
\newcommand{\Rp}{\R_{+}}
\newcommand{\Mmn}{\mathbb{M}_{n,m}}

\newcommand{\Bad}{\mathbf{Bad}}
\newcommand{\Sing}{\mathbf{Sing}}

\newcommand{\DI}{\mathbf{DI}}

\newcommand{\dminVB}{d_{\min}}
\newcommand{\dminVBone}{d_{\min,1}}
\newcommand{\dminVBtwo}{d_{\min,2}}
\newcommand{\dminVBthree}{d_{\min,3}}
\newcommand{\dminVBi}{d_{\min,i}}

\newcommand{\chaptwo}{2}
\newcommand{\chaptwosec}{}
\begin{document}

\title[Number Theory meets Wireless Communications: an introduction]{Number Theory meets Wireless Communications: an introduction for dummies like us}
\author{Victor Beresnevich \and  Sanju Velani}

\begin{abstract}
In this chapter we introduce the theory of Diophantine approximation via a series of basic examples from information theory relevant to wireless communications. In particular, we discuss Dirichlet's theorem, badly approximable points, Dirichlet improvable and singular points, the metric (probabilistic) theory of Diophantine approximation including the Khintchine-Groshev theorem and the theory of  Diophantine approximation on manifolds. We explore various number theoretic approaches used in the analysis of communication characteristics such as Degrees of Freedom (DoF).
In particular, we  improve the result of Motahari
et al regarding
the DoF of a two-user X-channel. In essence, we show that
the total DoF can be achieved for all (rather than almost all) choices of channel
coefficients with the exception of a subset of strictly smaller dimension than the ambient space. The
improvement utilises the concept of jointly non-singular points that we introduce and a general result of
Kadyrov et al on the $\delta$-escape of mass in the space of lattices. We also discuss follow-up open problems that incorporate  a breakthrough of Cheung and more generally  Das et al on the dimension of the set of singular points.
\end{abstract}

\maketitle

\vspace*{5ex}

{\footnotesize\sf
\noindent \underline{Note}: The copy of the chapter, as displayed on this website, is a draft, pre-publication copy only. The final, published version of the Chapter is a part of an edited volume entitled ``Number Theory meets Wireless Communications'' which shall be available for purchase from the publisher (Springer Nature) and other standard distribution channels. This draft copy is made available for personal use only.

\thispagestyle{empty}

\vspace*{20ex}

\hangindent=5ex\hangafter=1
\noindent VB\,:~ Department of Mathematics, University of York, Heslington, York, YO10 5DD, UK\\
{\tt victor.beresnevich@york.ac.uk, victor.beresnevich@gmail.com}

\hangindent=5ex\hangafter=1
\noindent SV\,:~ Department of Mathematics, University of York, Heslington, York, YO10 5DD, UK\\
{\tt sanju.velani@york.ac.uk}

}

\newpage

\tableofcontents

\newpage

\section{Basic examples and fundamentals of Diophantine approximation}  \label{beeg}

Let us start by addressing a natural question that a number theorist or more generally a mathematician  who has picked up this book may well ask:
{\em  what is the role of number theory  in the world of wireless communications? } We will come clean straightaway and say that by number theory we essentially mean areas such as Diophantine approximation and the geometry of numbers, and by wireless communication we essentially mean the design and analysis of lattice/linear codes for wireless communications which thus falls in the realm of information theory. To begin with, with this confession in mind, let us start by describing the role of one-dimensional Diophantine approximation.
Recall, that at the heart of Diophantine approximation is the classical theorem of Dirichlet on rational approximations to real numbers.

\begin{theorem}[Dirichlet\index{Dirichlet's theorem}, 1842]\label{Dir}
  For any $\xi\in\R$ and any $Q \in \mathbb{N}$ there exist
  $p,q \in \mathbb{Z}$ such that
\begin{equation}\label{eqn01}
    \left| \xi - \frac{p}{q} \right| < \frac{1}{qQ}  \qquad \textrm{ and }  \qquad  1 \leq  q \leq Q \, .
  \end{equation}
\end{theorem}

The proof can be found in many elementary number theory books and makes use of the wonderfully simple yet  powerful Pigeonhole Principle: if $n$ objects are placed in $m$ boxes and $ n > m $, then some box will contain at least two objects. See, for example, \cite[\S1.1]{MR3618787} for details. An easy consequence of the above theorem is the following statement.

 \begin{corollary}\label{cor1}
  Let $\xi \in \mathbb{R}\setminus\Q$, that is $\xi$ is a real irrational number. Then there exist infinitely many
  reduced rational fractions $p/q$ ~$(p,q\in\Z)$ such that
  \begin{equation}\label{eqn02}
    \left|\xi-\frac{p}{q}\right| < \frac{1}{q^2}.
  \end{equation}
\end{corollary}

The following exposition illustrates one of the many aspects of the role of Diophantine approximation in wireless communication. In particular, within this section we consider a basic example of a communication channel which brings into play the theory of Diophantine approximation. In \S\ref{sec2} we consider  a slightly more sophisticated example which also brings into play the theory of Diophantine approximation in higher dimensions.  This naturally feeds into \S\ref{sec3} in which the role of the theory of Diophantine approximation of dependent variables is discussed. The latter is also referred to as Diophantine approximation on manifolds since the parameters of interest are confined by some functional relations. To begin with, we consider a `baby' example of a communication channel intended to remove the language barrier for mathematicians and explicitly expose an aspect of communications that invites the use of Diophantine approximation.

\subsection{A `baby' example} \label{babyexample}

Suppose there are two {\em users} $S_1$ and $S_2$ wishing to send (\emph{transmit}) their {\em messages} $u_1$ and $u_2$ respectively along a shared (radio/wireless) communication channel to a \emph{receiver} $R$.
For obvious reasons, users are often  also referred to as transmitters.  Suppose for simplicity that $u_1,u_2 \in \{ 0,1\} $. Typically, prior to transmission, every message is encoded with what is called a {\em codeword}. Suppose that $x_1=x_1(u_1)$ and $x_2=x_2(u_2)$ are the codewords that correspond to $u_1$ and $u_2$. In general, $x_1$ and $x_2$ could be any functions on the set of messages. In principle, one can take $x_1=u_1$ and $x_2=u_2$. When the codewords $x_1$ and $x_2$ are being transmitted along a wireless communication channel, there is normally a certain degree of fading of the transmitted signals.
This for instance could be dependent on the distance of the transmitters
from the receiver and the reflection caused by obstacles such as  buildings in the path of the signal.
Let $h_1$ and $h_2$  denote the fading factors (often referred to as \emph{channel gains} or \emph{channel coefficients} or {\em paths loss}) associated with the transmission of signals from $S_1$ and $S_2$ to $R$ respectively. These are strictly positive numbers and for simplicity we will assume that their sum is one: $h_1+h_2=1$. Mathematically, the meaning of the channel coefficients is as follows: if $S_i$ transmits signal $x_i$,  the receiver $R$ observes $h_ix_i$. However, due to fundamental physical properties of wireless medium, when $S_1$ and $S_2$ simultaneously use the same wireless communication channel, $R$ will receive the superposition of $h_1x_1$ and $h_2x_2$, that is
\begin{equation}\label{eqn03}
y = h_1 x_1  +  h_2 x_2\, .
\end{equation}
For instance, assuming that $x_1=u_1$ and $x_2=u_2$, the outcomes of $y$ are
\begin{equation}\label{eqn04}
y=\begin{cases}
0 &\text{ if  } \quad  u_1=u_2=0\, \\[0ex]
h_1 &\text{ if  }\quad  u_1=0 \text{ and } u_2=1\,,\\[0ex]
h_2 &\text{ if  } \quad u_1=1 \text{ and } u_2=0\,,\\[0ex]
1= h_1+h_2  &\text{ if  }  \quad u_1=u_2=1 \, .
\end{cases}
\end{equation}
A pictorial description of the above setup is given below in Figure~\ref{fig2vbone}.

\setlength{\unitlength}{20mm}

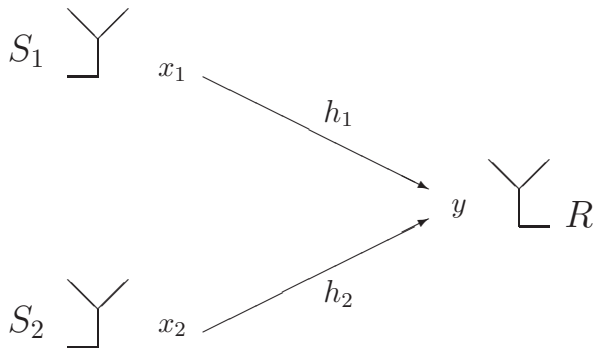
\begin{figure}[!ht]
\begin{center}
\begin{picture}(2, 2)(-0.75,-0.75)

\put(-1.5,0.8){\line(0,1){0.25}}
\put(-1.5,1.05){\line(-1,1){0.2}}
\put(-1.5,1.05){\line(1,1){0.2}}
\put(-1.5,0.8){\line(-1,0){0.2}}
\put(-1.1,0.8){$x_1$}

\put(1.3,-0.2){\line(0,1){0.25}}
\put(1.3,0.05){\line(-1,1){0.2}}
\put(1.3,0.05){\line(1,1){0.2}}
\put(1.3,-0.2){\line(1,0){0.2}}
\put(0.85,-0.1){$y$}

\put(-1.5,-1){\line(0,1){0.25}}
\put(-1.5,-0.75){\line(-1,1){0.2}}
\put(-1.5,-0.75){\line(1,1){0.2}}
\put(-1.5,-1){\line(-1,0){0.2}}
\put(-1.1,-0.9){$x_2$}

\put(-0.8,0.8){\vector(2,-1){1.5}}
\put(-0.8,-0.9){\vector(2,1){1.5}}

\put(-2.1,0.9){\Large{$S_1$}}
\put(-2.1,-0.9){\Large{$S_2$}}
\put(1.6,-0.2){\Large{$R$}}

\put(0,0.5){\large{$h_{1}$}}
\put(0,-0.7){\large{$h_{2}$}}

\end{picture}
\vspace*{4ex}
\end{center}
\caption{Two user Multiple Access Channel (no noise).}
\label{fig2vbone}
\end{figure}

The ultimate goal is for the receiver $R$ to identify ({\em decode}) the messages $u_1$ and $u_2$ from the observation of $y$. For example, with reference to \eqref{eqn04}, assuming the channel coefficients $h_1$ and $h_2$ are known at the receiver and are different, that is $h_1 \neq h_2$, the receiver is obviously able to do so. However, in real life there is always a degree of error in the transmission process, predominantly caused by the received signal $y$ being corrupted by (additive) \emph{noise}. The noise can result from a combinations of various factors including the interference of other users and natural electromagnetic radiation.
In short, if  $z$ denotes the noise, then instead of \eqref{eqn03}, $R$  receives the signal
\begin{equation}\label{eqn05}
y' = y+z=h_1 x_1  +  h_2 x_2  + z\,.
\end{equation}
Equation \eqref{eqn05} represents one the simplest models of what is known as an {\em Additive White Gaussian Noise Multiple Access Channel} (AWGN-MAC), see Chapter~\chaptwo{}
for a formal definition.  As before, the goal for the receiver $R$ remains to decode   the messages $u_1$ and $u_2$, but now from the observation of $y'=y+z$.
Let $\dminVB$ denote the {\em minimum distance} between the four outcomes of $y$. Then as long as the absolute value $|z|$ of the noise is strictly less than $\dminVB/2$, the receiver is  able to recover $y$ and consequently the messages $u_1$ and $u_2$ from the value of $y'$. This is simply due to the fact that the intervals of radius  $\dminVB/2$ centered at the four outcomes of $y$ are disjoint and $y'$ will lie in exactly one of these intervals, see Figure~\ref{f2}. In other words, $R$ is able to identify $y$ by rounding $y'$ to the closest possible outcome of $y$.

\

\setlength{\unitlength}{15mm}

\begin{figure}[h!]
\begin{center}
\begin{picture}(0,0)
\put(-3,0){\line(1,0){6}}
\put(-3,-0.1){\line(0,1){0.2}}
\put(3,-0.1){\line(0,1){0.2}}
\put(-1,-0.1){\line(0,1){0.2}}
\put(1,-0.1){\line(0,1){0.2}}

\put(-3.025,-0.3){0}
\put(2.975,-0.3){1}
\put(-1.1,0.2){$h_1$}
\put(0.9,0.2){$h_2$}

\put(-3,0.01){\line(1,0){0.5}}
\put(-3,-0.01){\line(1,0){0.5}}

\put(-1.5,0.01){\line(1,0){1}}
\put(-1.5,-0.01){\line(1,0){1}}

\put(0.5,-0.01){\line(1,0){1}}
\put(0.5,0.01){\line(1,0){1}}

\put(2.5,-0.01){\line(1,0){0.5}}
\put(2.5,0.01){\line(1,0){0.5}}

\put(-2.55,-0.05){\textbf{)}}

\put(-0.55,-0.06){\textbf{)}}
\put(-1.52,-0.06){\textbf{(}}

\put(1.45,-0.06){\textbf{)}}
\put(0.48,-0.06){\textbf{(}}

\put(2.48,-0.06){\textbf{(}}

\put(-3,0.2){\vector(1,0){0.5}}
\put(-2.5,0.2){\vector(-1,0){0.5}}

\put(-2.8,0.3){$z$}

\end{picture}
\end{center}
\caption{Separation of intervals of radius $|z|$ around each possible outcome of $y$ which contain the values of $y'$.}
\label{f2}
\end{figure}
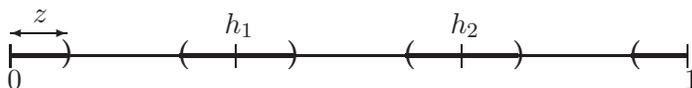

\noindent For example, it is easy to see that the maximum separation between the four outcomes given by \eqref{eqn04} is attained when $h_1 = 1/3$ and $h_2 = 2/3$.   In this case $\dminVB = 1/3$,  and we are able to recover the messages $u_1$ and $u_2$ assuming that  $|z| < 1/6$. The upshot of the above discussion is the following simple but fundamental conclusion.

\medskip

\noindent\textbf{Conclusion:} The greater the mutual separation $\dminVB$ of the outcomes of $y$,  the better the tolerance for noise we have during the transmission of the signal.

\medskip

In information theory achieving good separation between received signals translates into obtaining good lower bounds on the fundamental parameters of communication channels such as Rates-of-Communications, Channel Capacity and Degrees-of-Freedom, see Chapter~\chaptwo{} for formal definitions of these notions. Within this chapter we will concentrate on the role of Diophantine approximation in answering the following natural and important question:

  {\em How can a good separation of received signals be achieved and how often?}

 \noindent Indeed, to some extent, answering this and related questions using the tools of Diophantine approximation, algebraic number theory and the geometry of numbers is a reoccurring theme throughout the whole book. We will solely use {\em linear encoding} to achieve `good' separation. In particular, within the above `baby' example, one is able to achieve the optimal separation ($\dminVB=1/3$) at the receiver regardless of the values of $h_1$ and $h_2$ by applying the following simple linear encoding of the messages $u_1$ and $u_2$:
$$
x_1= \frac13 h_1^{-1} u_1 \qquad \text{and}\qquad   x_2 = \frac23 h_2^{-1} u_2 \,.
$$
Indeed,  before taking noise into consideration, under the above encoding the received signals become
\begin{equation}\label{eqn06}
y =   h_1 x_1 +  h_2 x_2  =   \textstyle{\frac13}  u_1 + \textstyle{ \frac23}  u_2
=\begin{cases}
0 &\text{ if  } \quad  u_1=u_2=0\,, \\[0ex]
1/3 &\text{ if  }\quad  u_1=0 \text{ and } u_2=1\,,\\[0ex]
2/3  &\text{ if  } \quad u_1=1 \text{ and } u_2=0\,,\\[0ex]
1  &\text{ if  }  \quad u_1=u_2=1 \, .
\end{cases}
\end{equation}

To summarise, the above discussion brings to the forefront  the importance of maximizing the minimal distance/separation $\dminVB$ of the received (noise-free) signals and at the same time indicates how a linear encoding allows us to achieve this.   Nevertheless,  the assumption that the messages $u_1$ and $u_2$ being sent by the transmitters $S_1$ and $S_2$ are binary in nature makes the discussion  over simplistic -- especially in terms of the use of number theory to analyse the outcomes.   We now modify the `baby' example to a more general situation  in which $S_1$ and $S_2$ wish to send messages $u_1$ and $u_2$ from the set of integers  $\{0,\dots,Q\}$ to a single receiver $R$.

\subsection{Example 1 (modified `baby' example)}\label{example1}
Unless stated otherwise,  here and throughout,  $Q\ge2$  is a fixed integer.  As we shall see, this slightly more complex setup, in which $u_1,u_2\in\{0,\dots,Q\}$, naturally  bring into play the rich theory of Diophantine approximation. So with this in mind, let us assume that the codewords $x_1$ and $x_2$ that are being transmitted by $S_1$ and $S_2$ are simply obtained by the linear encoding of the messages $u_1$ and $u_2$ as follows
\begin{equation}\label{van}
x_1= \alpha u_1    \quad {\rm and } \quad x_2= \beta u_2\,  \quad  \quad (0 \le u_1,u_2 \le Q) \,,
\end{equation}
where $\alpha$ and $\beta$ are some positive real numbers.  We emphasise that the  parameters $\alpha$ and $\beta$ are at our disposal and this fact will be utilized later.  As in the `baby' example let $h_1$ and $h_2$  denote the channel coefficients  associated with $S_1$ and $S_2$ respectively.
Then,  before taking noise into account, $R$  will receive the signal
\begin{equation}\label{eqn07}
y = h_1 x_1  +  h_2 x_2=h_1\alpha u_1+h_2\beta u_2\, .
\end{equation}
Clearly, $y$ takes the values
\begin{equation}\label{eqn08}
h_1\alpha u_1+h_2\beta u_2  : 0 \le u_1,u_2 \le Q \, .
\end{equation}
Thus, there are potentially $(Q+1)^2$ distinct outcomes of $y$ and they lie in the interval $[0, (h_1\alpha+h_2\beta)Q]$.  It is easily verified  that if they were equally separated then their mutual separation  would be precisely  \begin{equation}\label{eqn09}
\frac{h_1\alpha+h_2\beta}{Q+2}  \, .
  \end{equation}
However, this is essentially never the case. Indeed, let $\dminVB$ denote the minimal distance between the points $y$ given by \eqref{eqn08}. Without loss of generality, suppose for the sake of simplicity that
$$
0<h_1\alpha<h_2\beta
$$
and define the real number
\begin{equation}\label{eqn10}
\xi:=\frac{h_1\alpha}{h_2\beta}\,,
\end{equation}
which in view of the above assumption is between $0$ and $1$; {\em i.e.} $0<\xi<1$.
Then, by Dirichlet's theorem, we have that
\begin{equation}\label{eqn11}
\left|\xi-\frac pq\right|  \ \le  \ \frac{1}{qQ}
\end{equation}
for an integer pair $(p,q) \in \Z^2$ satisfying $1\le q\le Q$. Since $0<\xi<1$ and $1\le q\le Q$, we also have that $0\le p\le q\le Q$. On  multiplying \eqref{eqn11} by $h_2\beta q$, we find that
\begin{equation}\label{eqn12}
|h_1\alpha q-h_2\beta p|  \ \le \  \frac{C_1}Q\, \qquad (C_1:=h_2\beta\,),
\end{equation}
for some  integer pair $(p,q) \in \Z^2$  satisfying $1\le q\le Q$ and $0\le p\le q$. Now observe that the quantity  $|h_1\alpha q-h_2\beta p|$ on the left hand side   of \eqref{eqn12} is exactly the distance between the two specific values of $y$ within \eqref{eqn08} corresponding to $u_1=q, u_2=0$ and $u_1=0, u_2=p$.  Since $q\neq0$, this demonstrates that the minimal distance $\dminVB$ between the  values of $y$ given by \eqref{eqn08} is  always bounded above by  $C_1/Q$; i.e.
\begin{equation}\label{eqn13}
  \dminVB\le \frac{C_1}{Q}\, .
\end{equation}
For all intents and purposes, this bound on the minimal distance  is smaller  than the hypothetical `perfect' separation given by \eqref{eqn09}. In general, we have that
$$
\dminVB\le \min\left\{\frac{C_1}{Q}\,,\frac{h_1\alpha+h_2\beta}{Q+2}\right\}.
$$
It is easily seen that we can remove the assumption that $0<h_1\alpha<h_2\beta$ if we put  $ C_1 = \max\{h_1\alpha,h_2\beta\}$.

\begin{remark} \label{vanishing}
On looking at \eqref{eqn13}, the reader may be concerned (rightly)
that the minimal distance $\dminVB$ vanishes as $Q$ grows. Luckily,
this can be easily rectified by introducing a scaling factor
$\lambda \geq 1$ into the linear encoding of the messages $u_1$ and
$u_2$. The point of doing this is that the codeword $x_1$ (resp.
$x_2$) given by \eqref{van} becomes $\lambda \alpha u_1 $ (resp.
$\lambda \beta u_2$) and this has no effect on the point of
interest $\xi$ given by \eqref{eqn10} but it scales up by $\lambda$
the constant $C_1$ appearing in \eqref{eqn12}. Thus, by choosing
$\lambda$ appropriately (namely, proportional to $Q$) we
can avoid the right hand side of \eqref{eqn13} from vanishing as
$Q$ grows. In subsequent more `sophisticated' examples, the scaling
factor will be relevant to the discussion and will appear at the
point of linear encoding the messages.
\end{remark}

Now let us bring noise into the above setup. As in the `baby' example, if $z$ denotes the (additive) noise, then instead of \eqref{eqn07}, $R$  receives the signal
\begin{equation}\label{eqn14}
y' = y+z=h_1\alpha u_1+h_2\beta u_2+z\,.
\end{equation}
Note that as long as the absolute value $|z|$ of the noise is strictly less than $\dminVB/2$, the receiver $R$  is  able to recover $y$ and consequently $u_1$ and $u_2$ from the value of $y'$.
Commonly, the nature of noise is such that $z$ is a random variable having normal distribution. Without loss of generality we will assume that  $z\sim \cN(0,1)$, that is the mean value of noise is $0$ and its variance is $1$.  Therefore, when taking the randomness of noise into account, the problem of whether or not the receiver is able to recover messages sent by the transmitters becomes probabilistic in nature. Loosely speaking,  we are interested in the probability that $|z| < \dminVB/2$ -- the larger the probability the more likely the receiver is able to recover messages by rounding $y'$ to the closest possible outcome of $y$. Of course, if it happens that $|z| \ge \dminVB/2$, then we will have an error in the recovery of $y$ and thus the messages $u_1$ and $u_2$. When $z\sim \cN(0,1)$, the probability of this error can be computed using the Gauss error function and is explicitly equal to
$$
1-\sqrt{2/\pi}\int^{\dminVB/2}_0e^{-\theta^2/2}d\theta\,.
$$
This gets smaller as $\dminVB$ gets larger. Clearly, in view of the theoretic upper bound on $\dminVB$ given by \eqref{eqn13} the probability of error is bounded above by the probability that $|z|<C_1/2Q$. Thus, the closer $\dminVB$ is to the theoretic upper bound, the closer we are to minimizing the probability of the error and in turn the higher the threshold for tolerating noise.  With this in mind, we now demonstrate that on appropriately choosing the parameters   $\alpha$ and $\beta$  associated with the encoding procedure  it is possible to get within a constant factor of the theoretic upper bound.

\subsection{Badly approximable numbers}\label{Bad1}
The key is to make use of the existence of badly approximable numbers - a fundamental class of real numbers in the theory of Diophantine approximation.

\begin{definition}[Badly approximable numbers]
A real number $\xi$ is said to be {\em badly approximable} if there exists a constant $\kappa=\kappa(\xi)>0$ such that for all $q\in\N$, $p\in\Z$
  \begin{equation}\label{eqn15}
    \left|\xi-\frac{p}{q}\right| \ge \frac{\kappa}{q^2}\,.
  \end{equation}
\end{definition}

\noindent Note that by definition,  badly approximable numbers are precisely those  real numbers for which the right hand side of inequality  \eqref{eqn02} associated with  Dirichlet's corollary (Corollary~\ref{cor1}) cannot be `improved' by an arbitrary constant factor.   By Hurwitz's theorem \cite{MR3618787}, if $\xi$ is badly approximable then for the associated badly approximable constant $\kappa(\xi) $ we have that
$$
0 < \kappa(\xi) <  1/\sqrt5  \, .
$$
It is well known that the set of badly approximable numbers can be characterized as those  real numbers whose continued fraction expansions have bounded partial quotients.  Moreover, an irrational number has a periodic continued fraction expansion if and only if it is a quadratic irrational and thus every quadratic irrational is badly approximable.    In particular,
it is easily verified that for any given $ \ve >0$,
the golden ratio $$\gamma:=(\sqrt5+1)/2 $$ satisfies inequality \eqref{eqn15} with $ \kappa =1/(\sqrt5+\ve)$ for all $p \in \Z$ and $q \in \N$ with $q^2\ge 1/(\sqrt5\ve)$. This is obtained using the standard argument that involves substituting $p/q$ into the minimal polynomial $f$ of $\gamma$ over $\Z$ and using the obvious fact that $1\le q^2|f(p/q)|\le  q^2 |\gamma-p/q|\cdot|\bar{\gamma}-p/q|$, where $\bar{\gamma} = (\sqrt5 - 1)/2$ is the conjugate of $\gamma$. We leave further computational details to the reader. Observe that on taking $\ve=1/\sqrt5$,  we find that $\gamma$ is badly approximable with $\kappa(\gamma)\ge \sqrt5/6$.

The reason for us bringing into play the notion of badly approximable numbers is very easy to explain.  By definition, on choosing the parameters $\alpha$ and $\beta$ so that $\xi:=h_1\alpha/h_2\beta$ is badly approximable guarantees the existence of a constant $\kappa(\xi) > 0$ such that
$$
|h_1\alpha q-h_2\beta p|  \ \geq \ \kappa(\xi) \frac{C_1}{q} \,   \qquad    \forall \ q \in \N, \; p \in \Z \, .
$$
Thus, it follows that the separation between the points given by \eqref{eqn08} is  at least $\kappa(\xi) C_1/Q$. In other words, the minimal distance $\dminVB$  is within a constant factor of the theoretic upper bound $C_1/Q$ given by \eqref{eqn13}. Indeed, if we choose $\alpha$ and $\beta$ so that $h_1\alpha/h_2\beta$  is the golden ratio $\gamma$ we obtain that
\begin{equation}\label{eqn16}
\kappa(\gamma)\frac{C_1}{Q}\le \dminVB\le \frac{C_1}{Q}\,.
\end{equation}
The upshot is that equation \eqref{eqn16} gives an explicit `safe' threshold for the level of noise that can be tolerated. Namely, the probability that  $|z| < \dminVB/2$ is at least the probability that  $|z|<\kappa(\gamma) C_1/Q$.
In principle, one can manipulate the values of $Q\in \N$ and $\ve > 0$ within the above argument to improve the lower bound in \eqref{eqn16}. However,  any such manipulation will not enable us to surpass the hard lower bound limit of $C_1/(\sqrt5Q)$ imposed by the aforementioned consequence  of Hurwitz's theorem.  Therefore, we now explore a different  approach in an attempt to make improvements to \eqref{eqn16} beyond this hard limit.  Ideally, we would like to replace $ 1/\sqrt5$ by a constant arbitrarily close to one. We would also like to  move away from insisting that $\xi$ is badly approximable since this is a rare event. Indeed,  although the set of badly approximable number is of full Hausdorff dimension (a result of Jarn\'ik from the 1920s), it is a set of Lebesgue measure zero (a result of Borel from 1908). In other words, the (uniform) probability that a real number in the unit interval  is  badly approximable is zero. We will return to this in  \S\ref{BAD20}  and \S\ref{FR} below.

\subsection{Probabilistic aspects}\label{Prob1}
The approach we now pursue is motivated by the following probabilistic problem:
 {\em Given $0<\kappa'<1$ and $Q\in\N$, what
is the probability that a given real number  $\xi\in \II:=(0,1)$ satisfies
\begin{equation}\label{eqn17}
\left|\xi-\frac pq\right|\ge\frac{\kappa'}{qQ}
\vspace*{2ex}
\end{equation}
for all integers $p$ and $1\le q\le Q$?}
\noindent Note that these are the real numbers for which the right hand side of inequality  \eqref{Dir} associated with  Dirichlet's theorem  cannot be improved by the factor of $\kappa'$ ($Q$ is fixed here). It is worth mentioning at this point, in order to avoid confusion later,  that these real numbers are not the same as Dirichlet non-improvable numbers which will be introduced below in \S\ref{Improve1}.
To estimate the probability in question, we consider the complementary inequality
\begin{equation}\label{eqn18}
\left|\xi-\frac pq\right|<\frac{\kappa'}{qQ}\,.
\end{equation}
Let $1\le q\le Q$. Then for a fixed $q$,   the probability that a given $\xi\in \II:=(0,1)$ satisfies \eqref{eqn18}
for some $p\in\Z$ is exactly $2\kappa'/Q$  -- it corresponds to the measure of the set
$$
E_q := \bigcup_{p \in \Z} \textstyle{\Big(\frac{p}{q} -  \frac{\kappa'}{qQ} , \frac{p}{q} + \frac{\kappa'}{qQ}\Big)}  \, \cap \,  \II \, .
$$
On summing up these probabilities over $q$,  we conclude that the probability that a given $\xi\in \II$ satisfies \eqref{eqn18} for some integers $p$ and $1\le q\le Q$  is trivially  bounded above by $ 2 \kappa'$. This in turn implies that
for any $\kappa'<1/2$ and any $Q\in\N$ the probability that \eqref{eqn17} holds for all integers $p,q$ with $1\le q\le Q$ is at least
$$
1-2\kappa'\,.
$$
The following result shows that with a little more extra work it is  possible to improve this trivial bound.

\begin{lemma} \label{mum}
For any $0<\kappa'<1$ and any $Q\in\N$ the probability that \eqref{eqn17} holds for all integers $p,q$ with $1\le q\le Q$ is at least
\begin{equation}\label{eqn19}
1-\frac{12\kappa'}{\pi^2}\approx1-1.216\kappa'\,.
\end{equation}
\end{lemma}

\begin{remark}  Observe that  when
$$
\kappa'<\pi^2/12\approx0.822  \, ,
$$
the quantity $ 12 \kappa'/\pi^2$ is strictly less than $1$ and therefore the probability given by \eqref{eqn19} is greater than zero. Hence for any $Q \in \N$,  there exist real numbers $\xi$  satisfying \eqref{eqn17} for all integers $p$ and $1\le q\le Q$.
\end{remark}

\begin{remark}\label{vb3978}
Within Lemma~\ref{mum} the word `probability' refers to the uniform probability over $[0,1]$. However, in real world applications the parameter $\xi$ appearing in \eqref{eqn17} may not necessarily be a uniformly distributed random variable. For instance, the channel coefficients could be subject to Rayleigh distribution and this will have an obvious effect on the distribution of $\xi$ via \eqref{eqn10}. Nevertheless, as long as the distribution of $\xi$ is absolutely continuous, a version of Lemma~\ref{mum} can be established, albeit the constant that accompanies $\kappa'$ will be different. For further details we refer the reader to \cite{MR3545930}.
\end{remark}

\begin{proof}
The proof of Lemma \ref{mum} relies on `removing' the overlaps between the different sets $E_q$ as $q$ varies. Indeed, it is easily seen that
$$
E:=\bigcup_{q=1}^QE_q=\bigcup_{q=1}^Q\bigcup_{\substack{0\le p\le q\\[0.5ex] \gcd(p,q)=1}}\textstyle{\Big(\frac{p}{q} -  \frac{\kappa'}{qQ} , \frac{p}{q} + \frac{\kappa'}{qQ}\Big)}  \, \cap \,  \II \,.
$$
Therefore,
\begin{equation}\label{eqn20}
\textbf{Prob}(E)\le\sum_{q=1}^Q\sum_{\substack{1\le p\le q\\[0.5ex] \gcd(p,q)=1}}\frac{2\kappa'}{qQ}=
\sum_{q=1}^Q\frac{2\kappa'\varphi(q)}{qQ}=\frac{2\kappa'}{Q}\sum_{q=1}^Q\frac{\varphi(q)}{q}\,,
\end{equation}
where $\varphi$ is the Euler function. To estimate the above sum, it is convenient to use the M\"obius inversion formula, which gives that
$$
\frac{\varphi(q)}{q}=\sum_{d|q}\frac{\mu(d)}{d}
$$
where $\mu$ is the \emph{M\"obius function}. Recall that
$$
\sum_{d=1}^\infty\frac{\mu(d)}{d^2}=\frac{1}{\zeta(2)}=\frac{6}{\pi^2}\,.
$$
Then
\begin{align*}
\sum_{q=1}^Q\frac{\varphi(q)}{q}&=\sum_{q=1}^Q\sum_{d \mid q} \frac{\mu(d)}{d}
=\sum_{q=1}^Q\sum_{dd'=q} \frac{\mu(d)}{d}\\[1ex]
&=\sum_{dd'\le Q} \frac{\mu(d)}{d}
=\sum_{1\le d\le Q}\frac{\mu(d)}{d}\sum_{d'\le Q/d} 1\\[1ex]
&=\sum_{1\le d\le Q}\frac{\mu(d)}{d}[Q/d] \le Q\sum_{1\le d\le Q}\frac{\mu(d)}{d^2}\\[1ex]
&\le\frac{6Q}{\pi^2}\,.
\end{align*}
Combining this with \eqref{eqn20} gives the required estimate, that is a lower bound on $1-\textbf{Prob}(E)$, the probability of the complement to $E$.
\end{proof}

Let  $0<\kappa'<\pi^2/12$ and  $Q\in\N$ be given.
 The upshot of the above discussion is that there exist parameters $\alpha$ and $\beta$ so that with probability  greater than $1-12\kappa'/\pi^2>0$, the real number  $\xi:=h_1\alpha/h_2\beta$  satisfies \eqref{eqn17} for all integers $p$ and $1\le q\le Q$. It follows  that for such $\xi$ (or equivalently  parameters $\alpha$ and $\beta$) the separation between the associated  points given by \eqref{eqn08} is  at least $\kappa' C_1/Q$ and so the minimal distance $\dminVB$ satisfies
\begin{equation}\label{eqn21}
\kappa'\frac{C_1}{Q}\le \dminVB\le \frac{C_1}{Q}\,.
\end{equation}
In particular, we can choose $\kappa'$ so that $\kappa(\gamma)<\kappa'$ in which case the lower bound in \eqref{eqn21} is better than that in  \eqref{eqn16} obtained by making use of badly approximable numbers.   That is to say,  that the lower bound involving $\kappa'$ is closer to the theoretic upper bound $C_1/Q$.  Moreover, the set of badly approximable numbers is a set of measure zero whereas the set of real numbers  satisfying \eqref{eqn17} for all  integers $p$ and $1\le q\le Q$ has Lebesgue measure at least $1-12\kappa'/\pi^2$.  This is an important advantage of the probabilistic approach since in reality it is often the case that the channel coefficients $h_1$ and $h_2$ are random in nature. For example,  when dealing with mobile networks one has to take into consideration the obvious fact that the transmitters are not fixed. The upshot is that in such a scenario, we do  not have the luxury of specifying a particular choice of the parameters $\alpha$ and $\beta$ that leads to the corresponding points given by \eqref{eqn08} being well separated as in the sense of \eqref{eqn16}.  The probabilistic approach provides a way out.  In short, it  enables us to ensure that the minimal distance $\dminVB$  between the points given by  \eqref{eqn08} satisfies \eqref{eqn21} with good (explicitly computable) probability.  See \cite[Section VI.B]{MR3215324} for a concrete example where the above probabilistic approach is used for the analysis of the capacity of symmetric Gaussian multi-user interference channels.

Up to this point, $Q$ has been a fixed integer greater than or equal to 2 and reflects the size of the set of messages. We end our discussion revolving around Example 1 by  considering  the scenario  in which we have complete freedom in choosing $Q$.
In particular, one is often interested in the effect of allowing $Q $ to tend to infinity on the model under consideration.
This is relevant to understanding the so-called Degrees of Freedom (DoF) of communication channels, see \S\ref{sv2.4}.

\subsection{Dirichlet improvable and non-improvable numbers}\label{Improve1}
We now show that there are special values of $Q$ for which  the minimal distance $\dminVB$ satisfies \eqref{eqn21} with  $\kappa'$ as close to one as desired.   The key is to exploit the (abundant) existence of numbers for which Dirichlet's theorem cannot be improved.  Note that in the argument leading to \eqref{eqn16} we made use of the existence of badly approximable numbers; that is numbers for with Dirichlet's corollary cannot be improved.

\begin{definition}[Dirichlet improvable and non-improvable numbers]
Let $0<\kappa'<1$. A real number $\xi$ is said to be {\em $\kappa'$-Dirichlet improvable} if for all sufficiently large $Q\in\N$ there are integers $p$ and $1\le q\le Q$ such that
$$
    \left|\xi-\frac{p}{q}\right| < \frac{\kappa'}{qQ}\,.
$$
A real number $\xi$ is said to be {\em Dirichlet non-improvable} if for any $\kappa'<1$ it is not $\kappa'$-Dirichlet improvable. In other words, a real number $\xi$ is {\em Dirichlet non-improvable} if for any $0<\kappa'<1$ there exists arbitrarily large $Q\in\N$ such that for all integers $p$ and $1\le q\le Q$
$$
    \left|\xi-\frac{p}{q}\right| \ge \frac{\kappa'}{qQ}\,.
$$
\end{definition}

\vspace*{2ex}

\noindent  A well know  result of Davenport $\&$ Schmidt \cite{MR0272722}  states that:
$$
\text{\em a real number is Dirichlet non-improvable}
$$
$$
\Updownarrow
$$
$$
\text{\em it is not badly approximable.}
$$
Consequently, a randomly picked real number is Dirichlet non-improvable with probability one.  The upshot of this is the following remarkable consequence: for any random choice of channel coefficients $h_1$, $h_2$ and  parameters $\alpha$, $\beta$,  \textbf{with probability one} for any $\ve>0$ there exist arbitrarily large integers $Q$ such that the minimal  distance $\dminVB$  between the associated points  given by \eqref{eqn08} satisfies
$$
(1-\ve)\frac{C_1}{Q}\le \dminVB\le \frac{C_1}{Q}\,.
$$
Clearly, this is the best possible outcome for the basic wireless communication model considered in Example 1.  We now consider a slightly more sophisticated model which  demonstrates the role of higher dimensional Diophantine approximation in wireless communication.

\section{A `toddler' example and Diophantine approximation in higher dimensions }\label{sec2}

The discussion in this section is centred on analysing the model arising from adding another receiver within the setup of the modified `baby' example.

\subsection{Example 2} \label{eg2section}

Suppose there are two users $S_1$ and $S_2$ as in Example~1 but this time there are also two receivers $R_1$ and $R_2$.  Let $Q  \ge 1$ be an integer and suppose $S_1$  wishes to simultaneously transmit independent messages $u_1,v_1\in\{0,\dots,Q\}$, where $u_1$ is intended for  $R_1$  and $v_1$  for $R_2$. Similarly, suppose $S_2$  wishes to simultaneously transmit independent messages $u_2,v_2\in\{0,\dots,Q\}$, where $u_2$ is intended for  $R_1$  and $v_2$  for $R_2$. After (linear) encoding, $S_1$ transmits $x_1:=x_1(u_1, v_1) $ and $S_2$ transmits $x_2:=x_2(u_2, v_2)$; that is to say
\begin{equation}\label{eqn22}
x_1= \alpha_1u_1 + \beta_1v_1     \quad {\rm and } \quad x_2= \alpha_2u_2+ \beta_2v_2
\end{equation}
 where $\alpha_1,\alpha_2,\beta_1$ and $\beta_2$ are some positive real numbers.  Next, for $i,j=1,2$,  let $h_{ij}$ denote the channel coefficients associated with the transmission of  signals from $S_j$ to $R_i$. Also, let $y_i$ denote the signal received by  $R_i $  before noise is taken into account.  Thus,
\begin{eqnarray}\label{eqn23}
y_1  & = & h_{11} x_1   +  h_{12} x_2\,,  \\ [1ex]
y_2  & = & h_{21} x_1   +  h_{22} x_2\,.  \label{eqn24} \,
\end{eqnarray}
A pictorial description of the above setup is given in Figure~\ref{fig2vb} below.

\vspace*{5ex}

\setlength{\unitlength}{20mm}

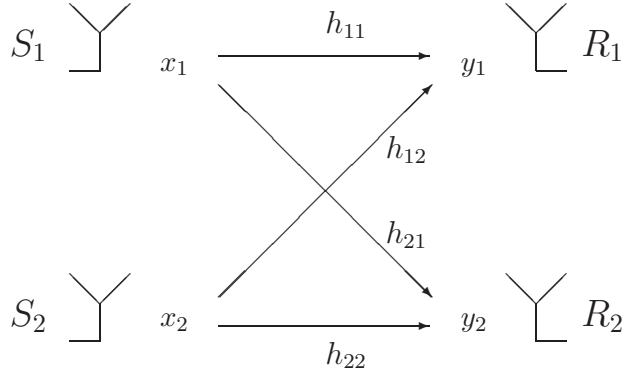
\begin{figure}[!ht]
\begin{center}
\begin{picture}(2, 2)(-0.75,-0.75)

\put(-1.5,0.8){\line(0,1){0.25}}
\put(-1.5,1.05){\line(-1,1){0.2}}
\put(-1.5,1.05){\line(1,1){0.2}}
\put(-1.5,0.8){\line(-1,0){0.2}}
\put(-1.1,0.8){$x_1$}

\put(1.4,0.8){\line(0,1){0.25}}
\put(1.4,1.05){\line(-1,1){0.2}}
\put(1.4,1.05){\line(1,1){0.2}}
\put(1.4,0.8){\line(1,0){0.2}}
\put(0.9,0.8){$y_1$}

\put(-1.5,-1){\line(0,1){0.25}}
\put(-1.5,-0.75){\line(-1,1){0.2}}
\put(-1.5,-0.75){\line(1,1){0.2}}
\put(-1.5,-1){\line(-1,0){0.2}}
\put(-1.1,-0.9){$x_2$}

\put(1.4,-1){\line(0,1){0.25}}
\put(1.4,-0.75){\line(-1,1){0.2}}
\put(1.4,-0.75){\line(1,1){0.2}}
\put(1.4,-1){\line(1,0){0.2}}
\put(0.9,-0.9){$y_2$}

\put(0,0){\vector(1,1){0.71}}
\put(0,0){\line(-1,-1){0.71}}
\put(0,0){\line(-1,1){0.71}}
\put(0,0){\vector(1,-1){0.71}}

\put(-0.71,0.9){\vector(1,-0){1.4}}
\put(-0.71,-0.9){\vector(1,-0){1.4}}

\put(-2.1,0.9){\Large{$S_1$}}
\put(-2.1,-0.9){\Large{$S_2$}}
\put(1.7,0.9){\Large{$R_1$}}
\put(1.7,-0.9){\Large{$R_2$}}

\put(0,1.05){\large{$h_{11}$}}
\put(0,-1.15){\large{$h_{22}$}}
\put(0.4,0.22){\large{$h_{12}$}}
\put(0.4,-0.35){\large{$h_{21}$}}

\end{picture}
\vspace*{6ex}
\end{center}
\caption{Two-user $X$-channel.}
\label{fig2vb}
\end{figure}

\vspace*{5ex}

\noindent Substituting \eqref{eqn22} into \eqref{eqn23} and \eqref{eqn24} gives that
\begin{eqnarray}\label{eqn25}
y_1  & = & h_{11}\alpha_1 u_1 +h_{11}\beta_1 v_1   +  h_{12} \alpha_2u_2+ h_{12}\beta_2 v_2\,,  \\ [1ex]
y_2  & = & h_{21} \alpha_1u_1 +h_{21}\beta_1 v_1   +  h_{22} \alpha_2u_2+h_{22}\beta_2v_2\,.  \label{eqn26} \,
\end{eqnarray}
Note that there are potentially $(Q+1)^4$ distinct outcomes of $y_i$ and they lie in the interval $[0, (h_{i1}\alpha_1+h_{i1}\beta_1+ h_{i2}\alpha_2+h_{i2}\beta_2)Q]$.

Now let us bring noise into the setup.
If $z_i$ denotes the (additive) noise at receiver $R_i$ ($i=1,2$), then instead of \eqref{eqn25} and \eqref{eqn26}, $R_1$ and $R_2$  receive the signals
\begin{equation}\label{eqn27}
y'_1   =  y_1+z_1\qquad\text{and}\qquad y'_2   =  y_2+z_2
\end{equation}
respectively. Equations \eqref{eqn22}--\eqref{eqn27} represent one of the simplest models of what is known as a {\em two-user $X$-channel}. The ultimate goal is for the receiver $R_1$ to decode the messages $u_1$ and $u_2$ from the observation of $y'_1$ and for the receiver $R_2$ to decode the messages $v_1$ and $v_2$ from the observation of $y'_2$. Clearly, this goal is attainable if $2|z_1|$ and $2|z_2|$ are smaller than the minimal distance between the outcomes of $y_1$ given by \eqref{eqn25} and the minimal distance between the outcomes of $y_2$ given by \eqref{eqn26} respectively.

Assume for the moment that $u_1, u_2, v_1, v_2 \in \{ 0,1\} $ and for the ease of discussion, let us just concentrate on the signal $y'_1$ received  at  $R_1$. Then there are generally up to 16 different outcomes for $y_1$.
Now there is one aspect of the above setup that we have not yet exploited: the receiver $R_1$  is not interested in the signals $ v_1 $  and  $v_2 $. So if these `unwanted' signals could be deliberately aligned (at the transmitters) via encoding into a single component  $v_1 + v_2 $,  then there would be fewer possible outcomes for $y_1$. This is merely down to the simple fact that there are 4 different pairs $(v_1,v_2)$ as opposed to 3 different sums $v_1 + v_2$ when $v_1$ and $v_2 $ take on  binary values. With this in mind, suppose that
\begin{equation}\label{eqn28}
 x_1= \lambda(h_{22} u_1 + h_{12}v_1)     \quad {\rm and } \quad x_2= \lambda(h_{21} u_2+ h_{11} v_2)  \,
\end{equation}
respectively. Here $\lambda \ge 1 $ is simply some scaling factor.  Thus, with reference to \eqref{eqn22},  we have that
\begin{equation}\label{eqn29}
\alpha_1=\lambda h_{22}, \ \beta_1=\lambda h_{12},   \ \alpha_2=\lambda h_{21}, \ \beta_2=\lambda h_{11}\, ,
\end{equation}
 and so \eqref{eqn23} and  \eqref{eqn24} become
\begin{eqnarray} \label{eqn30}
y_1  & = & \lambda\Big((h_{11} h_{22}) u_1   +  (h_{21} h_{12}) u_2  + (h_{11} h_{12})  (v_1 + v_2)\Big)  \\ [1ex]
\label{eqn31}
y_2  & = & \lambda\Big((h_{21} h_{12}) v_1   +  (h_{11} h_{22}) v_2  + (h_{21} h_{22})  (u_1 + u_2)\Big)  \, .
\end{eqnarray}

\noindent Clearly, there are now only 12 outcomes for either $y_1$ or $y_2$ rather than 16. The above discussion is a simplified version of that appearing in \cite[\S III: Example 3]{MR3245356} and  constitutes the basis for \emph{real interference alignment} - a concept introduced and developed in \cite{MR2451007, MGK, MR3245356} and subsequent publications.

\begin{remark} \label{history}
The original idea of interference alignment exploits the availability of `physical' dimensions  of wireless systems such as the frequency of the  signal or the presence of  multiple antennae. In short,  an antenna is a device (such as an old fashioned radio or television ariel) that is used to transmit or receive signals. In any case, by using several antennae it is possible for a user to simultaneously transmit several messages  and these can naturally be  thought of  as the coordinates of a point in a vector space, say $\R^n$. Thus, when analysing such wireless systems the transmitted signals can be treated as vectors in $\R^n$. The art of  interference alignment is to attempt to introduce an encoding at the transmitters (users) which result in unwanted (interfering) signals at the receivers being forced to lie in a subspace of $\R^n$ of smaller (ideally single) dimension. Such alignment is achieved by exploiting  elementary methods from linear algebra, see for instance \cite[Section 2.1]{JafarBook} for concrete examples and a detailed overview of the process. The novel idea of Motahari et al involves exploiting instead the abundance of rationally independent points in the real line $\R$. For instance, with reference to Example~2 above and the transmitted signals given by \eqref{eqn28}, assuming that $h_{22}/h_{12}$ is irrational, the signal $x_1$ transmitted by $S_1$ lies in the $2$-dimensional vector subspace of $\R$ over $\Q$ given by
$$
V_1=\lambda h_{22} \Q + \lambda h_{12}\Q  \, .
$$
Similarly, assuming that $h_{21}/h_{11}$ is irrational, the signal $x_2$ transmitted by $S_2$ lies in the $2$-dimensional vector subspace of $\R$ over $\Q$ given by
$$
V_2=\lambda h_{21} \Q + \lambda h_{11}\Q\,.
$$
In view of the alignment, the unwanted messages $v_1$ and $v_2$ at receiver $R_1$ are forced to lie in a subspace of $\R$ over $\Q$ of dimension one; namely $W_1=\lambda h_{11} h_{12}\Q$.   Similarly,
the unwanted messages $u_1$ and $u_2$  at receiver $R_2$ lie in the one-dimensional $\Q$-subspace $W_2=\lambda h_{21} h_{22}\Q$.
\end{remark}

As with the `baby' example, we  can easily modify  the above `binary' consideration to the more general situation when the messages $u_1, u_2, v_1, v_2$ are integers lying in $\{0,\dots,Q\}$; i.e.,  the setup of Example~2. It is easily seen that in this more general situation the savings coming from interference alignment are even more stark: there are $(2Q+1)(Q+1)^2\sim 2Q^3$ outcomes for either $y_1$ or $y_2$ after alignment as opposed to $(Q+1)^4\sim Q^4$ outcomes before alignment. Consequently, based on the outcomes for $y_1$ and $y_2$ after alignment being equally spaced, we have the following trivial estimates for the associated  minimal distances:
\begin{equation}\label{eqn32}
\dminVBone\le\frac{\lambda\Big(h_{11} h_{22}   +  h_{21} h_{12}  + 2h_{11} h_{12}\Big)Q}{(2Q+1)(Q+1)^2}
\end{equation}
and
\begin{equation}\label{eqn33}
\dminVBtwo\le\frac{\lambda\Big(h_{21} h_{12}   +  h_{11} h_{22}  + 2h_{21} h_{22}\Big)Q}{(2Q+1)(Q+1)^2}\,\,.
\end{equation}
\noindent We stress that $\dminVBone$ is the minimal distance between the outcomes of $y_1$ given by \eqref{eqn30} and $\dminVBtwo$ is the minimal distance between the outcomes of $y_2$ given by \eqref{eqn31}.
As in Example 1, `perfect' separation is essentially never the case and to demonstrate this we need to bring into play the appropriate  higher dimensional version of Dirichlet's theorem.

\begin{theorem}[Minkowski's theorem for systems of linear forms]\label{MTLF}
  Let $\beta_{i,j}\in\R$, where $1\le i,j\le k$, and let
  $\lambda_1,\dots,\lambda_k>0$. If
  \begin{equation}\label{eqn34}
    |\det(\beta_{i,j})_{1\le i,j\le k}|\le \prod_{i=1}^k\lambda_i,
  \end{equation}
  then there exists a non-zero integer point $\vvv a=(a_1,\dots,a_k)$
  such that
  \begin{equation}\label{eqn35}
    \left\{\begin{array}{lll}
             |a_1\beta_{i,1}+\dots+a_k\beta_{i,k}|<\lambda_i\,, && (1\le i\le k-1)\\[2ex]
             |a_1\beta_{k,1}+\dots+a_k\beta_{k,k}|\le \lambda_k\,.
           \end{array}\right.
       \end{equation}
\end{theorem}

\vspace*{3ex}

\noindent The simplest proof of the theorem makes use of Minkowski's fundamental convex body theorem from the geometry of numbers; see, for instance \cite[\S1.4.1]{MR3618787} or, indeed, Chapter~\chaptwo{} of this book.

We now show how the minimal distance $\dminVBone$ (and similarly, $\dminVBtwo$) can be estimated from above using Minkoswki's theorem. For simplicity, consider the case when
\begin{equation}\label{eqn36}
 \max\{h_{11} h_{22},\,h_{21} h_{12}\,,h_{11} h_{12}\}=   h_{11} h_{12}\,;
\end{equation}
that is,  $h_{11} \ge h_{21}$ and $h_{12}\ge h_{22}$.
Then, on applying  Theorem~\ref{MTLF} with $k=3$,
$
\lambda_1=(h_{11} h_{12})Q^{-2}, $  $\lambda_2=\lambda_3=Q $
and
$$
(\beta_{i,j})_{1\le i,j\le k}=\left(\begin{array}{ccc}
                                      \,h_{11} h_{22}\, & \,h_{21} h_{12}\, & \,h_{11} h_{12}\, \\
                                      1 & 0 & 0 \\
                                      0 & 1 & 0
                                    \end{array}
\right) ,
$$ we deduce the  existence of integers $a_1$, $a_2$ and $a_3$, not all zero, such that
\begin{equation}\label{eqn37}
\left\{
\begin{array}{l}
|(h_{11} h_{22}) a_1   +  (h_{21} h_{12}) a_2  + (h_{11} h_{12})a_3|< (h_{11} h_{12})Q^{-2}\,,\\[2ex]
|a_1|< Q,\\[2ex]
|a_2|\le Q\,.
\end{array}
\right.
\end{equation}

\begin{remark}
It is worth pointing out that the argument just given above  can be appropriately  adapted  to establish the following generalisation of Dirichlet's theorem.  For the details see for instance  \cite[Corollary 1.4.7]{MR3618787}.  Here and throughout, given a point  $\vvv x=(x_1,\dots,x_n)  \in \R^n$ we let $|\vvv x|:= \max\{|x_1|,\dots,|x_n|\}  \, . $

\vspace*{2ex}

\begin{theorem}\label{thm3}
For any $\bm\xi=(\xi_1,\dots,\xi_n)\in \R^n$ and any $Q\in\N$ there exists  $(p,\vvv q) \in \Z \times \Z^{n}$ such that

\begin{equation}\label{eqn38}
\left|q_1\xi_1+\dots+q_n\xi_n+p\right|<\frac{1}{Q^n}    \qquad \textrm{ and }  \qquad  1 \leq  |\vvv q| \leq Q \, .
\end{equation}
\end{theorem}

\end{remark}

\bigskip

We now return to determining an upper bound for $\dminVBone$. A consequence of \eqref{eqn37} is that for any given $Q\ge1$ there exist integers $a_1,a_2,a_3$, not all zero, such that
$$
\left|\frac{h_{11} h_{22}}{h_{11} h_{12}} a_1   +  \frac{h_{21} h_{12}}{h_{11} h_{12}} a_2  + a_3\right|< Q^{-2}\le1\,.
$$
This together with the triangle inequality implies that
$$
|a_3|<
\left|\frac{h_{11} h_{22}}{h_{11} h_{12}} a_1\right|   +  \left|\frac{h_{21} h_{12}}{h_{11} h_{12}} a_2\right|  + 1,
$$
and so in view of our `maximal' assumption \eqref{eqn36}, it follows that
$$
|a_3|<|a_1|+|a_2|+1\le Q+(Q-1)+1=2Q\,.
$$
Now observe that the quantity
$$
\lambda\times\big|(h_{11} h_{22}) a_1   +  (h_{21} h_{12}) a_2  + (h_{11} h_{12})a_3\big|
$$
is precisely the distance between the two specific outcomes of $y_1$ associated with \eqref{eqn30} given by the following choices:
$$
\begin{array}{llll}
\text{Choice 1: }&u_1=\max\{0,a_1\}, &u_2=\max\{0,a_2\}, &v_1+v_2=\max\{0,a_3\},\\[1ex]
\text{Choice 2: }&u_1=\max\{0,-a_1\}, &u_2=\max\{0,-a_2\}, &v_1+v_2=\max\{0,-a_3\}.
\end{array}
$$
We have just observed that Theorem~\ref{MTLF} guarantees that  $|a_1|\le Q$, $|a_2|\le Q$ and $|a_3|\le 2Q$ and so $u_1, u_2, v_1, v_2$ are integers lying in $\{0,\dots,Q\}$.  Hence, in view of  \eqref{eqn37} it follows (under the assumption  \eqref{eqn36}) that
\begin{equation}\label{eqn39}
\dminVBone\le \frac{\lambda  \,  h_{11} h_{12}}{Q^{2}}  \ = \ \frac{C_2}{Q^{2}}  \, ,\qquad\text{where } \ C_2:=\lambda\,h_{11}h_{12}\,.
\end{equation}
For all intents and purposes, this bound on the minimal distance is smaller than the `perfect' separation estimate given by
\eqref{eqn32}. A similar analysis can be carried out  when the maximum in \eqref{eqn36} is attained on another term, and for estimating $\dminVBtwo$. Obviously the parameter $C_2$ would reflect the situation under consideration.

As mentioned  earlier, the receivers $R_1$ and $R_2$ can decode the respective messages provided that the respective minimal distances $\dminVBone$ and $\dminVBtwo$ are at least  two times  larger than the noise at each receiver.  Given that the nature of noise is often a random variable with normal distribution, the overarching goal  is to ensure the probability that $|z_1|<\tfrac12\dminVBone$ and $|z_2|<\tfrac12\dminVBtwo$ is large. Indeed, as in Example 1,  the larger the probability the more likely the receivers $R_i$ $(i=1,2)$ are able to recover messages by rounding $ y'_i $ (given by \eqref{eqn27}) to the closest  possible outcome of $y_i$ (given by \eqref{eqn30} if $i=1$  and \eqref{eqn31} if $i=2$). It is therefore imperative to understand how close $\dminVBone$ and $\dminVBtwo$ can be to their theoretical upper bounds. With this in mind we now describe  various tools and notions from Diophantine approximation that can be used for this purpose.  In short, they allow us to get within a constant factor of the theoretical upper bounds.  As in Example 1, we start  by attempting to manipulate the encoding process so as to exploit the existence of badly approximable points in $\R^n$.
Before we embark on this discussion we make a remark concerning the scaling factor $\lambda$ that first appears in \eqref{eqn28}.

\begin{remark} Observe that estimating $\dminVBone$ and $\dminVBtwo$ from below is essentially the same as estimating from below the size of the linear forms
\begin{eqnarray}
&&(h_{11} h_{22}) u_1   +  (h_{21} h_{12}) u_2  + (h_{11} h_{12})  (v_1 + v_2)\,,\label{eqn40}  \\ [1ex]
&&(h_{21} h_{12}) v_1   +  (h_{11} h_{22}) v_2  + (h_{21} h_{22})  (u_1 + u_2)\,.\label{eqn41}
\end{eqnarray}
The factor $\lambda$ appearing in \eqref{eqn30} and \eqref{eqn31} only determines the scaling of $\dminVBone$ and $\dminVBtwo$ and can be used to `adjust' these quantities, namely, to prevent them from vanishing as $Q$ grows, see Remark~\ref{vanishing} for a similar consideration within Example~1. Indeed, the effect of multiplication by $\lambda$ can be simply understood as increasing the separation in the constellation of messages; i.e. the messages $u_1,v_1, u_2, v_2$ could be associated with  $\{0,\lambda,2\lambda,3\lambda,\dots,Q\lambda\}$ instead of $\{0,1,2,3,\dots,Q\}$.
\end{remark}

\subsection{Badly approximable points}  \label{BAD20}

We start by stating the following simple consequence of Theorem~\ref{thm3}.  It is the higher dimensional  analogue  of Corollary~\ref{cor1}.

\begin{corollary}\label{cor2}
For any point $\bm\xi\in\R^n$ there exists infinitely many $(p,\vvv q) \in \Z \times \Z^{n}\backslash \{\vvv 0\} $ such that
  \begin{equation}\label{eqn42}
    |q_1\xi_1+\dots+q_n\xi_n+p| < \frac{1}{|\vvv q|^n}\,  .
  \end{equation}
\end{corollary}

Note that in the corollary we have not  imposed the condition that $\bm\xi$ is not a point on a rational hyperplane.  This is since we do not impose, as in the one-dimensional statement,  the requirement that $(p,\vvv q)$ is {\em primitive}; that is, without a non-trivial common divisor. Naturally, badly approximable points in $\R^n$ are defined by requiring that the right hand side of \eqref{eqn42} cannot be `improved' by an arbitrary constant factor. This we now formally state.

\begin{definition}[Badly approximable points]  \label{BADdef}
A point $\bm\xi \in\R^n$ is said to be {\em badly approximable} if there exists a constant $\kappa=\kappa(\bm\xi)>0$ such that for all $(p,\vvv q) \in \Z \times \Z^{n}\backslash \{\vvv 0\} $
  \begin{equation}\label{eqn43}
    |q_1\xi_1+\dots+q_n\xi_n+p| \ge \frac{\kappa}{|\vvv q|^n}\, .
  \end{equation}
\end{definition}

\noindent The set of badly approximable points in $\R^n$ will be denoted by $\Bad(n)$.
It is relatively simple to verify that for any real algebraic number $\xi$ of degree $n+1$ the point $(\xi,\xi^2,\dots,\xi^n) \in \R^n$ is badly approximable. Indeed, consider the norm of the algebraic number
$$
\alpha_1=q_1\xi+q_2\xi^2+\dots+q_n\xi^n+p\in\Q(\xi)
$$
which (up to sign) is the product of $\alpha_1$ and its other conjugates, say $\alpha_2,\dots,\alpha_{n+1}$. For simplicity one can assume that $\xi$ is an algebraic integer. Furthermore, we can assume that the right hand side of \eqref{eqn43} is less than one and so without loss of generality we have that $|p| \ll |\vvv q| $.   Then, it is easily seen that $|\alpha_j|\ll|\vvv q|$ for all $j$, while the norm of $\alpha_1$ is bounded below by $1$. Here and elsewhere $\gg$ (respectively,  $\ll$) is the Vinogradov symbol meaning $\ge$ (respectively $\le$) up to a multiplicative constant factor. The upshot is that
$$
|q_1\xi+q_2\xi^2+\dots+q_n\xi^n+p|=|\alpha_1|\gg\prod_{j=2}^{n+1}|\alpha_j|^{-1}\gg |\vvv q|^{-n} \, ,
$$
whence the claim that $(\xi,\xi^2,\dots,\xi^n)\in\Bad(n)$ follows. This argument can be made explicit to obtain a specific lower bound for the badly approximable constant $\kappa(\xi,\dots,\xi^n)$. Examples of badly approximable algebraic points of this ilk were first given by Perron \cite{MR1512000}.

The reason for us bringing into play the notion of badly approximable numbers is similar to that in Example 1. If the channel coefficients happen to be such that
\begin{equation}\label{eqn44}
\bm\xi=(\xi_1, \xi_2):= \left(\frac{h_{11} h_{22}}{h_{11} h_{12}},\frac{h_{21} h_{12}}{h_{11} h_{12}}\right)
=\left(\frac{h_{22}}{h_{12}},\frac{h_{21}}{h_{11}}\right)
\end{equation}
is a badly approximable point in $\R^2$, then we are guaranteed the existence of a constant $\kappa(\bm\xi)>0$ such that
$$
\left|\frac{h_{11} h_{22}}{h_{11} h_{12}} q_1   +  \frac{h_{21} h_{12}}{h_{11} h_{12}} q_2  + p\right|\ge \frac{\kappa(\bm\xi)}{|\vvv q|^2}
$$
for all non-zero integer points $(p,\vvv q) \in \Z \times \Z^{2}\backslash \{\vvv 0\} $. Thus, it follows that for every $Q \in \N$:
$$
\left|h_{11} h_{22} q_1   +  h_{21} h_{12}q_2  + h_{11} h_{12}p\right|\ge \frac{\kappa(\bm\xi) h_{11} h_{12}}{Q^{2}}\,
$$
for all $(q_1,q_2,p)\in\Z^3$ with $1 \le |\vvv q | \le Q $,  and so the separations between any two points given by \eqref{eqn30} is at least $\frac{\kappa(\bm\xi) \lambda h_{11} h_{12}}{Q^{2}}$.
In other worlds,
\begin{equation}\label{eqn45}
\dminVBone\ge     \frac{\kappa(\bm\xi) C_2}{Q^2}  \,
\end{equation}
which complements the upper bound \eqref{eqn39}.
Note that instead of \eqref{eqn44} one can equivalently consider $\bm\xi=(\xi_1, \xi_2) $   to be  either of the points
\begin{equation}\label{eqn46}
\left(\frac{h_{21} h_{12}}{h_{11} h_{22}},\frac{h_{11} h_{12}}{h_{11} h_{22}}\right),
\quad\left(\frac{h_{11} h_{22}}{h_{21} h_{12}},\frac{h_{11} h_{12}}{h_{21} h_{12}}\right)\,,
\end{equation}
which will also be badly approximable if \eqref{eqn44} is badly approximable.
Thus, we can in fact show that \eqref{eqn45} with appropriately adjusted constant $\kappa(\bm\xi)$ holds with $C_2$ redefined as
\begin{equation}\label{eqn47}
C_2:=\max\{h_{11} h_{22},\,h_{21} h_{12}\,,h_{11} h_{12}\}\,.
\end{equation}

\noindent A similar lower bound to  \eqref{eqn45} can be established for $\dminVBtwo$ if
\begin{equation}\label{eqn48}
\left(\frac{h_{21} h_{12}}{h_{21} h_{22}}    , \frac{h_{11} h_{22}}{h_{21} h_{22}}\right)
=\left(\frac{h_{12}}{h_{22}}    , \frac{h_{11} }{h_{21} }\right)
\end{equation}
or equivalently
\begin{equation}\label{eqn49}
\left(\frac{h_{11} h_{22}}{h_{21} h_{12}}, \frac{h_{21} h_{22}}{h_{21} h_{12}}\right)~\text{or}~
\left(\frac{h_{21} h_{12}}{h_{11} h_{22}}, \frac{h_{21} h_{22}}{h_{11} h_{22}}\right)
\end{equation}
is a badly approximable point in $\R^2$.

\begin{remark}
We end this subsection with a short discussion that  brings to the forefront  the significant difference between  Examples 1 $\&$  2,  in attempting  to exploit the existence of badly approximable points.  In short, the encoding process
\eqref{eqn29}  leading to the alignment of the unwanted signals  in \eqref{eqn30} and \eqref{eqn31} comes at a cost.  Up to a scaling factor,  it fixes the parameters $ \alpha_1, \alpha_2, \beta_1, \beta_2$ in terms of  the given channel coefficients.  This in turn, means that our analysis of the linear forms  \eqref{eqn40} and \eqref{eqn41} gives rise to the points
\eqref{eqn44} and  \eqref{eqn48} in $\R^2$  that are dependent purely on the channel coefficients.  Now either these points are in $\Bad(2)$  or not. In other words,  there is no flexibility left in the encoding procedure (after  alignment) to force \eqref{eqn44} or \eqref{eqn48} to  be badly approximable in $\R^n$.
This is very different to the situation in Example~1.  There we had total freedom to choose the  parameters $\alpha$ and $\beta$  in order to force the point \eqref{eqn10} to be a badly approximable number.
The upshot is that in Example 2, there is no such flexibility and this exacerbates the fact that the probability of \eqref{eqn44} or \eqref{eqn48} being badly approximable is already zero.  The fact that $\Bad(n)$ has measure zero can be easily deduced from Khintchine's theorem, which will be discussed below in \S\ref{sv2.4}  - however see \S\ref{FR} for the actual derivation. Although of measure zero, for the sake of completeness, it is worth mentioning that $\Bad(n)$ is of full Hausdorff dimension, the same as the whole of $\R^n$. This was established by Schmidt \cite{MR195595, MR248090}  as an application of his remarkably powerful theory of $(\alpha,\beta)$-games.  In fact, he proved the full dimension statement for  badly approximable sets associated with systems of linear forms (see \S\ref{FR}).
\end{remark}

\begin{remark}\label{badbad}
 We note that if $\bm\xi$ is any of the points \eqref{eqn44} or \eqref{eqn46} and $\bm\xi'$ is any of the points \eqref{eqn48} or \eqref{eqn49}, then in order to simultaneously guarantee \eqref{eqn45} and its analogue for $\dminVBtwo$ both $\bm\xi$ and $\bm\xi'$ need to be badly approximable. This adds more constraints to an already unlikely (in probabilistic terms) event, since
the points $\bm\xi$ and $\bm\xi'$ are dependent. Indeed, concerning the latter,  it is easily seen that
\begin{equation}\label{eqn50}
  \bm\xi' = \bm f(\bm\xi)
\end{equation}
for one of the following choices of $\bm f:\R^2\to\R^2$
\begin{equation}\label{eqn51}
\bm f(x,y)=\left(\frac{1}{x},\frac{1}{y}\right),~
\left(x, \frac{x}{y}\right),~
\left(\frac{x}{y}, x\right),~
\left(y, \frac{y}{x}\right),~\text{or}~
\left(\frac{y}{x}, y\right).
\end{equation}
Clearly,  the set of pairs $(\bm\xi,\bm\xi')$ of badly approximable points confined by \eqref{eqn50} is a subset of the already  measure zero set  $\Bad (2)\times \Bad (2)$.  Nevertheless, they do exist, as was proved by Davenport \cite{MR166154}, and are in ample supply in the following sense: the set of pairs $(\bm\xi,\bm\xi')$ of badly approximable points subject to \eqref{eqn50} has full Hausdorff dimension, which is two. In other words, the dimension of $\Bad(2)\cap\bm f(\Bad(2))$ is equal to the dimension of $\Bad(2)$. This follows from the results of \cite{MR2981929}.
\end{remark}

\subsection{Probabilistic aspects} \label{probasn} In this section,
we consider within the higher dimensional context of Example 2,
the probabilistic approach set
out in \S\ref{Prob1}. Given $0<\kappa'<1$ and $Q\in\N$,
let $\cB_n(Q,\kappa')$ be the set of $\bm\xi\in \II^n:=(0,1)^n$ such that
\begin{equation}\label{eqn52}
\left|q_1\xi_1+\dots+q_n\xi_n+p\right|\ge\frac{\kappa'}{Q^n}
\vspace*{2ex}
\end{equation}
for all integer points $(p,\vvv q) \in \Z \times \Z^{n}$ such that $1\le |\vvv q|\le Q$. Note that $\bm\xi\in\cB_n(Q,\kappa')$ are precisely the points  in  $\II^n$ for which the right hand side of inequality  \eqref{eqn38} appearing in Dirichlet's $n$-dimensional theorem,  cannot be improved by the factor of $\kappa'$ ($Q$ is fixed here).
To estimate the probability of $\cB_n(Q,\kappa')$, we consider the complementary inequality
\begin{equation}\label{eqn53}
\left|q_1\xi_1+\dots+q_n\xi_n+p\right|<\frac{\kappa'}{Q^n}\,.
\end{equation}
Let $1\le |\vvv q|\le Q$. Then for a fixed $\vvv q$,   it can be verified  that the probability that a given $\bm\xi\in \II^n$ satisfies \eqref{eqn53}
for some $p\in\Z$ is exactly $2\kappa'Q^{-n}$  -- this is a relatively straightforward  calculation  the details of which can be found in \cite[Lemma~8]{MR548467}. On summing up these probabilities over $\vvv q$ with $q_1\ge0$ (this can be assumed without loss of generality),  we conclude that the probability that a given $\bm\xi\in \II^n$ satisfies \eqref{eqn53} for some integers $p$ and $1\le |\vvv q|\le Q$,  is bounded above by
$$
2 \kappa'Q^{-n}(2Q+1)^{n-1}(Q+1)\sim 2^n\kappa'\qquad\text{(as $Q\to\infty$)}.
$$
This in turn implies the following statement.

\begin{lemma} \label{mum2}
For any $0<\kappa'<1$ and any $Q\in\N$
\begin{equation}\label{eqn54}
{\bf Prob}(\cB_n(Q,\kappa'))\ge 1-2^n \kappa'\left(1+\frac{1}{2Q}\right)^{n-1}\left(1+\frac1Q\right)\,.
\end{equation}
\end{lemma}

\bigskip

\noindent Similarly to the one-dimensional case (cf.~\S\ref{Prob1}),   the above trivial estimate can be improved, however, we leave this task to the energetic  reader. We also note that the probability in Lemma~\ref{mum2} is assumed to be uniform but it is possible to obtain a version of Lemma~\ref{mum2} for other (absolutely continuous) distributions as mentioned in Remark~\ref{vb3978}. In any case, the upshot of the above discussion is that for sufficiently small $\kappa'>0$  the probability that the point $\bm\xi$ given by \eqref{eqn44}  modulo $1$ belongs to $\cB_n(Q,\kappa')$ is positive. Hence, it follows that for any $\rho \in(0,1)$ there exists an explicitly computable constant $\kappa'>0$ with the following  property:  with probability greater than $ \rho$,  for a random choice of the four channel coefficients $h_{ij}$  $(i,j=1,2)$,  the separation between the associated  points $y_1$ given by \eqref{eqn30}  is at least $\kappa' C_2/ Q^{2} $,  and so  the minimal distance $\dminVBone$  satisfies
\begin{equation}\label{eqn55}
\dminVBone  \ge \frac{\kappa' C_2  \ }{Q^{2}}.
\end{equation}
Moreover, the probability $ \rho$ can be made arbitrarily close to one.  However, the  cost is that the constant $\kappa'$  becomes  arbitrarily small.  The above analysis holds equally well at receiver $R_2$ and we obtain an analogous probabilistic bound for the minimal distance $\dminVBtwo $  associated  with the points $y_2$ given by \eqref{eqn31}.

\begin{remark}  \label{mummie}
Obviously  \eqref{eqn55} is a better lower bound for $\dminVBone   $  than \eqref{eqn45}  whenever  $\kappa'$ is greater than the badly approximable constant $\kappa(\bm\xi)$ appearing in \eqref{eqn45}. However, this really is not the point -- both approaches yield lower bounds for the minimal distance that lie within a constant factor of the theoretic upper bound \eqref{eqn39}.  The  main point is that the  badly approximable approach has zero probability of actually delivering \eqref{eqn45} whereas the  probabilistic approach yields \eqref{eqn45} with positive probability (whenever  $\kappa(\bm\xi)$ is sufficiently small so that the right hand side of \eqref{eqn54} with $\kappa' = \kappa(\bm\xi)$  is positive).
\end{remark}

\begin{remark}  \label{badbadprob}
In the same vein as Remark~\ref{badbad}, we first observe  that  in order to simultaneously  guarantee \eqref{eqn55} and its analogue for $\dminVBtwo$, both the points $\bm\xi$ and $\bm\xi' $ modulo one, where $\bm\xi$ is given by \eqref{eqn44} or \eqref{eqn46} and $\bm\xi'$ is given by \eqref{eqn48} or \eqref{eqn49}, need to simultaneously lie in $\cB_n(Q,\kappa')$. Thus to
 obtain the desired (simultaneous) probabilistic statement,  we need to show the probability  of both $\bm\xi$ and $\bm\xi' $ modulo one belonging to $\cB_n(Q,\kappa')$ is positive; say $1- \kappa'$ in line with \eqref{eqn54}.  This would be an easy task if the points  under consideration were independent.  However,  the points $\bm\xi$ and $\bm\xi'$ are confined by \eqref{eqn50} and therefore the events $\bm\xi({\rm mod 1})\in\cB_n(Q,\kappa')$ and $\bm\xi'({\rm mod 1})\in \cB_n(Q,\kappa')$ are dependent. Nevertheless, it can be shown that the probability of these two events holding simultaneously is at least $1- \sigma \times\kappa'$, where $\sigma$ is an explicitly computable positive constant.  We leave the details to the extremely energetic reader.
\end{remark}

\begin{remark}\label{mate}
For another specific  (and powerful) application of the probabilistic approach outlined in this section we refer the reader to \cite{NM13}.  In short,  in \cite{NM13}  the probabilistic approach is used to estimate the capacity of the two-user X channel from below and above with only a constant gap between the bounds.
\end{remark}

 Notice that the fundamental set $\cB_n(Q,\kappa') $ that underpins the  probabilistic approach  is dependent on $Q$.  Thus, as $Q$  varies, so does the random choice of channel coefficients that achieve \eqref{eqn55}.   As we shall see in the  next section, this can be problematic.

\subsection{The Khintchine-Groshev theorem and Degrees of Freedom}  \label{sv2.4}

The probabilistic approach of \S\ref{probasn}, relies on the  point $\bm\xi$ associated with the channel coefficients via  \eqref{eqn44}  being in the set $\cB_n(Q,\kappa') $.   Now, however large the probability of the latter (a lower bound is given by \eqref{eqn54}), it can be verified that
\begin{equation}\label{eqn56}
{\bf Prob}(\cB_n(Q,\kappa'))\le 1- \omega\kappa' \ ,
\end{equation}
where $\omega > 0$ is a constant depending only  on $n$. The proof of this can be obtained by utilizing the notion of ubiquity; in particular,  exploiting the ideas used in establishing  Proposition~4 in \cite[Section~12.1]{MR2184760}. Moreover, for any $\kappa'>0$ and any infinite subset $\cQ\subset\N$ the probability that $\bm\xi$ lies in $\cB_n(Q,\kappa')$ for all sufficiently large $Q \in \cQ$ (let alone all sufficiently large $Q$ in $\N$) is zero. This is a fairly straightforward consequence of Theorem~\ref{thm3} and \cite[Lemma~4]{MR2457266}.
This is an unfortunate downside  of the probabilistic approach, especially when it comes to estimating the so called {\em Degrees of Freedom} (DoF)  of communication channels. Indeed,  when estimating the DoF it is desirable to  achieve, with probability one, close to optimal  bounds on the minimal distances ($\dminVBone$ and $\dminVBtwo$ within the context of Example 2)  for all sufficiently large $Q$.
Of course, the badly approximable approach described in  \S\ref{BAD20} does this in the sense that it yields \eqref{eqn55} for all large $Q$ whenever $\bm\xi\in\Bad(2)$. However, as already discussed in Remark \ref{mummie}, the downside of the badly approximable approach is that the probability of hitting $ \Bad(2) $  is zero.   In this section we describe another approach which overcomes the inadequacies of both the probabilistic and badly approximable approaches.  It gives an `$\ve$-weaker'  estimate for the minimal distance but  as we shall soon see it is more than adequate for estimating the DoF.  The key is to make use of the fundamental Khintchine-Groshev theorem in metric Diophantine approximation and this is what we first  describe.

Given a function $\psi:\Rp \to\Rp$, where $\Rp$ denotes  the set of non-negative real numbers, let
\begin{equation}\label{Wn}
\cW_{n}(\psi):=\left\{ \bm\xi \in \II^n :  \begin{array}{l}
|q_1\xi_1+\dots+q_n\xi_n+p|<\psi(|\vvv q|)\\[1ex] \text{for i.m. }(p,\vvv q) \in \Z \times \Z^{n}\backslash \{\vvv 0\}
                                                     \end{array}
\right\}.
\end{equation}
Here and elsewhere, `i.m.' is short for `infinitely many' and
given a subset $X$ in $\R^n$, we will write $|X|_{n}$ for its $n$--dimensional Lebesgue measure.  For obvious reasons, points in $\cW_{n}(\psi)$ are referred to as $\psi$-approximable. When $n=1$,  it is easily seen that $\cW(\psi):=\cW_{1}(\psi)$ is the set of $\xi =\xi_{1} \in \II$ such that
$$
\left|\xi-\frac pq\right|<\frac{\psi(q)}{q}
$$
has infinitely many solutions $(p,q)\in\Z\times\N$. Investigating  the measure theoretic properties of $\cW(\psi)$ was the subject of the pioneering work of Khintchine \cite{MR1512207} almost a century ago.  The following generalisation of Khintchine's theorem is a special case of a result of Groshev \cite{Groshev} concerning systems of linear form (see Theorem~\ref{KGT2} in \S\ref{FR}).  In the one-dimensional case,  it provides a quantitative analysis of the density of the rationals  in the reals.

\medskip

\begin{theorem}[Khintchine-Groshev for one linear form]\label{KGT}
Let $\psi:\Rp\to\Rp$ be a  monotonic function. Then
$$
|\cW_{n}(\psi)|_{n}=\left\{\begin{array}{cl}
0 &\text{if } \ \sum_{q=1}^\infty q^{n-1}\psi(q)<\infty\,,\\[2ex]
1 &\text{if } \ \sum_{q=1}^\infty q^{n-1}\psi(q)=\infty\,.
                       \end{array}
\right.
$$
\end{theorem}

\medskip

\begin{remark}
The convergence case of Theorem~\ref{KGT} is a relatively simple application of the Borel--Cantelli Lemma from probability theory  and it holds for arbitrary functions $\psi$. In  the divergence case,  the  theorem was first obtained by Groshev  under the stronger assumption that $q^n\psi(q)$ is  monotonic. In fact, the monotonicity assumption can be completely removed from the statement of theorem  if $ n\geq 2 $. This is a consequence of Schmidt's paper \cite[Theorem~2]{MR118711} from the swinging  sixties  if $  n \geq 3 $ and the relatively recent paper \cite{MR2576284} covers the $n=2$ case.  In
1941, Duffin $\&$ Schaeffer \cite{MR4859} constructed a
non-monotonic approximating function $\psi$ for which the sum
$\sum_q \psi(q)$ diverges but $|W(\psi)|=0$. Thus, the monotonicity assumption cannot be removed in dimension one.   For completeness, we mention that in the same paper Duffin $\&$ Schaeffer formulated an alternative statement for arbitrary functions. This soon became known as the notorious Duffin-Schaeffer Conjecture and it remained unsolved for almost eighty years  until the breakthrough work of  Koukoulopoulos $\&$ Maynard  \cite{Kouk-May}.
\end{remark}

An immediate consequence of the convergence case of Theorem~\ref{KGT} is the following statement.

\begin{corollary}\label{cor3}
Let  $\psi:\Rp\to\Rp$ be a function such that
\begin{equation}\label{eqn57}
  \sum_{q=1}^\infty q^{n-1}\psi(q)<\infty\,.
\end{equation}
Then, for almost all $\bm\xi \in \II^n$ there exists a constant $\kappa (\bm\xi)  >0 $  such that
\begin{equation}\label{eqn58}
 |q_1\xi_1+\dots+q_n\xi_n+p | \ >  \  \kappa(\bm\xi)  \,   \psi(|\vvv q|)  \qquad   \forall  \    (p,\vvv q) \in \Z \times \Z^{n}\backslash \{\vvv 0\} \,.
\end{equation}
\end{corollary}

\

Now consider the special case when $\psi: q \to q^{-n-\ve}$ for some $\ve>0$. Then Corollary~\ref{cor3} implies that for almost all  $\bm\xi\in\II^n$ there exists a constant $\kappa(\bm\xi)>0$ such that
$$
\left|q_1\xi_1+\dots+q_n\xi_n+p\right|\ge\frac{\kappa(\bm\xi)}{|\vvv q|^{n+\ve}}
$$
for all  $ (p,\vvv q) \in \Z \times \Z^{n}\backslash \{\vvv 0\}$. In particular, for almost all $\bm\xi\in\II^n$ and every $Q\in\N$ we have that
\begin{equation}\label{eqn59}
\left|q_1\xi_1+\dots+q_n\xi_n+p\right|\ge\frac{\kappa(\bm\xi)}{Q^{n+\ve}}
\end{equation}
for all $ (p,\vvv q) \in \Z \times \Z^{n}\backslash \{\vvv 0\}$ with $1\le|\vvv q|\le Q$.  Now in the same way if
$\bm\xi$  given  by \eqref{eqn44} is badly approximable leads to the minimal distance estimate  \eqref{eqn45}, the upshot of \eqref{eqn59} is the following statement: with probability one, for every $Q \in \N$ and  a random choice of channel coefficients $h_{ij}$  $(i,j=1,2)$,  the separation between the associated  points $y_1$ given by \eqref{eqn30}  is at least $\kappa(\bm\xi) C_2/ Q^{2 + \ve} $ and so
\begin{equation}\label{eqn60}
\dminVBone\ge     \frac{\kappa(\bm\xi)C_2}{Q^{2+ \ve}}  \, .
\end{equation}
Just to clarify,  that  $\bm\xi$ in the above  corresponds to the point given by \eqref{eqn44} associated with  the choice of the channel coefficients.  Note that instead of \eqref{eqn44},  one can equivalently consider $\bm\xi$  to be  either of the points given by \eqref{eqn46} and  this would lead to  \eqref{eqn60} with $C_2$ defined by \eqref{eqn47}.
A similar lower bound statement  holds for the minimal distance $\dminVBtwo $  associated  with the points $y_2$ given by \eqref{eqn31}. Of course,  in this case   $\bm\xi$  need to be replaced by $\bm\xi'$ given by \eqref{eqn48} or equivalently \eqref{eqn49}.

\begin{remark}\label{kgkg}
Recall that $\bm\xi$ is given by \eqref{eqn44} or \eqref{eqn46} and $\bm\xi'$ is given by \eqref{eqn48} or \eqref{eqn49} and they are dependent via \eqref{eqn50} and \eqref{eqn51}. Note that any of the maps in \eqref{eqn51} is a diffeomorphism on a sufficiently small neighborhood of almost every point in $\R^2$. Therefore, if $\bm\xi$ avoids a subset of $\R^2$ of measure zero, then so does $\bm\xi'$. Thus, \eqref{eqn60} and  an analogous  bound for $\dminVBtwo$ are simultaneously valid for almost all choices of the channel coefficients.
\end{remark}

\begin{remark} \label{whynotlog}
Note that in the above analysis, if we had worked with the function $\psi: q \to q^{-n} (\log q)^{-1-\ve}$ for some $\ve>0$, we would have obtained the stronger estimate
\begin{equation*}\label{eq2020svsv}
\dminVBone\ge     \frac{\kappa(\bm\xi)C_2}{Q^{2} (\log Q)^{1+ \ve}}  \, .
\end{equation*}
It will be soon be clear that \eqref{eqn60} is all we need  for estimating the DoF  within the context of Example~2.
\end{remark}

A natural question arising from the above discussion is: {\em can the constant $\kappa(\bm\xi)$ within Corollary~\ref{cor3} and thus \eqref{eqn60} be made independent of $\bm\xi$?} Unfortunately, this is impossible to guarantee with probability one; that is, for almost all $\bm\xi \in \II^n$. To see this, consider the set

\begin{equation} \label{eqn61}
\mathcal{B}_{n}(\psi,\kappa):=\left\{ \bm\xi \in \II^n :  \begin{array}{l}
|q_1\xi_1+\dots+q_n\xi_n+p|> \kappa\psi(|\vvv q|)\\[1ex] ~ \  \ \forall   \ \  (p,\vvv q) \in \Z \times \Z^{n}\backslash \{\vvv 0\}
                                                     \end{array}
\right\}.
\end{equation}

\noindent Then for any $\kappa$ and $\psi$, observe that $\mathcal{B}_{n}(\psi,\kappa)$ will not contain the region
$$
[-\kappa\psi(|\vvv q|),\kappa\psi(|\vvv q|)]\times\R^{n-1}
$$
when  $\vvv q=(1,0,\dots,0) \in \Z^n$.
This region has positive probability; namely $2 \kappa \psi(1)$), and so the complement (which contains $\mathcal{B}_{n}(\psi,\kappa)$) cannot have probability one.
Nevertheless, the following result  provides not only an explicit dependence on the probability of  $\mathcal{B}_{n}(\psi,\kappa)$ on $\kappa$, but shows that it  can be made arbitrarily close to one upon taking $\kappa$ sufficiently small.

\begin{theorem}[Effective convergence Khintchine-Groshev for one linear form] \label{EKG}
Let  $\psi:\Rp\to\Rp$ be a function such that
\begin{equation*}
  \sum_{q=1}^\infty q^{n-1}\psi(q)<\infty\,.
\end{equation*}
Then, for any $\kappa>0$
$$
{\bf Prob}(\mathcal{B}_{n}(\psi,\kappa)) \, \ge  \,  1-4n\kappa\sum_{q=1}^\infty (2q+1)^{n-1}\psi(q)\,.
$$

\end{theorem}

\begin{proof} Note  that
$$
\mathcal{B}_{n}(\psi,\kappa)=\II^n\setminus \bigcup_{\vvv q\in\Z^n \backslash \{ \vvv 0 \} }E_{\vvv q}(\psi)\,,
$$
where
$$
E_{\vvv q}:=\big\{\bm\xi\in\II^n:|q_1\xi_1+\dots+q_n\xi_n+p| \,  \le \, \kappa\psi(|\vvv q|)\text{ for some } p\in\Z\big\}\,.
$$

\noindent Now, it is  not difficult to verify that $|E_{\vvv q}|_{n}=2\kappa\psi(|\vvv q|)$ - see \cite[Lemma~8]{MR548467} for details.  Thus, it follows that
\begin{align*}
{\bf Prob}(\mathcal{B}_{n}(\psi,\kappa)):=|\mathcal{B}_{n}(\psi,\kappa)|_{n}
\ & \ \ge  \ 1-\sum_{\vvv q\in\Z^n \backslash \{ \vvv 0 \}  }|E_{\vvv q}|_{n}\\[0ex]
\ & \ = \ 1-\sum_{\vvv q\in\Z^n \backslash \{ \vvv 0 \} }2\kappa\psi(|\vvv q|)\\[0ex]
\ & \ = \ 1-\sum_{q=1}^\infty\sum_{\substack{\vvv q\in\Z^n\\ |\vvv q|=q}}2\kappa\psi(|\vvv q|)\\[0ex]
\ & \ = \ 1-2\kappa\sum_{q=1}^\infty\psi(q)\sum_{\substack{\vvv q\in\Z^n\\ |\vvv q|=q}}1\\[0ex]
\ & \ \ge \ 1-2\kappa\sum_{q=1}^\infty\psi(q)2n(2q+1)^{n-1}\,,
\end{align*}
as desired.
\end{proof}

Having set up the necessary mathematical theory, we now  turn our attention to calculating  the DoF for the two-user $X$-channel considered in Example~2.  The advantage of utilising the Khintchine-Groshev approach rather than the badly approximable approach, is that the value we obtain  is not only sharp but it is valid  for almost every realisation of the four channel coefficients  $h_{ij}$  $(i,j=1,2)$.  Here, almost every is naturally with respect to $4$-dimensional Lebesgue measure.  At this point, a  mathematician with little or no background in communication theory (like us) may rightly be crying out for an explanation of what is meant by the Degrees of Freedom of communication channels.  We will attempt to provide a basic and in part a  heuristic  explanation within the context of Example~2.  For a more in depth and general discussion we refer the reader to Chapter~\chaptwo.

The simplest  example of a communication channel is one involving just one transmitter and one receiver.  For obvious reasons,  such a setup is referred to as a {\em point to point channel}.   The DoF of any other communication channel model is in essence a measure of its efficiency compared with using multiple point to point channels. In making any comparison, it is paramount to compare like with like.
Thus,  given  that the noise $z_i$  ($i=1,2$) at both receivers $R_i$ within Example~2 is assumed to have  normal distribution $\cN(0,1)$,  we assume that the noise within the benchmark point to  point channel has normal distribution $\cN(0,1)$.  In the same vein,  we assume that  the messages the users transmit within both models are integers lying in $\{0, \ldots, Q\}$; that is to say that $Q$ is the same in  Example~2 and the point to  point channel model.  The parameter $Q \in \N$ is obviously a bound on the message size and it  provides a  bound on the number of binary digits ({\em bits})  that can be transmitted instantaneously as a single bundle. Indeed, sending the integer $Q$ requires transmitting a bundle of
$\lfloor \log Q \rfloor+1\approx \log Q$ bits, where the logarithm is to the base 2.  Loosely speaking, the larger the message to be sent  the larger the ``power'' required to transmit the message  (transmitting  instantaneously more bits requires more energy).  Thus a bound on the message size $Q$ corresponds to imposing  a {\em power constraint} $P$ on the channel  model under consideration. For physical reasons, that are not particularly relevant to the discussion here, the power is comparable to the square of the message size.  The upshot is that a  power constraint $P$ on the channel model places a bound on the maximal number of bits  that can be reliably transmitted as a single bundle. With this in mind, the (total) DoF of the channel characterises the number (possibly fractional) of simple point-to-point channels, needed to reliably transmit the same maximal number of  bits as the power constraint $P$ tends to infinity.  We now calculate the total DoF  for the concrete setup of Example~2. The exposition given below is a simplified version of that presented in \cite{MR3245356}.

 In relation to Example 2, the power constraint $P$ means that
\begin{equation}\label{eqn62}
\text{$|x_1|^2\le P$ and $|x_2|^2\le P$}  \, ,
\end{equation}
where $ x_1$ and $x_2$ are the codewords transmitted by $S_1$ and $S_2$  as given by \eqref{eqn28}.   Now notice that since the messages $u_1, u_2, v_1, v_2$ are integers lying in $\{0,\dots,Q\}$, it follows that $P$ is comparable to $(\lambda Q)^2$  -- the channel coefficients $h_{ij}$ are fixed. Recall, that $ \lambda \ge 1 $ is a scaling factor which is at our disposal and this will be utilized shortly.
It is shown in \cite{MR3245356}, that the probability of error in transmission within Example~2 is bounded above by
\begin{equation}\label{eqn63}
\exp\left(-\frac{\dminVB^2}{8}\right)\, ,
\end{equation}
where
$$\dminVB=\min\{\dminVBone,\dminVBtwo\}.$$

\noindent It is a standard requirement that this probability should tend to zero as $P\to\infty$.  In essence, this is what it means for the transmission  to be reliable.  Then, on  assuming \eqref{eqn60} -- which holds for almost every realisation of the channel coefficients --  it follows that
 \begin{equation}\label{eqn64}
\dminVB  \gg \frac{\lambda}{Q^{2+\ve}}\,,
\end{equation}
and so the quantity  \eqref{eqn63} will tend to zero as $Q\to\infty$ if we set
$$\lambda=Q^{2+2\ve} \, . $$
The upshot of this is that we will achieve reliable transmission under the power constraint \eqref{eqn62} if we set $P$ to be comparable to $Q^{6 + 4 \ve}$; that is
$$
 Q^{6 + 4 \ve}  \ll P  \ll Q^{6 + 4 \ve} \, .
$$
Now in Example~2, we simultaneously transmit 4 messages, namely $u_1,u_2,v_1,v_2$, which independently take values between $0$ and $Q$. Therefore, in total we transmit approximately $ 4 \times \log Q $   bits, which with our choice of $P$ is an  achievable total rate of reliable transmission; however, it may not be maximal.
We now turn our attention to the simple point to point channel in which the noise has normal distribution $\cN(0,1)$.   In his pioneering work during the forties, Shannon \cite{MR28549} showed that such  a channel subject to the power constraint $P$ achieves the maximal rate of reliable transmission $\frac12\log (1+P)$ -- for further details see \chaptwosec{}
Chapter~\chaptwo.
On comparing the  above rates of reliable transmission for the two models under the same power constraint,    we get that the total DoF of the two-user $X$-channel described in Example 2 is at least
\begin{equation}\label{eqn65}
\lim_{P\to\infty}\frac{4\log Q}{\frac12\log(1+P)}=\lim_{Q\to\infty}\frac{4\log Q}{\frac12\log(1+Q^{6+4\ve})}=\frac{4}{3+2\ve}.
\end{equation}
Given that $\ve>0$ is arbitrary, it follows that for almost every realisation of the channel coefficients
  $$
  {\rm DoF} \ge \frac43 \, .
  $$
Now it was shown in \cite{MR2446746} that the total DoF of a two-user $X$-channel is upper bounded by $4/3$  for all choices of the channel coefficients, and so it follows that for almost every realisation of the channel coefficients
\begin{equation}\label{eqn66}
{\rm DoF}=\frac43\,.
\end{equation}

\noindent For ease of reference we formally state these findings,  the full details of which can be found in \cite{MR3245356},  as a theorem.

\begin{theorem}  \label{MotDOF}
 For almost every realisation of the four channel coefficients $h_{ij}$  $(i,j=1,2)$, the total DoF of the two-user $X$-channel  is $\frac43$.
\end{theorem}

\bigskip

\begin{remark} \label{reiterate}
We reiterate  that by  utilising the Khintchine-Groshev approach rather than the badly approximable approach (i.e. exploiting the lower bound \eqref{eqn60} instead of  \eqref{eqn45} or equivalently \eqref{eqn55} for the minimal distance), we obtain  \eqref{eqn66} for the DoF  that is  valid  for almost every realisation of the four channel coefficients  $h_{ij}$  $(i,j=1,2)$ rather than on a set of $4$-dimensional Lebesgue measure zero.  In \S\ref{singSV},  we shall go further and show that any  exceptional set of channel coefficients for which \eqref{eqn66} fails  is a subset arising from the notion of jointly singular points. This  subset is then shown (see Theorem~\ref{prob3})  not only to have measure zero but to have dimension strictly less than $4$ -- the dimension of the space occupied by the channel coefficients.  In short, our improvement of  Theorem~\ref{MotDOF} is given by Theorem~\ref{MotDOF2}.
\end{remark}

\subsection{Dirichlet improvable and non-improvable points: achieving optimal separation}   \label{eyesight}

We now show that there are special values of $Q$ for which  the minimal distance $\dminVBone$ satisfies \eqref{eqn55}  with  $\kappa'$ as close to one as desired. Recall, the larger the minimal distance   the more tolerance we have for noise. The key is to exploit the (abundant) existence of points for which Dirichlet's theorem cannot be improved.

\begin{definition}[Dirichlet improvable and non-improvable points]  \label{DIDEF}
Let $0<\kappa'<1$. A point $\bm\xi\in\R^n$ is said to be {\em $\kappa'$-Dirichlet improvable} if for all sufficiently large $Q\in\N$ there are integer points $(p,\vvv q)\in \Z \times \Z^n$ with $1\le |\vvv q|\le Q$ such that
  \begin{equation}\label{eqn67}
    \left|q_1\xi_1+\dots+q_n\xi_n+p\right| < \kappa'Q^{-n}\,.
  \end{equation}
A point $\bm\xi\in\R^n$  is said to be {\em Dirichlet non-improvable} if for any $\kappa'<1$ it is not $\kappa'$-Dirichlet improvable. Thus, explicitly, $\bm\xi\in\R^n$ is {\em Dirichlet non-improvable} if for any $0<\kappa'<1$ there exists arbitrarily large $Q\in\N$ such that for all integer points $(p,\vvv q)\in \Z \times \Z^n$ with $1\le |\vvv q|\le Q$
  \begin{equation}\label{eqn68}
    \left|q_1\xi_1+\dots+q_n\xi_n+p\right| \ge \kappa'Q^{-n}\,.
  \end{equation}
\end{definition}

\vspace*{2ex}

\begin{remark} Note that  Dirichlet non-improvable points are not the same as those considered in the probabilistic approach of \S\ref{probasn}.  There  the emphasis is on both $\kappa'$ and  $Q$ being uniform.
\end{remark}

\vspace*{2ex}
In a follow-up paper \cite{MR279040} to their one-dimensional work cited in \S\ref{Improve1}, Davenport $\&$ Schmidt  showed that Dirichlet improvable points in $\R^n$ form a set  $\DI(n)$  of $n$-dimensional Lebesgue measure zero.
Hence, a randomly picked point in $\R^n$ is Dirichlet non-improvable with probability one.  The upshot of this is the following consequence: for almost every random choice of the four channel coefficients  $h_{ij}$  $(i,j=1,2)$  and for any $\ve>0$ there exist arbitrarily large integers $Q$ such that the minimal  distance $\dminVBone$  between the associated points  given by \eqref{eqn30} satisfies
 \begin{equation}\label{eqn69}
\dminVBone\ge \frac{(1-\ve)\lambda h_{11} h_{12}}{Q^{2}} = (1-\ve) \, \frac{C_2}{Q^{2}}\,.
\end{equation}
 To conclude,   the Dirichlet non-improvable  approach allows us to almost surely  achieve the best possible separation, within the factor $(1-\ve)$ of the theoretic upper bound \eqref{eqn39}, for an infinite choice of integer parameters $Q\in \mathcal{Q}_1$.

 \begin{remark} \label{nooverlap}  Obviously, we can obtain an analogous lower bound statement for $\dminVBtwo$ for an infinite choice of integer parameters $Q\in \mathcal{Q}_2$.  However,  it is not guaranteed that the integer sets $ \mathcal{Q}_1 $ and $ \mathcal{Q}_2 $  overlap and thus the problem of optimising $\dminVBone$  and $\dminVBtwo$ simultaneously remains open.
 \end{remark}

\subsection{Singular and non-singular points: the DoF of $X$-channel revisited} \label{singSV}

With reference to Example 2, the Khintchine-Groshev and the Dirichlet non-improvable  approaches  allows us to achieve good separation for the minimal distances  (i.e., lower  bounds for $\dminVBone$ and $\dminVBtwo$ that are at most `$\ve$-weaker' than the theoretic upper bounds) for almost all choices of the four channel coefficients $h_{ij}$ $ (i,j =1,2)$. We now turn to the question of {\em whether good separation can be achieved for a larger class of channel coefficients?  For example, is it possible that the set of exceptions not only has measure zero (as is the case with the aforementioned approaches) but has dimension strictly less than four (the  dimension of the space occupied by the channel coefficients)?}   In short the answer is yes. The key is to make use of the following weaker notion than that of  Dirichlet non-improvable points (cf. Definition~\ref{DIDEF}).

\begin{definition}[Singular and non-singular points]\label{def5}
A point $\bm\xi\in\R^n$ is said to be {\em singular} if it is $\kappa'$-Dirichlet improvable for any $\kappa'>0$. A point  $\bm\xi\in\R^n$  is said to be {\em non-singular} (or {\em regular}) if it is not singular. Thus, explicitly, $\bm\xi\in\R^n$ is {\em non-singular} if there exists a constant $\kappa'=\kappa'(\bm\xi)>0$ such that there exist arbitrarily large integers $Q\in\N$ so that for all integer points  $(p,\vvv q)\in \Z \times \Z^n$ with $1\le |\vvv q|\le Q$
  \begin{equation}\label{eqn70}
    \left|q_1\xi_1+\dots+q_n\xi_n+p\right| \ge \kappa'Q^{-n}\,.
  \end{equation}
\end{definition}

\noindent By definition, any singular point is trivially Dirichlet improvable.  Equivalently,  any Dirichlet non-improvable point is trivially  non-singular.

We let $\Sing(n)$ denote the set of singular points in  $\R^n$.  It is easily verified   that $\Sing(n)$ contains every rational hyperplane in $\R^n$. Therefore,
$$
n-1 \le  \dim \Sing(n) \le n   \, .
$$
Here and throughout,  $\dim X$ will denote the Hausdorff dimension of a
subset $X$ of  $\R^n$.   For the sake of completeness,  we provide the definition.

\begin{definition}[Hausdorff dimension]  \label{DefHD}
Let $X\subset\R^n$. Then the Hausdorff dimension $ \dim X$ of $X$ is defined to be the infimum of $s>0$
such that for any $\rho>0$ and any $\ve>0$ there exists a cover of $X$ by a countable
family $B_i$ of balls of radius $r(B_i)<\rho$ such that
$$
\sum_{i=1}^\infty r(B_i)^s<\ve\,.
$$
\end{definition}

\begin{remark} \label{usefulHD}
For most sets upper bounds for the Hausdorrf dimension can be obtained using natural covering by small balls.  Indeed, let $X\subset\R^n$   and  $\rho>0$ and  suppose $X$ can be covered by $N_\rho(X)$ balls of radius  at most  $\rho$.   Then, it immediately follows for the above definition that
$$
\dim X   \le   \limsup_{\rho \to 0}  \frac{ \log N_\rho(X) }{-\log \rho}   \, .
$$
\end{remark}

\medskip

\noindent Note that the Hausdorff dimension of planes and more generally smooth submanifolds of $\R^n$ is the same as their usual `geometric' dimension.  The  middle third Cantor set $\cK$ is the standard classical example of a set with fractal dimension.  Recall, $\cK$ consists of all real numbers in the unit interval whose base 3 expansion does not contain the `digit' 1; that is
$$
\cK :=  \{ \xi \in [0,1]   : \, \xi = \textstyle{\sum_{i=1}^{\infty} } a_i 3^{-i}   \quad {\rm with \ } a_i = 0 \ {\rm or \ } 2   \} \, .
$$
It is well known that
$$
\dim\cK=\frac{\log2}{\log3}\,.
$$
For a proof of this and a lovely  introduction to the mathematical world of fractals, see the bible \cite{MR1102677}.

Now returning to singular points,  in the case $n=1$, a nifty argument due to
Khintchine \cite{MR1512207}
dating back to the twenties shows that a real number is singular if and only
if it is rational; that is
\begin{equation}\label{eqn71}
\Sing(1) = \Q \, .
 \end{equation}
 Recently, Cheung $\&$ Chevallier \cite{MR3544282}, building on the spectacular $n=2$
 work of Cheung \cite{MR2753601}, have proved the following dimension statement for
 $\Sing(n)$.

\begin{theorem}[Cheung \&{} Chevallier]\label{thm6}
Let $n\ge2$. Then
$$
\dim\Sing(n)=\frac{n^2}{n+1} \, .
$$
Thus,  $$\operatorname{codim}\Sing(n)=\dfrac{n}{n+1}   \, . $$
\end{theorem}

\begin{remark} Note that since $ \frac{n^2}{n+1} > n-1 $,  the theorem
immediately implies that in higher dimensions $ \Sing(n) $
does not simply correspond to rationally dependent $\bm\xi \in \R^n$
as in the one-dimensional case -- the theory is much richer.
Also observe, that since $ \frac{n^2}{n+1} < n $,  the set $\Sing(n)$
is strictly smaller than $\R^n$ in terms of its Hausdorff dimension. How much
smaller is measured by its codimension; {\em i.e.} $n -\dim\Sing(n)$.
\end{remark}

Now  if the four channel coefficients  $h_{ij}$  $(i,j=1,2)$   happen to be such that the corresponding point  $ \bm\xi  \in \R^2$ given by \eqref{eqn44}
is non-singular, then there exist arbitrarily large integers $Q$ such that the minimal  distance $\dminVBone$  between the associated points  given by \eqref{eqn30} satisfies
 \begin{equation}\label{eqn72}
\dminVBone\ge \frac{\kappa'({\bm\xi}) \lambda h_{11} h_{12}}{Q^{2}} = \, \frac{\kappa'({\bm\xi}) \, C_2}{Q^{2}}\,.
\end{equation}
This of course is similar to the statement in which the point  ${\bm\xi}$  is Dirichlet non-improvable with the downside that we cannot replace the constant $\kappa'({\bm\xi}) $ by $(1-\ve)$  as in \eqref{eqn69}.  However, the advantage is that it is valid for a much larger set of  channel coefficients; namely, the exceptional set of channel coefficients $(h_{11},h_{12}, h_{21}, h_{22})\in \Rp^4$ for which  \eqref{eqn72} is not valid has dimension
$\frac{10}{3}$, which is strictly smaller than $4$ -- the dimension of the ambient space occupied by
$(h_{11},h_{12}, h_{21}, h_{22})$. This result seems to be new and we state it formally.

\begin{proposition} \label{ohyes2}
For  all choices of  channel coefficients $(h_{11},h_{12}, h_{21}, h_{22}) \in \Rp^4$, except on a subset of codimension $\frac23$, there exist arbitrarily large integers $Q$ such that the minimal  distance $\dminVBone$  between the associated points  given by \eqref{eqn30} satisfies \eqref{eqn72}.
\end{proposition}

\noindent The proof of the proposition will make use of the following two well known results from fractal geometry \cite{MR1333890}.

\begin{lemma}[Marstrand's Slicing Lemma] \label{lem3}
For any $X\subset \R^k$  and $l \in \N$, we have that
$$
\dim(X\times\R^\ell)=\dim X+\ell\,.
$$
\end{lemma}

\begin{lemma}\label{lem4}
Let $X\subset \R^k$ and $g:\R^k\to\R^k$ be a locally bi-Lipschitz map. Then
$$
\dim \big(g(X)\big)=\dim X\,.
$$
\end{lemma}

\begin{proof}[Proof of Proposition~\ref{ohyes2}.]
Consider the following map on the channel coefficients
$$
g:\Rp^4\to\Rp^4\quad\text{such that}\quad g(h_{11},h_{12},h_{21},h_{22})=\left(h_{11},h_{12},\frac{h_{22}}{h_{12}},\frac{h_{21}}{h_{11}}\right)\,.
$$
As we have already discussed, for any $\bm\xi$ given by \eqref{eqn44} such that $\bm\xi\in\Rp^2\setminus\Sing(2)$ we have that \eqref{eqn72} holds. Hence, \eqref{eqn72} holds for any choice of channel coefficients such that
\begin{equation}\label{eqn73}
(h_{11},h_{12},h_{21},h_{22})\not\in g^{-1}\Big(\Rp^2\times \big(\Rp^2\cap\Sing(2)\big)\Big)\,.
\end{equation}
By Lemma~\ref{lem3} and Theorem~\ref{thm6}, it follows that
$$
\textrm{codim}\,\Big(\Rp^2\times \big(\Rp^2\cap\Sing(2)\big)\Big)=\frac23 \, .
$$
Finally, note that locally at every point of $\Rp^4$ the map $g$ is a $C^1$ diffeomorphism and hence is bi-Lipschitz. Therefore, by Lemma~\ref{lem4} it follows that  $g^{-1}$ preserves dimension and thus the codimension of the right hand side of \eqref{eqn73} is $\tfrac23$. This completes the proof.
\end{proof}

\begin{remark}
Just to clarify,  that  $\bm\xi$ appearing  in \eqref{eqn72}  corresponds to the point given by \eqref{eqn44} associated with  the choice of the channel coefficients $h_{ij}$  $(i,j=1,2)$ and $\kappa'(\bm\xi)>0$  is a constant dependent on $\bm\xi$. Note that instead of \eqref{eqn44}, one can equivalently consider $\bm\xi$  to be  either of the points given by \eqref{eqn46} and  this would lead to  \eqref{eqn72} with $C_2$ defined by \eqref{eqn47}.
\end{remark}

 \medskip

Naturally, the analogue of Proposition~\ref{ohyes2} holds  for the minimal distance $\dminVBtwo$ between the  associated points  given by \eqref{eqn33}.   However, as in the Dirichlet non-improvable setup (cf. Remark \ref{nooverlap}),  we cannot guarantee that the arbitrary large  integers $Q$ on which the lower bounds for the minimal distances are attained,  overlap.  If we could guarantee infinitely many overlaps,  it would enable us to strengthen Theorem \ref{MotDOF}  concerning the Degrees of Freedoms (DoF) of the two-user $X$-channel described in Example~2.    With this goal in mind,   it is appropriate to introduce the following notion of jointly singular points.

 \begin{definition}[Jointly singular and non-singular points]\label{def6}
The pair of points $(\bm\xi_1,\bm\xi_2)\in\R^n\times\R^n$ is said to be {\em jointly singular} if for any $\ve>0$ for all sufficiently large $Q\in\N$ there exists an integer point  $(p,\vvv q)\in \Z \times \Z^n$ with $1\le |\vvv q|\le Q$ satisfying
$$
\min_{1\le j\le 2}|q_1\xi_{j,1}+\dots+q_n\xi_{j,n}+p|<\ve Q^{-n}\,,
$$
where $\bm\xi_j=(\xi_{j,1},\dots,\xi_{j,n})$, $j=1,2$.
The pair $(\bm\xi_1,\bm\xi_2)\in\R^n\times\R^n$ will be called {\em jointly non-singular} if it is not jointly singular, that is if there exists a constant $\kappa'=\kappa'(\bm\xi_1,\bm\xi_2)>0$ such that there exist arbitrarily large $Q\in\N$ so that for all integer points  $(p,\vvv q)\in \Z \times \Z^n$ with $1\le |\vvv q|\le Q$
  \begin{equation}\label{eqn74}
    \min_{1\le j\le 2}\left|q_1\xi_{j,1}+\dots+q_n\xi_{j,n}+p\right| \ge \kappa'Q^{-n}\,.
  \end{equation}
\end{definition}

\noindent The set of jointly singular pairs in $\R^n\times\R^n$
will be denoted by $\Sing^{2}(n)$.  This set is not and should not
be confused with the standard simultaneous singular set corresponding to two linear forms in $n$ variables (see  \S\ref{FR}).


\medskip

The above notion of jointly non-singular pairs enables us to prove the following DoF statement.

\begin{proposition}\label{thm7}
Let $(h_{11},h_{12}, h_{21}, h_{22}) \in \Rp^4$ be given and let $\bm\xi$ be any of the points \eqref{eqn44} or \eqref{eqn46}, let $\bm\xi'$ be any of the points \eqref{eqn48} or \eqref{eqn49}. Suppose that
\begin{equation}\label{eqn75}
(\bm\xi,\bm\xi')\not\in\Sing^{2}(2)\,.
\end{equation}
Then \eqref{eqn66} holds, that is the total DoF of the two-user $X$-channel with $h_{ij}$  $(i,j=1,2)$ as its channel coefficients  is $\frac43$.
\end{proposition}

\begin{proof}
To start with, simply observe that condition \eqref{eqn75} means that
 there exist $\kappa'>0$ and an infinite subset $\cQ\subset\N$ such that for every $Q\in\cQ$ and all integer points  $(p,\vvv q)\in \Z \times \Z^2$ with $1\le |\vvv q|\le Q$
  \begin{equation}\label{eqn76}
    \left|q_1\xi_{1}+q_2\xi_{2}+p\right| \ge \kappa'Q^{-2}\quad\text{and}\quad \left|q_1\xi'_{1}+q_2\xi'_{2}+p\right| \ge \kappa'Q^{-2}\, .
  \end{equation}
Consequently, for every $Q\in\cQ$ we can guarantee  that \eqref{eqn72} and its analogue for  $\dminVBtwo$ are simultaneously valid.
This in turn implies \eqref{eqn64} for every $Q\in\cQ$. From this point onwards, the rest of the argument given in \S\ref{sv2.4}  leading to \eqref{eqn66}  remains unchanged apart from the fact that the limit in \eqref{eqn65} is now along $Q\in\cQ$ rather than the natural numbers.
\end{proof}

Proposition~\ref{thm7} provides a natural pathway for strengthening Theorem~\ref{MotDOF}. This we now describe.  It is reasonable to expect that the set of $(\bm\xi,\bm\xi')$ not satisfying \eqref{eqn75} is of dimension strictly smaller than four -- the dimension of the ambient space. Indeed, this is something that we are able to prove.

\begin{theorem}\label{prob1} Let $n\ge2$. Then
\begin{equation}\label{77sv}
\dim\Sing^2(n) = 2n-\frac{n}{(n+1)} \, .
\end{equation}
\end{theorem}

\noindent The theorem will easily follow from a more general statement concerning systems of linear forms proved in \S\ref{FR} below; namely, Theorem~\ref{prob5}.
Note that Theorem~\ref{prob1} is not enough for improving Theorem~\ref{MotDOF}. Within Proposition~\ref{thm7}, the point $\bm\xi$ is given by \eqref{eqn44} or \eqref{eqn46} and $\bm\xi'$ is given by \eqref{eqn48} or \eqref{eqn49}, and are therefore dependent via \eqref{eqn50} and \eqref{eqn51}.
The above theorem does not take into consideration this dependency. This is rectified by the following result.

\begin{theorem}\label{prob3}
Let $\bm f:U\to\R^n$ be a locally bi-Lipschitz map defined on an open subset $U\subset\R^n$ and let
$$
\Sing^2_{\bm f}(n):=\{\bm\xi\in U:(\bm\xi,\bm f(\bm\xi))\in\Sing^2(n)\}\,.
$$
Then
\begin{equation}\label{vb100}
\dim \Sing^2_{\bm f}(n)\le n-\frac{n}{2(n+1)}<n.
\end{equation}
\end{theorem}

As with Theorem~\ref{prob1}, we  defer the proof of the above theorem till \S\ref{FR}. Combining the $n=2$ case of Theorem~\ref{prob3} with Proposition~\ref{thm7} gives the following strengthening of the result of Motahari et al on the DoF of a two-user X-channel (Theorem~\ref{MotDOF}).

\begin{theorem}  \label{MotDOF2}
 The total DoF of the two-user $X$-channel given by \eqref{eqn66} can be achieved for all realisations of the channel coefficients $h_{ij}$  $(i,j=1,2)$ except on a subset of Hausdorff dimension $\le 4-\frac{1}{3}$; that is, of codimension $\ge\frac13$.
\end{theorem}

Clearly, $\Sing(n)$ is a subset $\Sing^2_{\bm f}(n)$. Therefore, it follows that
 $$\dim\Sing^2_{\bm f}(n)\ge \dim\Sing(n)$$

\noindent which together with  Theorem~\ref{thm6} implies that for $n\ge2$

$$
\dim\Sing^2_{\bm f}(n) \ge  \frac{n^2}{n+1} = n-\frac{n}{n+1}\,.
$$

\noindent The gap between this lower bound and the upper bound of Theorem~\ref{prob3} leaves open the  natural  problem of determining $ \dim\Sing^2_{\bm f}(n)$  precisely.  We suspect that the lower bound is sharp.

\begin{problem}\label{prob4}
Let $n\ge2$ and $\bm f:U\to\R^n$ be a locally bi-Lipschitz map defined on an open subset $U\subset\R^n$. Verify if
$$
\dim\Sing^2_{\bm f}(n)= \frac{n^2}{n+1}  \,.
$$
\end{problem}

\noindent
Note that to improve Theorem~\ref{MotDOF2} we are only interested in the case $n=2$ of Problem~\ref{prob4} with $\bm f$ given by \eqref{eqn51}.

\subsection{Systems of linear forms} \label{FR}

 To date, we have in one form or another exploited  the theory of Diophantine approximation of a single linear form in $n$ real variables.   In fact, Example~1 only really requires the notions and results with  $n=1$ while Example~2 requires them  with  $n=2$.  It is easily seen,  that in either of these examples, if we increase the number of users (transmitters) $S$  then we increase the numbers of variables appearing in the linear form(s) associated with the received message(s) $y$.  Indeed, within the setup of Example~2 (resp. Example~1)  we would need to use the  general  $n$ (resp. $n-1$) variable theory if we had $n$ transmitters.

The majority of the  Diophantine approximation theory for a single linear form is a special case of a general theory addressing  systems of $m$ linear forms in $n$ real variables.  For the sake of completeness, it is appropriate to provide a brief taster of the general Diophantine approximation theory with an emphasis on those aspects used in analysing communication channel models. It should not come as a surprise that the natural starting point is Dircihlet's theorem for systems of linear forms. Throughout, let $n,m \geq 1$ be integers and  $\Mmn$ denote the set of $n\times m$ matrices $\bm\Xi=(\xi_{i,j})  $ with entries from $\R$.   Clearly, such a matrix represents the coordinates of a  point in $\R^{nm}$.  Also, given $(\vvv p,\vvv q) \in  \Z^m \times \Z^{n}$ let
$$
|\vvv q\bm\Xi+\vvv p| :=  \max_{1\le j\le m}| \vvv q .  \bm\xi_j
 +p_j|  \, ,
$$
where $\bm\xi_j :=(  \xi_{1,j},  \ldots,\xi_{n,j})^t \in \R^n$ is the $j$'th column vector of $\bm\Xi$ and $\vvv q . \bm\xi_j := q_1\xi_{1,j} + \ldots + q_n\xi_{n,j}$ is the standard dot product.

\vspace*{2ex}

\begin{theorem}[Dirichlet's Theorem for systems of linear forms]\label{thmsv3}
For any $\bm\Xi\in \Mmn$ and any $Q\in\N$ there exists  $(\vvv p,\vvv q) \in \Z^m \times \Z^{n}$ such that
\begin{equation*}\label{eqn38sv}
|\vvv q\bm\Xi+\vvv p|  <   Q^{-\frac{n}{m}}    \qquad \textrm{ and }  \qquad  1 \leq  |\vvv q| \leq Q \, .
\end{equation*}
\end{theorem}

The theorem is a relatively straightforward consequence of Minkowski's theorem for systems of linear forms; namely Theorem \ref{MTLF} in \S\ref{eg2section}.   For the details of the deduction  see for example \cite[Chapter~2]{schmidt1996diophantine}.   In turn, a straightforward consequence of the above theorem is the following natural extension of Corollary~\ref{cor1} to systems of linear form.

\begin{corollary}\label{cor2sv3}
For any $\bm\Xi\in \Mmn$  there exists infinitely many $(\vvv p,\vvv q) \in \Z^m \times \Z^{n}\backslash \{\vvv 0\} $ such that
  \begin{equation*}\label{eqn42sv3}
   |\vvv q\bm\Xi+\vvv p| < |\vvv q|^{-\frac{n}{m}}\,  .
  \end{equation*}
\end{corollary}

\noindent Armed with Theorem~\ref{thmsv3} and its corollary, it does not require much imagination to extend the single linear  form notions of badly approximable (cf. Definition~\ref{BADdef}) and Dirchlet improvable (cf. Definition~\ref{DIDEF})  to systems of linear forms.  Indeed, concerning the former we arrive at the set
$$
\Bad(n,m):=   \left\{  \bm\Xi \in \Mmn  \  : \  \liminf_{\substack{\vvv q\in\Z^n  : \\[0.2ex] |\vvv q|  \to \infty}} \,  |\vvv q|^{\frac{n}{m}} |\vvv q\bm\Xi-\vvv p| > 0  \right\}  \, .
$$
This clearly coincides with $\Bad(n)$  when $m=1$.  As we shall soon see, $\Bad(n,m)$  it is a set of zero $nm$-dimensional Lebesgue measure.  Even still, Schmidt \cite{MR195595, MR248090}  showed that it is a large set in the sense that it is of maximal dimension; i.e.
$
\dim \Bad(n,m) = nm \, .
$
 Moving swiftly on, given a function $\psi:\R_+\to\R_+$ let
\[
\cW_{n,m}(\psi):=\left\{ \bm\Xi \in \Mmn(\II) :  \begin{array}{l}
|\vvv q\bm\Xi-\vvv p|<\psi(|\vvv q|) \ \ \text{ for}\\[1ex] \text{i.m. }(\vvv p,\vvv q) \in \Z^m \times \Z^{n}\backslash \{\vvv 0\}  \end{array}
\right\}  \, .
\]

\noindent  Here and below, $  \Mmn(\II) \subset  \Mmn$ denotes the set of $n\times m$ matrices  with entries from $\II=(0,1)$.  The following provides an elegant criterion for   the size of the set $\cW_{n,m}(\psi)$ expressed in terms of $nm$-dimensional Lebesgue measure. When $m=1$, it  coincides with Theorem~\ref{KGT}  appearing in \S\ref{sv2.4}.

\begin{theorem}[The Khintchine-Groshev Theorem]\label{KGT2}
Given any monotonic function $\psi:\R_+\to\R_+$, we have that
$$
|\cW_{n,m}(\psi)|_{nm}=\left\{\begin{array}{cl}
0 &\text{if } \ \sum_{q=1}^\infty q^{n-1}\psi(q)^m<\infty\,,\\[2ex]
1 &\text{if } \ \sum_{q=1}^\infty q^{n-1}\psi(q)^m=\infty\,.
                       \end{array}
\right.
$$
\end{theorem}

\bigskip

\noindent   Consider for the moment the function $\psi(r) =  r^{-\frac{n}{m}}  (\log r)^{-m} $ and observe that
$$
\Bad(n,m) \cap   \Mmn(\II)  \ \subset  \ \Mmn(\II)  \setminus    \cW_{n,m}(\psi) \, .
$$
By Theorem~\ref{KGT2}, $|\cW_{n,m}(\psi)|_{nm} = 1 $. Thus $|\Mmn(\II) \setminus \cW_{n,m}(\psi)|_{nm} = 0 $
and on using the fact that set $\Bad(n,m)$ is invariant under translation by integer $n\times m$ matrices,  it follows that
$$
|\Bad(n,m)|_{nm} = 0  \, .
$$

\noindent  Another immediate consequence of the Khintchine-Groshev Theorem is the following statement (cf. Corollary~\ref{cor3}).

\begin{corollary}\label{cor4}
Let  $\psi:\Rp\to\Rp$ be any function such that
$$
  \sum_{q=1}^\infty q^{n-1}\psi(q)^m<\infty\,.
$$
Then, for almost all $\bm\Xi \in \Mmn$ there exists a constant $\kappa (\bm\Xi)  >0 $  such that
$$
 |\vvv q\bm\Xi+\vvv p | \ >  \  \kappa(\bm\Xi)  \,   \psi(|\vvv q|)  \qquad   \forall  \   \  (\vvv p,\vvv q) \in \Z^m \times \Z^{n}\backslash \{\vvv 0\} \, .
$$
\end{corollary}

\bigskip

The following is the natural generalisation of the set given by \eqref{eqn61} to systems of linear forms and the subsequent statement is the natural generalisation of Theorem~\ref{EKG}. Let

\begin{equation} \label{eq61+}
\mathcal{B}_{n,m}(\psi,\kappa):=\left\{ \bm\Xi \in \Mmn(\II) :  \begin{array}{l}
|\vvv q\bm\Xi+\vvv p|  > \kappa\psi(|\vvv q|)  \\[1ex] ~ \  \ \forall   \  \ (\vvv p,\vvv q) \in \Z^m \times \Z^{n}\backslash \{\vvv 0\}
                                                     \end{array}
\right\}.
\end{equation}

\begin{theorem}[Effective convergence Khintchine-Groshev Theorem]\label{eff}
Suppose that
$$
\sum_{q=1}^\infty q^{n-1}\psi(q)^m<\infty\,.
$$
Then, for any $\kappa>0$
$$
{\bf Prob}(\mathcal{B}_{n,m}(\psi,\kappa))\ge1-2^mn\kappa^m\sum_{q=1}^\infty (2q+1)^{n-1}\psi(q)^m\,.
$$
\end{theorem}

\noindent We highlight the fact that the probability in Theorem~\ref{eff} is assumed to be uniform but it is possible to obtain a version for absolutely continuous distributions as already mentioned in Remark~\ref{vb3978}.
Recall, that the  Khintchine-Groshev theorem (with $m=1$ and $n=2$) underpinned the approach taken in \S\ref{sv2.4} for calculating the Degrees of Freedom of the two-user $X$-channel (cf. Theorem~\ref{MotDOF}).

We bring our selective overview of the general  Diophantine  approximation theory to a close by describing  singular and jointly singular sets for  systems of linear forms.  In the process we shall prove  Theorem~\ref{prob1} and  Theorem~\ref{prob3}.  Recall, that the latter allows us to improve Theorem~\ref{MotDOF}. For ease of comparison, it is convenient  to define the sets of interest as follows:
\begin{equation*}
\Sing(n,m):=\left\{ \bm\Xi \in \Mmn :  \,  \begin{array}{l}
\displaystyle{\min_{\substack{(\vvv p,\vvv q) \in \Z^m \times \Z^{n} \!: \\[0.2ex]0<|\vvv q|\le Q}}} \ \max_{1\le j\le m}Q^{\frac{n}{m}} |\vvv q.\bm\xi_j + p_j |\to 0     \\[-1ex] ~ \hspace*{20ex}  \text{ as \ }Q\to\infty
                                                     \end{array}
\right\}
\end{equation*}
and
\begin{equation} \label{tired}
\Sing^m(n):=\left\{ \bm\Xi \in \Mmn  :  \, \begin{array}{l}
\displaystyle{\min_{\substack{(\vvv p,\vvv q) \in \Z^m \times \Z^{n} \!: \\[0.2ex]0<|\vvv q|\le Q}}} \  \min_{1\le j\le m}Q^{n} |\vvv q.\bm\xi_j + p_j |\to 0     \\[-1ex] ~ \hspace*{20ex}  \text{ as \ }Q\to\infty
                                                     \end{array}
\right\}.
\end{equation}

Clearly, when $m=1$ the above two sets are equal and the elements coincide with the single linear form notion of singular points (cf. Definition~\ref{def5}).   In recent groundbreaking  work \cite{das2019variational},  Das,  Fishman, Simmons $\&$ Urba\'{n}ski   proved the following dimension statement (cf. Theorem~\ref{thm6}) for the set of singular $n \times m$ matrices:    {\em for all $(n,m) \neq (1,1)$, we have that}
\begin{equation*} \label{dsfu}
\dim\Sing(n,m)  \, =  \,  mn \left( 1 - \frac{1}{m+n}  \right)  \, .
\end{equation*}
This resolved a conjecture of Kadyrov, Kleinbock, Lindenstrauss $\&$ Margulis \cite{MR3736492}. In short,  they  showed  that $ \dim\Sing(n,m)  \le  mn ( 1 - 1/(m+n) ) $ and conjectured that their upper bound is in fact sharp.

Regarding the set of jointly singular $n \times m$ matrices, it is clear that when $m=2$ its elements coincide with the single linear form notion of jointly singular points (cf. Definition~\ref{def6}).  Furthermore, it follows from the definition that  for any integers $m_1,m_2 \geq 1$
\begin{equation*}
\Sing^{m_1}(n)\times\R^{n \times m_2}  \, \subset  \, \Sing^{m_1+m_2}(n)\,.
\end{equation*}
This together with Marstrand's Slicing Lemma  and the fact  $\Sing^1(n)=\Sing(n)$,  implies that
\begin{equation}\label{eqn+1}
\dim\Sing^m(n)   \ge (m-1)n+   \dim \Sing(n)  \, .
\end{equation}

\noindent In turn, this together with Theorem~\ref{thm6}, implies that for $ n \ge 2$
\begin{equation}  \label{yes}
\dim\Sing^m(n)   \ge  nm-\frac{n}{(n+1)}  \, .
\end{equation}
The following statement showing that we  have equality in \eqref{yes} is a natural generalisation of Theorem~\ref{prob1} to systems of linear forms.

\begin{theorem} \label{prob5}
Let $m\ge1$, $n \geq 2$. Then
\begin{equation} \label{niceone}
\dim\Sing^m(n)= nm-\frac{n}{(n+1)}\,.
\end{equation}
\end{theorem}

Clearly, when $m=2$ the theorem coincides with Theorem~\ref{prob1}.  In view of \eqref{yes}, the key to establishing Theorem~\ref{prob5} (and  thus Theorem~\ref{prob1})  is  the following upper bound statement.

\begin{theorem} \label{prob5+}
Let $m,n\ge1$. Then
\begin{equation} \label{niceone+}
\dim\Sing^m(n)\le nm-\frac{n}{(n+1)}\,.
\end{equation}
\end{theorem}

\medskip

Note that this upper bound estimate is valid for  $n=1$.  Clearly, in this case it is not sharp when $m=1$ since  $\Sing^1(1)= \Sing(1)= \Q$ and so $\dim \Sing^1(1) =0 $.
Also, note that the lower bound given by \eqref{eqn+1} does not match the upper bound given by \eqref{niceone+}.
Nevertheless, we suspect that \eqref{niceone+} is sharp when $m  \ge 2$.

\begin{problem} \label{prob6}
Let $m\geq 2$. Verify if $\dim\Sing^m(1)=m -  \frac12$.
\end{problem}

\noindent  Clearly, if true then we can replace the conditions on $m$ and $n$ in Theorem~\ref{prob5} by $mn>1$. Although, not explicitly stated or even discussed, it is worth mentioning that Problem~\ref{prob4} concerning the set $\Sing^2_{\bm f}(n)$ also has a natural generalisation  to systems of linear form.

\medskip

The proof of Theorem~\ref{prob5+}  (and indeed Theorem~\ref{prob3}) makes use of the powerful connection between problems in Diophantine approximation an  homogeneous dynamics. This we now briefly explain.  The various Diophantine notions discussed  in this chapter correspond to certain types of orbits of unimodular lattices under the action by diagonal matrices. For instance, as was famously discovered by Dani \cite{dani1985divergent}, a point $\bm\xi=(\xi_1,\dots,\xi_n)\in\R^n$ is badly approximable if and only if the orbit
$$
\left\{ g_tu_{\bm\xi}\Z^{n+1} : t >0 \right\}
$$
is bounded in the homogeneous space $X_{n+1} = \operatorname{SL}_{n+1}(\R)/\operatorname{SL}_{n+1}(\Z)$ of unimodular lattices in $\R^{n+1}$.  Here and throughout,

$$
g_t:=\left(\begin{array}{ccccccc}
                   e^{nt} &&        &&        &\,& \\
                          && e^{-t} &&        && \\
                          &&        && \ddots && \\
                          &&        &&        && e^{-t}
                 \end{array}
\right)
\qquad  \qquad   {\rm for } \quad  t \in \Rp
$$
and

$$
u_{\bm\xi}:=\left(\begin{array}{ccccccc}
                   1 &\,& \xi_1 &\,& \dots &\,& \xi_n \\
                   0 &\,& 1 &\,& &\,& \\
                   \vdots &\,&  &\,&\ddots &\,& \\
                   0 &\,&  &\,& &\,& 1
                 \end{array}
\right)
\qquad  \qquad   {\rm for } \quad  \bm\xi=(\xi_1,\dots,\xi_n)\in\R^n
\,.
$$

\noindent Today this beautiful and powerful equivalence between badly approximable  points and the behaviour of orbits in  $X_{n+1}$ is  simply refereed to as Dani's correspondence.  For background and further details see for instance \cite{MR1043706, MR1928528}.

Recall that the homogeneous space $X_{n+1}$ is non-compact and, by Mahler's criterion, every bounded subset of $X_{n+1}$ is contained in
$$
K_{\varepsilon} := \left\{ \Lambda \in X_{n+1}: \inf_{\vvv v\in\Lambda,\,\vvv v\neq\vvv0}\|\vvv v\|\ge\varepsilon \right\}
$$
for some $\varepsilon>0$, where $\|\cdot\|$ is any norm on $\R^{n+1}$.
With this in mind,  in the same paper \cite{dani1985divergent}, Dani went on to show  that $\bm\xi\in\R^n$ is singular if and only if the orbit
$g_tu_{\bm\xi}\Z^{n+1}$ diverges as $t\to\infty$; that is, for any $\ve>0$ there exists a constant $t_{\ve,\bm\xi}>0$ such that
$$
\forall~t\ge t_{\ve,\bm\xi}\qquad g_tu_{\bm\xi}\Z^{n+1}\not\in K_\ve\,.
$$
This means that the orbit $g_tu_{\bm\xi}\Z^{n+1}$ leaves any bounded set `forever' from some `time' point $t_{\ve,\bm\xi}$.
In the same vein, it can be verifed that the matrix $\bm\Xi \in \Mmn$ composed of the columns  $\bm\xi_1,\dots,\bm\xi_m\in\R^{n}$ is jointly singular if and only if for any $\ve>0$ there exists a constant $t_{\ve,\bm\Xi}>0$ such that
\begin{equation}\label{vb700}
\forall~t\ge t_{\ve,\Xi}\quad\exists~j\in\{1,\dots,m\}\qquad g_tu_{\bm\xi_j}\Z^{n+1}\not\in K_\ve\,.
\end{equation}
Unlike for singular points, for every $j\in\{1,\dots,m\}$ the individual orbit $g_tu_{\bm\xi_j}\Z^{n+1}$ need not be divergent and could in fact for some $\ve>0$ return to the bounded set $K_\ve$ arbitrarily often.

The proof of Theorem~\ref{prob5+} and indeed Theorem~\ref{prob3} rely on the following powerful statement adapted for our application in mind due to Kadyrov, Kleinbock, Lindenstrauss  $\&$ Margulis \cite[Theorem 1.5]{MR3736492}.
Given $\bm\xi\in\R^n$, $N>1$, $s>0$ and $\ve>0$, let
$$
S_{\bm\xi}(N,s,\ve):=\{\ell\in\{1,\dots,N\}:g_{s\ell}u_{\bm\xi}\Z^{n+1}\not\in K_\ve\}\,.
$$
Thus, $S_{\bm\xi}(N,s,\ve)$ corresponds to those times $t=sl$ ($1\le\ell\le N$) for which the orbit  $ g_{t}u_{\bm\xi}\Z^{n+1}$  does not lie in $K_\ve$. In what follows, given a set $X$ we let $\# X $ denote its cardinality.

\begin{theorem}[Kadyrov, Kleinbock, Lindenstrauss  $\&$ Margulis]\label{thm4}
Let $B_1^n$ be the unit ball in $\R^n$ centred at the origin.
Then there exist $s_0 > 1$ and $C>0$ such that for any $s > s_0$, there exists $\ve>0$ such that for any $N\in\N$ and $\delta\in[0,1)$, the set
$$
Z(\ve,N, s, \delta):=\left\{\bm\xi \in B_1^n: \frac{\#S_{\bm\xi}(N,s,\ve)}{N}\ge\delta\right\}
$$
can be covered with $Cs^{3N} e^{(n+1-\delta)nsN}$ balls of radius $e^{-(n+1)sN}$.
\end{theorem}

Note that ${\bm\xi}\in Z(\ve,N, s, \delta)$ if and only if the proportion of times $t=sl\le sN$ ($1\le\ell\le N$) for which the orbit   $ g_{t}u_{\bm\xi}\Z^{n+1}$   avoids $K_\ve$ is at least $\delta$. To be absolutely precise, the case when $\delta =0 $ is not covered by \cite[Theorem 1.5]{MR3736492}.  However, it is trivially true since then $Z(\ve,N, s, \delta)= B_1^n $ and the unit ball can easily be seen to be covered with $C e^{(n+1-\delta)nsN}$ balls of radius $e^{-(n+1)sN}$.  The next statement relates the jointly singular sets of interest to those appearing in  Theorem~\ref{thm4}.

\begin{proposition}\label{prop3}
Let  $\ve>0$ and $s\ge1$.  Then
\begin{equation}\label{vb777}
\Sing^m(n)\cap \left(B_1^n\right)^m\subset \bigcup_{\bm\delta\in\Delta_s}\bigcup_{N_0=1}^\infty\bigcap_{N=N_0}^\infty Z_m(\ve,N, s, \bm\delta)\,,
\end{equation}
where
$$
\Delta_s:=\Big\{\bm\delta=(\delta_1,\dots,\delta_m)\in\tfrac1s\Z^m \cap [0,1)^m : \delta_1+\dots+\delta_m\ge 1-\tfrac{m+1}s\Big\}\,
$$
and
$$
Z_m(\ve,N, s, \bm\delta):=Z(\ve,N, s, \delta_1)\times\dots\times Z(\ve,N, s, \delta_m)\,.
$$
\end{proposition}

\begin{proof} Recall, that given any $\bm\Xi \in \Mmn$ its column vectors are denoted by $\bm\xi_1,\dots,\bm\xi_m\in\R^n$.
Now, suppose that $\bm\Xi\in\Sing^m(n)\cap \left(B_1^n\right)^m$. Then, by \eqref{vb700}, for any  $\ve>0$ and all $N>s^{-1}t_{\ve,\Xi}$ we have that
$$
\{\ell\in\N:s^{-1}t_{\ve,\Xi}\le \ell\le N\}\subset\bigcup_{j=1}^m S_{\bm\xi_j}(N,s,\ve)\,.
$$
It follows that
$$
\sum_{j=1}^{m}\#S_{\bm\xi_j}(N,s,\ve)~\ge~N-s^{-1}t_{\ve,\bm\Xi}\,  .
$$
This implies that
\begin{equation}\label{vb6}
  \sum_{j=1}^m\frac{\#S_{\bm\xi_j}(N,s,\ve)}{N}\ge 1 - \frac{t_{\ve,\bm\Xi}}{s N }  \,.
\end{equation}
For each $j\in\{1,\dots,m\}$,  let $\delta_j\in\frac1s\Z$ be the largest number such that
$$
\frac{\#S_{\bm\xi_j}(N,s,\ve)}{N}\ge \delta_j\,.
$$
Then, with $\bm\delta=(\delta_1,\dots,\delta_m)$ we have that
\begin{equation}\label{vb888}
  \Xi\in Z_m(\ve,N, s, \bm\delta)\,.
\end{equation}
We now show that $\bm\delta\in\Delta_s$.   Since $\#S_{\bm\xi_j}(N,s,\ve)\le N$, we have that $0\le \delta_j\le 1$. By the maximality of $\delta_j$ we have that
$$
\delta_j+\frac1s\ge\frac{\#S_{\bm\xi_j}(N,s,\ve)}{N}\ge \delta_j\,.
$$
By \eqref{vb6}, it follow that for $N$ sufficiently large
\begin{equation}\label{vb6++}
  \sum_{j=1}^m\delta_j  \, \ge  \,  1 - \frac ms- \frac{t_{\ve,\bm\Xi}}{s N }
\, \ge \,   1-\frac {m+1}s   \, .
\end{equation}
Therefore, $\bm\delta\in\Delta_s$. Since $\Delta_s$ is finite, the latter condition together with \eqref{vb888} implies \eqref{vb777} and thereby completes the proof of the proposition.
\end{proof}

As we shall now see, armed with Theorem~\ref{thm4} and Proposition~\ref{prop3}, it is relatively straightforward to establish Theorem~\ref{prob5+} and indeed Theorem~\ref{prob3}.

\begin{proof}[Proof of Theorem~\ref{prob5+}]
Without loss of generality,  it suffices  to show \eqref{niceone+} for the set $\Sing^m(n)\cap \left(B_1^n\right)^m$ instead of $\Sing^m(n)$. In short, this makes use of the fact that  $\Sing^m(n)$ is contained in a countable union of translates of $\Sing^m(n)\cap \left(B_1^n\right)^m$.   By Theorem~\ref{thm4}, for $ s > s_0$ and each $\bm\delta\in\Delta_s$, there exists a cover of $Z_m(\ve,N, s, \bm\delta)$ by
$$
\prod_{j=1}^m Cs^{3N} e^{(n+1-\delta_j)nsN}\ll s^{3mN} e^{(n+1)nmsN-(1-\frac{m+1}{s})nsN}
$$
balls of the same radius
\begin{equation}\label{diam}
 r  =e^{-(n+1)sN}\,.
\end{equation}
Thus, in view of Proposition~\ref{prop3} and the trivial fact that
\begin{equation*}
\#\Delta_s\le (s+1)^{m}\, ,
\end{equation*}
 it follows that we have a cover of $\Sing^m(n)\cap \left(B_1^n\right)^m$ by
$$
\ll (s+1)^ms^{3mN} e^{(n+1)nmsN-(1-\frac{m+1}{s})nsN}
$$
balls of the same radius satisfying \eqref{diam}. Therefore, by the definition of Hausdorff dimension (see Definition~\ref{DefHD} and Remark~\ref{usefulHD}  immediately following it), for every $ s > s_0$  we have that

\begin{align*}
\displaystyle\dim&\left(\Sing^m(n)\cap \left(B_1^n\right)^m\right) \ \le\\[2ex]
&\displaystyle \qquad \le \  \limsup_{N\to\infty}\frac{\log\left((s+1)^ms^{3mN} e^{(n+1)nmsN-(1-\frac{m+1}{s})nsN}\right)}{-\log(e^{-(n+1)sN})}\\[2ex]
&\displaystyle \qquad  = \ \limsup_{N\to\infty}\frac{3mN\log s+\left({(n+1)nmsN-(1-\frac{m+1}{s})nsN}\right)}{{(n+1)sN}}\\[2ex]
&\displaystyle \qquad =  \  \frac{3m\log s+(n+1)nms-(1-\frac{m+1}{s})ns}{(n+1)s}\,.
\end{align*}
Letting  $s\to\infty$  gives
$$
\dim\left(\Sing^m(n)\cap \left(B_1^n\right)^m\right) \, \le  \,  \frac{(n+1)nm-n}{n+1} \, =  \,  mn-\frac{n}{n+1}   \,  ,
$$
and thereby completes the proof of Theorem~\ref{prob5+}.
\end{proof}

\begin{proof}[Proof of Theorem~\ref{prob3}]
Given $\bm f:U\to\R^n$ as in the statement of the theorem, let $$\cM_{\bm f}:=\{\bm\Xi \in \mathbb{M}_{n,2}:\bm\xi_2=\bm f(\bm\xi_1)\}  \, .
$$
Since $f\bm $ is bi-Lipschitz,
$$
\dim \big(\Sing^2_{\bm f}(n)\big)= \dim\big(\Sing^2(n)\cap \cM_{\bm f}\big)\,.
$$
Therefore, \eqref{vb100} is equivalent to
\begin{equation}\label{vb100+}
\dim\big(\Sing^2(n)\cap \cM_{\bm f}\big) \le n-\frac{n}{2(n+1)}.
\end{equation}
As in the previous proof,  it suffices to show \eqref{vb100+} for $\Sing^2(n)\cap \cM_{\bm f} \cap \big(B_1^n\big)^2$ instead of $\Sing^2(n)\cap \cM_{\bm f}$. With this in mind, by Proposition~\ref{prop3}, for any $\ve>0$ and any $s\ge1$ we have that
\begin{equation}\label{vb777++}
\Sing^2(n)\cap \cM_{\bm f}\cap \big(B_1^n\big)^2\subset \bigcup_{\bm\delta\in\Delta_s}\bigcup_{N_0=1}^\infty\bigcap_{N=N_0}^\infty Z_2(\ve,N, s, \bm\delta)\cap \cM_{\bm f}\,.
\end{equation}
Observe that $$\max\{\delta_1,\delta_2\}\ge\tfrac12-\tfrac{3}{2s}  \, ,  $$ and so  by Theorem~\ref{thm4}, for $s > s_0$ and each $\bm\delta\in\Delta_s$, we have a cover of $Z_2(\ve,N, s, \bm\delta)\cap \cM_{\bm f}$ by
$$
\min_{1\le j\le 2} Cs^{3N} e^{(n+1-\delta_j)nsN}\le Cs^{3N} e^{\left(n+1-\tfrac12+\tfrac{3}{2s}\right)nsN}
$$
balls of the same radius
\begin{equation}\label{diam+}
  r =  e^{-(n+1)sN}\,.
\end{equation}
Thus, in view of Proposition~~\ref{prop3} and the trivial fact that
\begin{equation*}
\#\Delta_s\le (s+1)^{2}\, ,
\end{equation*}
it follows that we have a cover of $\Sing^2(n)\cap \cM_{\bm f}\cap \left(B_1^n\right)^2$ by
$$
\ll (s+1)^2s^{3N} e^{\left(n+1-\tfrac12+\tfrac{3}{2s}\right)nsN}
$$
balls of the same radius $r$ as given by \eqref{diam+}. Therefore, for every $ s > s_0$ we have that
\begin{align*}
\displaystyle\dim&\left(\Sing^2(n)\cap \cM_{\bm f}\cap \left(B_1^n\right)^2\right) \ \le\\[2ex]
&\displaystyle   \qquad \le \ \limsup_{N\to\infty}\frac{\log\left((s+1)^2s^{3N} e^{\left(n+1-\tfrac12+\tfrac{3}{2s}\right)nsN}\right)}{-\log(e^{-(n+1)sN})}\\[2ex]
&\displaystyle  \qquad = \ \limsup_{N\to\infty}\frac{3N\log s+\left(n+1-\tfrac12+\tfrac{3}{2s}\right)nsN}{{(n+1)sN}}\\[2ex]
&\displaystyle  \qquad = \ \frac{3\log s+\left(n+1-\tfrac12+\tfrac{3}{2s}\right)ns}{(n+1)s}\,.
\end{align*}
On letting  $s\to\infty$,  gives
$$
\dim\left(\Sing^2(n)\cap \cM_{\bm f}\cap \left(B_1^n\right)^2\right)\le \frac{\left(n+1-\tfrac12\right)n}{n+1} = n-\frac{n}{2(n+1)}  \, ,
$$
and thereby completes the proof of Theorem~\ref{prob3}.
\end{proof}

\bigskip

As mentioned at the start of this subsection, even if we increased the number of users in the basic setup of Examples~1~$\&$~2 we would still only need to call upon the general Diophantine approximation theory described above for a singular linear form (i.e., $m=1$).
A natural question that a reader may well be  asking at this point is, whether or not there is a  model  of a communication channel that in its analysis requires us to genuinely exploit the general systems of linear forms theory with $m > 1$?
The answer to this is emphatically yes.  The simplest setup that demonstrates this involves $n$ users and one receiver equipped with $m$ antennae.  Recall,  an antenna is a device (such as an old fashioned radio or television ariel) that is used to transmit or receive signals.  Within Examples~1 $\&$ 2, each  transmitter and receiver are implicitly understood to have a single antenna. This convention is pretty standard whenever the number of antennae at a transmitter or receiver is not specified.
For a single receiver to be equipped with $m$ antennae is in essence equivalent to  $m$ receivers (each with a single antenna)  in cahoots  with one another.   The overall effect of sharing information  is an increase in the probability that the receivers will be able to decode the transmitted messages.  We now briefly  explain how the setup alluded to above naturally brings into play the general Diophantine approximation  theory for  systems of linear forms.

\noindent {\bf Example 2A (multi-antennae receivers).} Suppose there are $n $ users $S_1,\dots,S_n$ and two receivers $R_1$ and $R_2$ which `cooperate' with one another. Furthermore, assume that $n \ge 3$.  Let $Q~\ge~1$ be an integer and suppose $S_j$  wishes to  transmit the  message $u_j \in\{0,\dots,Q\}$ simultaneously to   $R_1$  and  $R_2$.  Next, as in Example~2, for $i=1,2$ and $j=1,\dots,n$,  let $h_{ij}$ denote the channel coefficients associated with the transmission of  signals from $S_j$ to $R_i$. Also, let $y_i$ denote the signal received by  $R_i $  after (linear) encoding but before noise $z_i$ is taken into account.  Thus,
\begin{eqnarray}\label{eqn23FR}
y_1  & = & \lambda\sum_{j=1}^n h_{1j} \alpha_ju_j\,,  \\ [1ex]
y_2  & = & \lambda\sum_{j=1}^n h_{2j} \alpha_ju_j\,.  \label{eqn24FR} \,
\end{eqnarray}
where $\lambda,\alpha_1,\dots,\alpha_n$ are some positive real numbers.
Now let $\dminVBi$ the minimal distance between the $(Q+1)^n$ potential outcomes of $y_i$. Now,  the larger the  minimal distance $\dminVBi$ $(i=1,2)$ the greater the tolerance for noise and thus the more likely  the receivers $R_i$  are able to recover the messages $u_{1},\dots,u_n$ by rounding $y'_i = y_i + z_i $  to the closest  possible outcome of $y_i$ (given by \eqref{eqn24FR}). Thus, it is imperative to understand how $\dminVBi$ can be bounded below. Since $R_1$ and $R_2$ are sharing information  (in fact it is better than that, they are actually the same person but they are not aware of it!),  it is only necessary that at least one of  $\dminVBone$ or  $\dminVBtwo$  is relatively large  compared to the noise. In other words, we need that the points $(y_1,y_2)\in\R^2$ are sufficiently separated.
In order to analysis this, we first apply the inverse to the linear transformation
$$
L:=\left(\begin{array}{cc}
        h_{11}\alpha_1 & h_{12}\alpha_2 \\
        h_{21}\alpha_1 & h_{22}\alpha_2
      \end{array}
\right)
$$
to $(y_1,y_2)^t$.  Without loss of generality,  we can assume that the matrix norm of $ L$ and its inverse $L^{-1}$ are bounded above. Therefore, the separation between the points $(y_1,y_2) \in \R^2 $ is comparable to the separation between the points $(\tilde y_1,\tilde y_2)  \in \R^2$,  where
$$
(\tilde y_1,\tilde y_2)^t:=L(y_1,y_2)^t   \, .
$$
Let $(\bm\xi_1, \bm\xi_2) \in \R^{n-2} \times \R^{n-2}$  be the pair corresponding  to the two columns vectors of the  matrix
$$
\bm\Xi:=\left(L^{-1}\left(\begin{array}{ccc}
        h_{13}\alpha_3 & \dots & h_{1n}\alpha_n \\
        h_{23}\alpha_3 & \dots& h_{2n}\alpha_n
      \end{array}
\right)\right)^t   \, .
$$
 The upshot, after a little manipulation, is that analysing the separation of the  points $(y_1,y_2)\in\R^2$ equates to understanding the quantity
$$
\max  \{  |\vvv q\bm\xi_{1}+p_1|,  \vvv q\bm\xi_{2}+p_2|   \}
$$
for $(\vvv p,\vvv q)\in \Z^2 \times \Z^{n-2}$ with $1\le |\vvv q|\le Q$.  In particular,  asking for good separation equates to obtaining  good lower  bounds on the quantity in question.  In turn,  this naturally brings into play the general Diophantine approximation theory for  systems of $2$ linear forms in $n-2$ real variables.  Note that assuming  the number $n$ of users is strictly greater than two (the number of cooperating receivers) simply avoids the degenerate case.   For further details of the setup just described and its more sophisticated  variants,  we refer the reader to \cite[Example~1]{LayeredMotahari} and \cite[Section~3.2]{JafarBook} and references within.

\section{A `child' example and Diophantine approximation on manifolds}\label{sec3}

The theory of \emph{Diophantine approximation on manifolds} (as coined by Bernik \& Dodson in their Cambridge Tract \cite{MR1727177}) or \emph{Diophantine approximation of dependent quantities} (as coined by Sprind\v zuk in his monograph \cite{MR548467}) refers to the study of Diophantine
properties of points in $\R^n$ whose coordinates are confined by
functional relations or equivalently are restricted to a submanifold
$\cM$ of $\R^n$.  In this section we consider an example of a communication channel which brings to the forefront the role of the theory of Diophantine approximation on manifolds in wireless communication.

\begin{remark}
The reader may well argue that in our analysis of the wireless communication model considered in Example~2, we have already touched upon the theory of Diophantine approximation on manifolds. Indeed, as pointed out on several occasions (see in particular Remarks~\ref{badbad} and \ref{kgkg}),  the points of interest $\bm\xi=(\xi_1, \xi_2) $ and $\bm\xi'=(\xi'_1, \xi'_2)$  associated  with the example are functionally dependent.  The explicit dependency is given by \eqref{eqn50} and \eqref{eqn51}.  However, it is important to stress that the actual coordinates of each of these points are not subject to any dependency and so are not restricted to a sub-manifold of $\R^2$.  The upshot of this is that we can analyse the points independently using the standard single linear form theory of Diophantine approximation in $\R^n$.  In other words, the analysis within Example~2 does not require us to exploit the theory of Diophantine approximation on manifolds.
\end{remark}

\subsection{Example 3}

In this example we will consider a model that involves several ``transmitter-receiver'' pairs who simultaneously communicate using shared communication channels. For the sake of simplicity we will concentrate on the case of three  transmitter-receiver pairs; that is,  we suppose that there are three users $S_1$, $S_2$ and $S_3$ and there are also three receivers $R_1$, $R_2$ and $R_3$.  Let $Q  \ge 1$ be an integer and suppose for each $j=1,2,3$  the user $S_j$ wishes to send a message $u_{j}\in\{0,\dots,Q\}$ to receiver $R_j$.    After (linear) encoding, $S_j$ transmits
\begin{equation}\label{eqnnew22}
x_j:= \lambda\alpha_{j}u_{j}
\end{equation}
where $\alpha_{j}$ is a positive real number and $\lambda \ge 1 $ is a scaling factor.
Note that apart form the obvious extra user $S_3$ and receiver $R_3$,  the current setup is significantly different to that of Example~2 in that  $S_j$ does not wish to send  independent messages to the receivers $R_i$ ($i \neq j$).  In other words,  we are not considering  a three-user X-channel and thus, unlike Example~2,  the codeword of user $S_j$ does not have any component intended for any other receiver but $R_j$. Nevertheless, since the communication channel  is being shared, as in Example~2,  the signal $x_j$ transmitted by $S_j$ is being received by every receiver $R_i$ with appropriate channel coefficients and thereby causing interference. Formally, for $i,j=1,2,3$ let $h_{ij}$ denote the channel coefficients associated with the transmission of  signals from $S_j$ to $R_i$. Also, let $y_i$ denote the signal received by  $R_i $  before noise is taken into account.  Thus,
\begin{equation}\label{eqnnew25}
y_i   =  \sum_{j=1}^3h_{ij} x_j~\stackrel{\eqref{eqnnew22}}{=}~\lambda\sum_{j=1}^3h_{ij}\alpha_{j}u_{j}\,.
\end{equation}
Now as usual, let us bring noise into the setup.
If $z_i$ denotes the (additive) noise at receiver $R_i$ ($i=1,2,3$), then instead of \eqref{eqnnew25}, $R_i$ receives the signal
\begin{equation}\label{eqnnew27}
y'_i   =  y_i+z_i\,.
\end{equation}
Equations \eqref{eqnnew25} and \eqref{eqnnew27} represent one the simplest models of what is known as a {\em Gaussian Interference Channel} (GIC).
The ultimate goal is for the receivers $R_i$ $(i=1,2,3)$ to decode the messages $u_{i}$ from the observation of $y'_i$. This is attainable if $2|z_i|$ is smaller than the minimal distance between the outcomes of $y_i$ given by \eqref{eqnnew25}, which will be denoted by $\dminVBi$. As before, given that the nature of noise is often a random variable with normal distribution, the overarching goal  is to ensure the probability that $|z_i|<\tfrac12\dminVBi$ is large. Indeed, as in Examples~1~$\&$~2,  the larger the probability the more likely the receivers $R_i$ $(i=1,2,3)$ are able to recover messages by rounding $ y'_i $ (given by \eqref{eqnnew27}) to the closest  possible outcome of $y_i$ (given by \eqref{eqnnew25}). Thus, as in previous examples it is imperative to understand how $\dminVBi$ can be bounded below.   Note that there are potentially $(Q+1)^3$ distinct outcomes of $y_i$ and that
\begin{equation}\label{ub1}
0\le y_i\ll \lambda Q\qquad(1\le i\le 3),
\end{equation}
where the implicit implied constants depend on the maximum of the channel coefficients $h_{ij}$ and the  encoding coefficients
$\alpha_{j}$.
It is then easily verified, based on  the outcomes of $y_i$ given by \eqref{eqnnew25} being equally spaced, that  the minimal distance satisfies the following inequality
\begin{equation}\label{ub2}
  \dminVBi\ll\frac{\lambda}{Q^2}\qquad(1\le i\le 3)\,.
\end{equation}
Ideally, we would like to obtain lower bounds for  $\dminVBi$ that are both ``close'' to this ``theoretic'' upper bound and are valid for a large class of possible choices of channel coefficients.     Before we embark on the discussion of tools from Diophantine approximation that can be used for this purpose, we discuss how the idea of interference alignment introduced in the context of Example~2 extends to the setup of Example~3. This will naturally bring the theory of Diophantine approximation on manifolds into play.

Assume for the moment that $u_{j}\in \{ 0,1\} $ and for the ease of discussion, let us just concentrate on the signal $y_1$ received  at  $R_1$. Then there are generally up to $2^3=8$ different outcomes for $y_1$. However,
receiver $R_1$  is not interested in the signals $ u_{2}$ and $u_3$. So if these signals could be deliberately aligned (at the transmitters) via encoding into a single component, then there would be fewer possible outcomes for $y_1$. Clearly, such an alignment would require that the ratio $h_{12}\alpha_2/h_{13}\alpha_3$ is a rational number. For example, if this ratio is equal to one, that is $h_{12}\alpha_2=h_{13}\alpha_3$, then
$$
y_1=\lambda\Big(h_{11}u_1+h_{12}\alpha_{2}(u_{2}+u_3)\Big)\,.
$$
Clearly, in this case the number of distinct outcomes of $y_1$ is reduced from $8$ to $6$, since there are 4 different pairs $(u_{2},u_{3})$ as opposed to 3 different sums $u_{2}+u_{3}$ when $u_{j}$ take on  binary values.
Let us call the scenario described above {\em a perfect alignment}. For the received signals to be perfectly aligned at each receiver would require imposing highly restrictive constraints on the channel coefficients, which in practice would never be realised. Indeed, an encoding realising perfect alignment simultaneously at each receiver  would   necessarily  have that the following three ratios
\begin{align*}
\frac{h_{12}\alpha_2}{h_{13}\alpha_3}\,,\qquad
\frac{h_{21}\alpha_1}{h_{23}\alpha_3}\,,\qquad
\frac{h_{31}\alpha_1}{h_{32}\alpha_2}
\end{align*}
are all  rational numbers. For example, if all these ratios are equal to one  then we have that
$$
\det\left(\begin{array}{ccc}
0 & h_{12} & -h_{13}\\
h_{21} & 0 & - h_{23}\\
h_{31} & -h_{32} & 0
          \end{array}
\right)=0\,,
$$
or equivalently, that
$$
h_{12}h_{23}h_{31}=h_{32}h_{21}h_{13} \, .
$$
In reality, for the channel coefficients to satisfy this equality  would be so extraordinary  that  it is not worth considering. The upshot is that perfect alignment is simply not feasible.

Motahari et al \cite{MR3245356} proposed a scheme based on the method introduced by Cadambe et al  \cite{MR2451005}, which  simultaneously at each receiver realises a {\em partial  alignment} that is effectively arbitrarily close to perfect alignment. The basic idea is to split the messages $u_j$ into `blocks' and apply different linear encodings to each `block'. As it happens, there is a choice of encodings that allows for all but a few of the received `blocks' to  be appropriately aligned as each receiver. On increasing the number of blocks one can approach perfect alignment with arbitrary accuracy.   We now provide the details of the alluded scheme within  the context of Example~3.  Recall,  the user $S_j$ ($j=1,2,3$) wishes to send a message $u_{j}\in\{0,\dots,Q\}$ to receiver $R_j $.   In the first instance, given an integer $B \ge 2$ we let
$$
u_{j,\vvv s}\in\{0,\dots,B-1\}
$$
be a collection of `blocks' that determine (up to order) the coefficients in the base $B$ expansion of $u_j$.
Here and throughout, for $m, k \in \N$
$$
\vvv s=(s_1,\dots,s_m)\in\cS_k:=\{0,\dots,k-1\}^m
$$
is a multi-index which is used to enumerate the blocks -- in a moment  we will take $m=6$.
Clearly, the number of different blocks (i.e. digits available to us when considering the base $B$ expansion of a number) is equal to $$M:=k^m$$ and so the size of the message $u_j$ that $S_j$ can send to  $R_j $ is bounded above by  $B^M-1$.  Without loss of generality, we can assume that
\begin{equation} \label{QBM}  Q = B^M-1  \, .
\end{equation}
Now, instead of transmitting \eqref{eqnnew22}, after  encoding $S_j$ transmits the message
\begin{equation}\label{xij}
x_{j}=\lambda\sum_{\vvv s\in\cS_k}\vvv T^{\vvv s}u_{j,\vvv s}\, .
\end{equation}
Here and throughout, for $ \vvv s \in\cS_k$
\begin{equation}\label{monom}
\vvv T^{\vvv s}:=T_{1}^{s_1}\cdots T_{m}^{s_m}
\end{equation}
are real parameters called {\em  transmit directions} obtained from a fixed finite set
$$
\vvv T:=\{T_{1},\dots,T_{m}\}
$$
of positive real numbers, called {\em generators}. As we shall soon see, the  generators will be  determined by the channel coefficients.  In short, they play the role the positive real numbers $\alpha_{j}$ appearing in  the  encoding leading  to \eqref{eqnnew22}. It is worth highlighting  that the (linear) encoding leading to \eqref{xij}  varies from block to block.  It follows that with this more sophisticated `block' setup, instead of \eqref{eqnnew25}, the signal received by $R_i$ before noise is taken into account is given by
\begin{align}
\nonumber y_i   & =  \sum_{j=1}^3h_{ij}x_j ~\stackrel{\eqref{xij}}{=}~\lambda\sum_{j=1}^3h_{ij}\underbrace{\sum_{\vvv s\in\cS_k}\vvv T^{\vvv s}u_{j,\vvv s}}_{x_j}\\
&=\lambda\left(\rule{0ex}{5ex}\right.\underbrace{\sum_{\vvv s\in\cS_k}h_{ii}\vvv T^{\vvv s}u_{i,\vvv s}}_{\text{wanted at $R_i$}} ~+~ \underbrace{\sum_{\substack{j=1\\ j\neq i}}^3\sum_{\vvv s\in\cS_k}h_{ij}\vvv T^{\vvv s}u_{j,\vvv s}}_{\text{unwanted at $R_i$}}\left.\rule{0ex}{5ex}\right)\,.\label{eqnnew30}
\end{align}
Thus, the unwanted message blocks $u_{j,\vvv s}$ from $S_j$ ($j\neq i$) arrive at $R_i$ with the transmit directions $\vvv T^{\vvv s}$ multiplied by two possible channel coefficients $h_{ij}$. It follows that the unwanted blocks appearing in \eqref{eqnnew30} constitute a  linear form with $2M=2 k^m $ terms.   We now choose the generators in such a way so as to align some of these unwanted blocks  with the net effect of reducing the number of terms in the linear form.  With this in mind, define the set of generators to be the collection of all channel  coefficient with $ i \neq j$; namely
\begin{equation}\label{specialT}
\vvv T=\{h_{12},h_{13},h_{21},h_{23},h_{31},h_{32}\}\,.
\end{equation}
Thus, $m=6$ with respect to the general description above.  With this choice of generators, it follows that the unwanted  part within \eqref{eqnnew30} can now  be written as
\begin{equation}\label{eq103}
\sum_{\vvv s\in\cS_{k+1}}\vvv T^{\vvv s}v_{i,\vvv s}   \,
\end{equation}
where the terms $$v_{i,\vvv s} \in\{0,\dots,2B-2\}$$ are integers formed as sums of up to two blocks $u_{j,\vvv s}$.
Note that the coefficients of $v_{i,\vvv s}$ are monomials in the generators given by \eqref{specialT}. Due to the multiplication by $h_{ij}$ in \eqref{eqnnew30} the exponents in the monomials appearing in \eqref{eq103} are up to $k$ rather than just $k-1$.  This explains why the summation in \eqref{eq103} is taken over $\cS_{k+1}$ rather than just  $\cS_k$.  The upshot of choosing $\vvv T$ as in \eqref{specialT} is that the `unwanted' linear form of $ 2M=2k^6 $ terms appearing in \eqref{eqnnew30} is replaced by a linear form given by \eqref{eq103} of $(k+1)^6=M(1+1/k)^6$ terms. In other words,  asymptotically (as $k$ increases) we have halved the number of terms associated with unwanted message blocks.  On substituting \eqref{eq103} into \eqref{eqnnew30} we get that
\begin{align}
y_i =\lambda\left(\rule{0ex}{5ex}\right.\underbrace{\sum_{\vvv s\in\cS_k}h_{ii}\vvv T^{\vvv s}u_{i,\vvv s}}_{\text{wanted at $R_i$}} ~+~ \underbrace{\sum_{\vvv s\in\cS_{k+1}}\vvv T^{\vvv s}v_{i,\vvv s}}_{\text{unwanted at $R_i$}}\left.\rule{0ex}{5ex}\right)\,.\label{eqnnew30+}
\end{align}
Thus, $y_i$ is a linear form of $$M':=k^6 +(k+1)^6$$ terms\footnote{Observe that essentially half of the terms in \eqref{eqnnew30+} are wanted at $R_i$  compared to only a third (before alignment) in  \eqref{eqnnew30} or indeed in \eqref{ub1}.}.   Up to the factor $\lambda$, the coefficients of the integers $ u_{i,\vvv s}$ and $ v_{i,\vvv s} $ in \eqref{eqnnew30+} are monomials in the six generators of $ \vvv T$ and are all different. It is convenient to  represent  these coefficients as a `coefficient' vector
\begin{equation}\label{G_i}
\vvv G_i:=(G_{i,0},G_{i,1},\dots,G_{i,n})   \quad {\rm where } \quad  n:=M'-1 \, .
\end{equation}
To reiterate, the components $G_{i,0},G_{i,1},\dots,G_{i,n}$ are the real numbers
\begin{equation}\label{eqn105}
\vvv T^{\vvv s}~\text{ with  \ }\vvv s\in\cS_{k+1} \quad \text{and}\quad h_{ii}\vvv T^{\vvv s} ~\text{ with  \ }\vvv s\in\cS_k
\end{equation}
written in any fixed order.
It is easily verified that for any $\epsilon >0$, for $ k$ sufficiently large
\begin{equation} \label{nM}
 2M <  n  < 2M + \epsilon  \, .
\end{equation}
Now let
\begin{equation}\label{newxi}
\bm\xi_i=(\xi_{i,1},\dots,\xi_{i,n}):= \Big(\frac{G_{i,1}}{G_{i,0}},\dots,\frac{G_{i,n}}{G_{i,0}}\Big)   \qquad(1\le i\le 3)\,.
\end{equation}

Returning to \eqref{eqnnew30+}, it is easily seen that there are potentially $B^{M'}$ distinct outcomes of $y_i$ and as before  (cf. \eqref{ub1})
\begin{equation}\label{ub1sv}
0\le y_i\ll 2 \lambda B\qquad(1\le i\le 3),
\end{equation}
where the implicit implied constants depend on the maximum of the channel coefficients $h_{ij}$ and the integer $k$.  Now let $\dminVBi$ denote the minimal distance between the outcomes of $y_i$ given by \eqref{eqnnew30+}.
It is then easily verified, based on  these  outcomes being equally spaced, that  the minimal distance satisfies the following inequality (cf. \eqref{ub2})
\begin{equation}\label{ub2sv}
  \dminVBi\ll\frac{\lambda B}{B^{M'}} =  \frac{\lambda}{B^{n}}  \, \le  \,  \frac{\lambda}{Q^{2}}   \qquad(1\le i\le 3)\,.
\end{equation}
The last inequality makes use of \eqref{QBM} and \eqref{nM}.
Recall, that our goal is the same as in all previous examples.   We wish to obtain lower bounds for  $\dminVBi$ that are both ``close'' to this ``theoretic'' upper bound and at the same time are  valid for a large class of possible choices of channel coefficients.  As we have seen in Examples~1~$\&$~2, the goal is intimately related to the Diophantine properties of certain points defined via the channel coefficients.  Within the context of Example~3, the points of interest are  precisely those corresponding to $\bm\xi_i \in \R^n$ as given by \eqref{newxi}. In \S\ref{IloveMarking}, we will  demonstrate that this is indeed the case by calculating the  DoF of the three-user Gaussian Interference Channel (GIC).   First we make an important observation: {\em the coordinates of each point $\bm\xi_i$   ($ i =1,2,3 $) are functions of seven variables and are therefore dependent}.  The latter follows since $ k \ge 1 $ and so by definition $n \ge 2^6 > 7 $.
The fact that the point $\bm\xi_i$ of interest is of dependent variables implies that $\bm\xi_i$ lies on a submanifold $\cM$ of $\R^n$ of dimension strictly smaller that $n$.    Trivially, since the dimension of $\cM$ is strictly less than $n$, we have that the $n$-dimension Lebesgue measure of $\cM$ is zero. The upshot of the dependency is that all the measure theoretic Diophantine approximation results (such as those concerning badly approximable, $\psi$-approximable, Dirichlet improvable, singular, etc etc)  that we have exploited so far in our analysis of Examples~1~$\&$~2 are pretty much redundant.  We need a theory which takes into account that the points of interest lie on a submanifold $\cM$ of $\R^n$.  Luckily, today the metric theory of Diophantine approximation on manifolds is in reasonable shape.  Indeed, for a large class of so called non-degenerate manifolds there exists
\begin{itemize}
  \item[{\rm(i)}]
  \,\,\,\,a rich  badly approximable theory concerning $\Bad(n) \cap \cM $ -- see for example \cite{MR3733884,MR3231023,MR3425389, BNYV,BNY1,MR3990196} and references within,
  \item[{\rm(ii)}] \,\,\,\,a rich  $\psi$-approximable theory concerning $\cW_{n}(\psi)  \cap \cM $ -- see for example \cite{MR3545930, MR1905790, MR1944505, MR1829381, MR3980277, MR2156655, MR2805874, MR2812651, MR3310488, MR1982150, MR1652916} and references within,  and
  \item[{\rm(iii)}] \,\,\,\,a rich  Dirichlet improvable theory concerning $\DI(n) \cap \cM $  -- see for example \cite{MR3975500,MR2366229,MR2534098} and references within.
\end{itemize}
For a general overview of the manifold theory we refer the reader to \cite[Section~6]{MR3618787}.   In short, the recent state of the art results for the sets just listed suffice to implement  the approaches taken in \S\ref{BAD20} to \S\ref{eyesight} within the context of Example~3.  As already mentioned,  we will shortly provide the details of how the `Khintchine-Groshev' approach of \S\ref{sv2.4} translates to the current setup.

Observe  that in above list of Diophantine sets restricted to $\cM$  there is a notable exception.  We have not mentioned  singular (resp. jointly singular) sets $\Sing(n)$  (resp. $\Sing^2(n)$)  and in turn  we have avoided mentioning  the approach  taken in \S\ref{singSV} that enables us to improve the result of Motahari et al on the DoF of a two-user X-channel.  The reason for this is simple --  our current knowledge of  $\Sing(n)  \cap \cM $ is not sufficient. We will come back to this in \S\ref{open}.

\subsection{The Khintchine-Groshev theorem for manifolds and DoF} \label{IloveMarking}
   The goal of this section is twofold.    The first is to introduce the analogue of the Khintchine-Groshev Theorem for one linear form (i.e. Theorem~\ref{KGT}  in \S\ref{sv2.4}) in  which the points of interest are restricted to a submanifold of $\R^n$.    The second is to exploit this so called Khintchine-Groshev theorem for manifolds to calculate the DoF of the three-user GIC considered in Example~3.

Let   $\cM$ be a submanifold of  $\R^n$  and  let $\cW_{n}(\psi)$ be  the set
of $\psi$-approximable points in $\R^n$ defined by \eqref{Wn}.   In short, if the manifold  is ``sufficiently'' curved the   Khintchine-Groshev theorem for manifolds provides a `zero-one' criterion for the Lebesgue measure of the set
$$
\cW_n(\psi)\cap\cM   \, .
$$
Observe that if the dimension of  the manifold is strictly less than $n$,
then with respect to  $n$-dimensional Lebesgue measure we trivially have that $|\cW_n(\psi)\cap\cM|_n=0$ irrespective of the
approximating function $\psi$. Thus, when referring to the Lebesgue
measure of the set $ \cW_n(\psi)\cap\cM $ it is always  with
reference to the induced Lebesgue measure on $\cM$. More
generally, given a subset $S$ of $\cM$ we shall write
$|S|_{\cM} $ for the  measure of $S$ with respect to the induced
Lebesgue measure on $\cM$. Without loss of generality, we will assume that $$|\cM|_\cM=1$$ since otherwise the induced measure can be re--normalized accordingly.   It is not particularly difficult to show that in order  to  obtain an analogue of Theorem~\ref{KGT} (both the convergence and divergence aspects) for $ \cW_n(\psi)\cap\cM $  we need to avoid  hyperplanes --  see  \cite[Section~4.5]{MR3618787}.  To overcome such natural counterexamples, we insist that $\cM$ is a  {\em non--degenerate} manifold.

\noindent{\em Non--degenerate manifolds. } Essentially, these are smooth
submanifolds of $\R^n$ which are sufficiently curved so as to
deviate from any hyperplane. Formally, a  manifold $\cM$ of
dimension $d$ embedded in $\R^n$ is said to be \emph{non--degenerate} if it
arises from a non--degenerate map $\vvv f: U\to \R^n$ where $U$ is an
open subset of $\R^d$ and $\cM:=\vvv f(U)$. The map $\vvv f: U\to
\R^n,\vvv x\mapsto \vvv f(\vvv x)=(f_1(\vvv x),\dots,f_n(\vvv x))$ is said to be
\emph{$l$--non--degenerate at} $\vvv x\in U$, where $l\in\N$,
if $\vvv f$ is $l$ times continuously differentiable on some
sufficiently small ball centred at $\vvv x$ and the partial derivatives
of $\vvv f$ at $\vvv x$ of orders up to $l$ span $\R^n$. The map $\vvv f$ is
\emph{non--degenerate} at $\vvv x$ if it is $l$--non--degenerate at $\vvv x$ for some $l\in\N$. The map $\vvv f$ is \emph{non--degenerate} if it is
non--degenerate at almost every (in terms of $d$--dimensional Lebesgue
measure) point $\vvv x$ in $U$; in turn the manifold $\cM=\vvv f(U)$ is also
said to be non--degenerate.
It is well known, that any
real connected analytic manifold not contained in any hyperplane of
$\R^n$ is non--degenerate at every point \cite{MR1652916}.  In the case the manifold $\cM$ is a planar curve
$\cC$, a point on $\cC$ is non-degenerate if the curvature
at that point is non-zero.  Moreover,
it is not difficult to show that the set of points on a planar curve
at which the curvature vanishes but the curve is non-degenerate is at
most countable, see \cite[Lemmas~2~\&~3]{MR1387861}. In view of this, the curvature completely describes
the non-degeneracy of planar curves. Clearly, a straight line is
degenerate everywhere.

The convergence part of the following statement was independently established in \cite{MR1905790} and \cite{MR1829381},  while the divergence part was established in \cite{MR1944505}.

\begin{theorem}[Khintchine-Groshev for manifolds]\label{thm17}
Let $\psi:\Rp\to\Rp$ be a  monotonic function and let $\cM$ be a non-degenerate submanifold of $\R^n$. Then
$$
|\cW_n(\psi)\cap\cM|_\cM=\left\{\begin{array}{cl}
0 &\text{if } \ \sum_{q=1}^\infty q^{n-1}\psi(q)<\infty\,,\\[2ex]
1 &\text{if } \ \sum_{q=1}^\infty q^{n-1}\psi(q)=\infty\,.
                       \end{array}
\right.
$$
\end{theorem}

\medskip

\begin{remark}
In view of Corollary~\ref{cor2} in \S\ref{BAD20}, it follows  that
$$
\cW_n(\psi)\cap\cM = \cM  \quad {\rm if   } \quad  \psi: q \mapsto q^{-n} \, .
$$
Now, given $\ve >0$ consider the  function $\psi_\ve: q \mapsto q^{-n-\ve }$.  A  submanifold $\cM$ of $\R^n$ is called {\em extremal} if
$$
\left| \cW_n(\psi_\ve)\cap\cM \right|_{\cM} =0\,.
$$
Sprind\v zuk  (1980) conjectured that any analytic non-degenerate
submanifold  is extremal. In their pioneering work \cite{MR1652916}, Kleinbock $\& $ Margulis proved that any non-degenerate
submanifold $\cM$ of $\R^n$ is extremal and thus established Sprind\v zuk's conjecture.   It is easy to see that this
implies the convergence case of Theorem~\ref{thm17} for functions of the
shape $\psi_\ve$.
\end{remark}

\begin{remark}
For the sake of completeness, it is worth mentioning that the externality theorem for non-degenerate submanifolds of $\R^n$ has been extended in recent years to submanifolds of $n\times m$ matrices, see \cite{MR3777412, MR3346961, MR2679461}.
\end{remark}

An immediate consequence of the convergence case of Theorem~\ref{thm17} is the following statement  (cf. Corollary~\ref{cor3}).

\begin{corollary}\label{cor6}
Let  $\psi:\Rp\to\Rp$ be a function such that
\begin{equation}\label{eqn57+}
  \sum_{q=1}^\infty q^{n-1}\psi(q)<\infty\,.
\end{equation}
Suppose that $\cM$ is as in Theorem~\ref{thm17}. Then, for almost all $\bm\xi \in \cM$ there exists a constant $\kappa (\bm\xi)  >0 $  such that
\begin{equation}\label{eqn58+}
 |q_1\xi_1+\dots+q_n\xi_n+p | \ >  \  \kappa(\bm\xi)  \,   \psi(|\vvv q|)  \qquad   \forall  \    (p,\vvv q) \in \Z \times \Z^{n}\backslash \{\vvv 0\} \,.
\end{equation}
\end{corollary}

In line with the discussion in \S\ref{sv2.4} preceding the statement of  the effective convergence Khintchine-Groshev theorem (i.e. Theorem~\ref{EKG}),
a natural question to consider is: {\em can the constant $\kappa(\bm\xi)$ within Corollary~\ref{cor6} be made independent of $\bm\xi$?} The argument involving  the set  $\cB_n(\psi,\kappa)$  given by \eqref{eqn61} can be modified to show that this is impossible to guarantee with probability one; that is, for almost all $\bm\xi \in \cM$. Nevertheless, the  following result  provides an effective solution to the above question. It is a special case of \cite[Theorem~3]{MR3545930}.

\begin{theorem}[Effective convergence Khintchine-Groshev for manifolds] \label{EKGM}
Let $l\in\N$ and let $\cM$ be a compact $d$--dimensional $C^{l+1}$ submanifold of $\R^n$ that is $l$--non--degenerate at every point. Let $\psi:\Rp\to\Rp$ be a monotonically decreasing function such that
\begin{equation} \label{def_S_Psi}
\Sigma_\psi:=\sum_{q=1} q^{n-1}\psi(q)<\infty\,.
\end{equation}
Then there exist positive constants $\kappa_0,C_1$ depending on $\psi$ and $\cM$ only and $C_0$ depending on the dimension of $\cM$ only such that for any $0<\delta<1$, the inequality
\begin{equation} \label{theorem_linear_forms_ie_statement_v2++}
|\cB_n(\psi,\kappa)\cap\cM|_\cM\ge 1-\delta
\end{equation}
holds with
\begin{equation}\label{kappavb}
\kappa:=\min\left\{\kappa_0,\ \frac{C_0 \delta}{\Sigma_\psi},\ C_1\delta^{d(n+1)(2l-1)}\right\}\,.
\end{equation}
\end{theorem}

\medskip

\begin{remark}
The constants appearing in \eqref{kappavb} are explicitly computable, see \cite[Theorem~6]{MR3545930} for such a  statement. In \cite{MR3980277} Theorem~\ref{EKGM} was also extended to a natural class of affine subspaces, which by definition are degenerate.
\end{remark}

We now move onto our second goal: to exploit the Khintchine-Groshev theorem for manifolds to calculate the DoF of the three-user GIC considered in Example~3.    The overall approach is similar to that used in \S\ref{sv2.4} to calculate the DoF of the two-user X-channel considered in Example~2.  In view of this we will keep the following exposition rather brief and refer the reader to   \S\ref{sv2.4} for both the motivation and the details.   With this in mind, let $\cM$ denote the $7$-dimensional submanifold of $\R^n$ arising from the implicit dependency within \eqref{newxi}.  In other words,  a point $\bm\xi_i \in \cM$ if and only if it is of the form \eqref{newxi}. That $\cM$ is of dimension $7$ follows from the fact that the monomials $G_{i,0},G_{i,1},\dots,G_{i,n}$ depend on $h_{ii}$ and the other $6$ channel coefficients that form the set $\vvv T$ of generators. It is also not difficult to see that these monomials are all different and therefore linearly independent over $\R$. Consequently, $1,\xi_{i,1},\dots,\xi_{i,n}$ are linearly independent over $\R$ as functions of the corresponding channel coefficients. Hence $\cM$ cannot be contained in any hyperplane of $\R^n$. Also note that $\cM$ is connected and analytic, and therefore, it is non-degenerate.

Now suppose that
\begin{equation}\label{eqn106}
\bm\xi \not\in\cW_{n}(\psi)
\end{equation}
where  $\psi: q \to q^{-n-\ve}$ for some $\ve>0$. Then,  Corollary~\ref{cor6} implies that for almost all  $\bm\xi\in\cM$ there exists a constant $\kappa(\bm\xi)>0$ such that
$$
\left|q_1\xi_1+\dots+q_n\xi_n+p\right|\ge\frac{\kappa(\bm\xi)}{|\vvv q|^{n+\ve}}
$$
for all  $ (p,\vvv q) \in \Z \times \Z^{n}\backslash \{\vvv 0\}$.
Here and throughout the rest of this section, almost all is  with respect to $7$-dimensional Lebesgue measure induced on $\cM$. In particular, it follows that for almost all $\bm\xi\in\cM$ and every $B\in\N$ we have that (cf. \eqref{eqn59})
\begin{equation}\label{eqn59sv}
\left|q_1\xi_1+\dots+q_n\xi_n+p\right|\ge\frac{\kappa(\bm\xi)}{B^{n+\ve}}
\end{equation}
for all $ (p,\vvv q) \in \Z \times \Z^{n}\backslash \{\vvv 0\}$ with $1\le|\vvv q|\le B$.    Then, the analysis as  in \S\ref{sv2.4} that leads to \eqref{eqn60}, enables us to make the following analogous statement:  with probability one, for every $B \ge 2$ and  a random choice of channel coefficients $h_{ij}$  $(i,j=1,2,3)$,  the minimum separation between the associated  points $y_i$ given by \eqref{eqnnew30+}  satisfies
\begin{equation}\label{eqn108}
\dminVBi \gg     \frac{ \lambda   \, \kappa(\bm\xi_i)  \ }{B^{n+ \ve} }   \qquad(1\le i\le 3)\,.
\end{equation}
We stress, that $ \bm\xi_i  $  corresponds to the point given by \eqref{newxi} associated with the choice of channel coefficients.  Recall, that the latter determine the set of generators \eqref{specialT} which in turn  determine the coefficient vector $\vvv G_i$ and therefore the point $ \bm\xi_i$.
Note that apart from  the extra  $\ve$ term in the power, the lower bound \eqref{eqn108} coincides (up to constants)  with the upper bound  \eqref{ub2sv}.

Now, in relation to Example 3, the power constraint $P$ on the channel model  means that
\begin{equation}\label{eqn62+}
\text{$|x_j|^2\le P$  \qquad ($j=1,2,3$)} \, ,
\end{equation}
where $ x_j$ is the codeword transmitted by $S_j$  as given by \eqref{xij}.   Now notice that since the blocks $
u_{j,\vvv s}  $  ($\vvv s  \in \cS_k$) are integers lying in $\{0,\dots,B-1\}$, it follows that
\begin{equation*}\label{eq102}
|x_j| \ll \lambda B\,,
\end{equation*}
where the implied implicit constant is independent from $B$ and $\lambda$.
Hence, we conclude that  $P$ is comparable to $(\lambda B)^2$.
It is  shown in \cite[\S5]{MR3245356}, that the probability of error in transmission within Example~3 is bounded above by
\eqref{eqn63} with
$$\dminVB=\min\{\dminVBone,\dminVBtwo,\dminVBthree\}.$$

\noindent Recall,  in order to achieve reliable transmission one requires that this probability  tends to zero as  $P\to\infty$. Then, on  assuming \eqref{eqn108} -- which holds for almost every $\bm\xi_i \in \cM$ --  it follows that
\begin{equation}\label{eqn64+}
\dminVB  \gg \frac{\lambda}{B^{n+\ve}}\,,
\end{equation}
and so the quantity \eqref{eqn63} will tend to zero as $B\to\infty$ if we set
$$\lambda=B^{n+2\ve} \, . $$
The upshot of this is that we will achieve a reliable transmission  rate under the power constraint \eqref{eqn62+} if we set $P$ to be comparable to $B^{2n+2 + 4 \ve}$; that is
$$
 B^{2n+2 + 4 \ve}  \ll P  \ll B^{2n+2 + 4 \ve} \, .
$$

\noindent Next, recall that the largest message $u_j$ that user $S_j $ can send to  $R_j $  is given by \eqref{QBM}.  Thus, it follows that the number of bits (binary digits) that user $S_j$ transmits is approximately
\begin{equation*}\label{bits}
\log B^{M}=M\log B\,.
\end{equation*}
Therefore, in total the three users $S_j$  ($j=1,2,3$)  transmit approximately
$3M \times \log B$ bits, which with our choice of $P$ is an  achievable total rate of reliable transmission; however, it may not be maximal.
On comparing this to the rate of reliable transmission for the simple  point to point channel under the same power constraint, we get that the total DoF of the three-user GIC  is at least
\begin{equation}\label{eqn65+}
\lim_{P\to\infty}\frac{3M\log B}{\frac12\log(1+P)}=\lim_{B\to\infty}\frac{3M\log B}{\frac12\log(1+B^{2n+2+4\ve})}=\frac{3M}{n+1+2\ve}   \, .
\end{equation}
Given that $\ve>0$ is arbitrary, it follows that for almost every (with respect to the $7$-dimensional Lebesgue measure) realisation of the channel coefficients
  $$
  {\rm DoF} \ge \frac{3M}{n+1} \, .
  $$
  Now recall that $n+1=M'=(k+1)^6+k^6$ and $M=k^6$.  On substituting these values into the above lower bound,  we obtain  that
  $$
  {\rm DoF} \ge \frac{3k^6}{(k+1)^6+k^6} \, .
  $$
Given that $k$ is arbitrary,   it follows (on letting  $k\to\infty$) that
for almost every realisation of the channel coefficients
  $$
  {\rm DoF} \ge \frac{3}{2} \, .
  $$
Now it was shown in \cite{MR2451005} that the DoF of a three-user GIC
is upper bounded by $3/2$ for all choices of the channel coefficients, and so it follows that for almost every realisation of the channel coefficients
\begin{equation}\label{eqn66+}
{\rm DoF}=\frac32\,.
\end{equation}

\subsection{Singular and non-singular points on manifolds}
\label{open}  With reference to Example 3, we have seen in the previous section that the Khintchine-Groshev theorem for non-degenerate manifolds allows us to achieve good separation between the received signals  $y_i$ given by \eqref{eqn108}.  More precisely, for almost all choices of the channel coefficients $h_{ij}$ $ (i,j =1,2,3)$ we obtain the lower  bounds \eqref{ub2sv} for the  minimal distances $\dminVBi$   that are only `$\ve$-weaker' than the `theoretic'' upper bounds  as given by \eqref{eqn108}.
%
As in the discussion at the start of  \S\ref{singSV}, this motivates  the   question of {\em whether good separation and indeed if the total DoF of 3/2 for the three-user GIC can be achieved for a larger class of channel coefficients?}  Concerning the latter, what we have in mind is a statement along the lines of Theorem~\ref{MotDOF2} that improves the  Motahari et al result (Theorem~\ref{MotDOF}) for the total DoF of the two-user $X$-channel.  Beyond this, but  still in a similar vein,  one can ask if the more general DoF results of Motahari et al  \cite{MR3245356} for  communications channels involving  more users and receivers can be improved?   Clearly, the approach taken in \S\ref{singSV} and \S\ref{FR}  based on the Diophantine approximation theory of non-singular and jointly non-singular points can be utilized to make the desired improvements. However there is a  snag -- we would require the existence of such a theory in which the points of interest are restricted to  non-degenerate manifolds.  Unfortunately, the analogues of Theorems~\ref{thm6}, \ref{prob1}, \ref{prob3}, \ref{prob5} and \ref{prob5+} for manifolds are not currently available.  In short, obtaining any such statement represents a  significant open problem in the theory of Diophantine approximation on manifolds.  Indeed, even partial statements such as the following  currently seem out of reach.  As we shall see, it has non-trivial implications for both number theory  and wireless communication.

\begin{problem}\label{prob3+}
Let $n\ge2$ and $\cM$ be any analytic non-degenerate  submanifold of $\R^n$  of dimension $d$. Verify if
\begin{equation}\label{e122}
\dim\big(\Sing(n)\cap\cM\big)< d := \dim\big( \cM\big)  \, .
\end{equation}
\end{problem}

\noindent Recall, that  $\Sing(n)$ is the set of  singular points in $\R^n$ - see Definition \ref{def5}~in~\S\ref{singSV}.

\begin{remark}
Determining  the actual value for the Hausdorff dimension of the set $\Sing(n)\cap\cM$ for special  classes of submanifolds $\cM$  (such as polynomial curves -- see below) would be most desirable.   It is not difficult to see that the intersection of $\cM$ with any rational hyperplane is contained in $\Sing(n)$. Therefore,
\begin{equation*}\label{vb680}
\dim\big(\Sing(n)\cap\cM\big)\ge d - 1\,.
\end{equation*}
When $d>1$,  this gives a non-trivial lower bound.   Obviously, when $d=1$ the lower bound is  trivial.
\end{remark}

From a purely number theoretic point of view,  Problem~\ref{prob3+} is of particular interest  when the manifold is a curve ($d=1$). It has a  well-known connection to the famous and notorious problem posed by Wirsing (1961) and later restated in a stronger form by Schmidt \cite[pg.~258]{Schmidt-1980}.  This we now briefly describe.  The Wirsing-Schmidt conjecture is concerned with the approximation of real numbers by algebraic numbers of bounded degree.  The proximity of the approximation is measured in terms of the height of the algebraic numbers.  Recall,  that given a polynomial $P$ with integer coefficients, the height $H(P)$ of $P$ is defined to be the maximum of the absolute values of the coefficients of $P$. In turn the height $H(\alpha)$ of an
algebraic number $\alpha$ is the height of the minimal defining polynomial $P$ of $\alpha$ over $\Z$.

\begin{conjecture}[Wirsing-Schmidt]\label{prob4+}
{\em Let $n\ge2$ and $\xi$ be any real number that is not algebraic of degree $\le n$. Then there exists a constant $C=C(n,\xi)$ and infinitely many algebraic numbers $\alpha$ of degree $\le n$, such that}
\begin{equation}\label{e123}
|\xi-\alpha|<C  \,  H(\alpha)^{-n-1}   \, .
\end{equation}
\end{conjecture}

\noindent Note that when $n=1$ the conjecture is trivially true since it coincides with the classical corollary to  Dirichlet's theorem  -- the first theorem stated in this chapter.  For $n=2$ the conjecture was proved  by Davenport $\&$ Schmidt (1967). For $n\ge3$ there are only partial results. For recent progress and an overview of previous results we refer the reader to \cite{BadzSchl} and references within.

The connection between the Wirsing-Schmidt conjecture and Problem~\ref{prob3+}  comes about via the  well know fact that the former is intimately related to singular points on the Veronese curves $\cV_n:=\{ (\xi,\xi^2,\dots,\xi^n) : \xi \in \R \} $.

\begin{lemma}
Let $n \ge 2$ and $\xi \in \R$.  If $(\xi,\xi^2,\dots,\xi^n)\not\in\Sing(n)$, then the Wirsing-Schmidt conjecture holds for $\xi$.
\end{lemma}

\noindent The proof of the lemma is pretty standard. For example, it easily follows by adapting the argument appearing in \cite[Appendix~B]{MR3425389} in an obvious manner.  A straightforward consequence of the lemma is that any  upper bound for $\dim\big(\Sing(n)\cap\cV\big)$ gives an upper bound on the dimension of the set of potential counterexamples to the Wirsing-Schmidt conjecture. When $n \geq 3$, currently  we do not even know that the set of potential counterexamples has dimension strictly less than one - the trivial bound.  Clearly, progress on Problem~\ref{prob3+}  with  $\cM = \cV_n$ would rectify this gaping hole in our knowledge.

We now turn our attention to the question raised  at the start of this subsection; namely,  whether good separation and the total DoF of 3/2 within the setup of Example~3 can be achieved for a larger class of channel coefficients?
To start with we recall that the $7$-dimensional submanifold  $\cM$ of $\R^n$ arising from the implicit dependency within \eqref{newxi} is both analytic and  non-degenerate.  Thus it falls under the umbrella of  Problem~\ref{prob3+}. In turn, on naturally adapting the argument used to establish Proposition~\ref{ohyes2}, a consequence of the upper bound  \eqref{e122} is the following statement: for all choice of channel coefficients $\{h_{ii}, h_{12},h_{13},h_{21},h_{23},h_{31},h_{32}\} $  $(i=1,2,3)$ except on a subset of strictly positive codimension,   the minimum separation $\dminVBi$ between the associated  points $y_i$ given by \eqref{eqnnew30+}  satisfies \eqref{eqn108}.
The upshot is that if true, Problem~\ref{prob3+} enables us to obtain good separation for a larger class of channel coefficients than the (unconditional) Khintchine-Groshev approach outlined in  \S\ref{IloveMarking}.

As we have seen within the setup of Example~2, in order to improve  the `almost all'  DoF result (Theorem~\ref{MotDOF}) of Motahari et al   we  need to work with the jointly singular set  $\Sing^2_{\bm f}(n) $ appearing in Theorem~\ref{prob3}.  This theorem  provides  a non-trivial upper bound for the Hausdorff dimension of such sets  and  is the key to establishing the stronger DoF statement Theorem~\ref{MotDOF2}.  With this in mind, we suspect that progress on the following problem is at the heart of improving the `almost all'  DoF result  for the three-user GIC  (see \eqref{eqn66+}) obtained via the Khintchine-Groshev approach.  In any case, we believe that the problem is of interest in its own right.  Recall, that $\Sing^m(n)$  is given by \eqref{tired} and is  the  jointly singular set for  systems of linear forms.

\begin{problem}\label{prob5++}
Let $k,\ell,m,d\in\N$, $n=k+\ell$, $U\subset\R^d$ and $V\subset\R^m$ be open subsets. Suppose that
$\bm f:U\to\R^k$ and $\bm g:U\to\R^\ell$ are polynomial non-degenerate maps. For each $\vvv u\in U$ and $\vvv v\in V$ let
$\Xi(\vvv u,\vvv v)$ be the matrix with columns $(v_i\bm f(\vvv u),\bm g(\vvv u))^t$ and let
$$
\Sing^m_{\bm f,\bm g}(n):=\Big\{(\vvv u,\vvv v)\in U\times V: \Xi(\vvv u,\vvv v)\in\Sing^m(n)\Big\}\,.
$$
Verify if
$$
\dim\left(\Sing^m_{\bm f,\bm g}(n)\right)<d+m\,.
$$
\end{problem}

Of course, it would be natural to generalise  the problem by  replacing `polynomial' with `analytic' and  by widening  the scope of the  $n \times m$  matrices under consideration.  On another front,  staying within the setup of Problem~\ref{prob5++}, it would be highly desirable to  determine  the actual value for the Hausdorff dimension of the set $\Sing^m_{\bm f,\bm g}(n)$. This represents a major challenge.

\vspace*{2ex}

\noindent{\em Acknowledgements.} The authors are grateful to Anish Ghosh and Cong Ling for their valuable comments on an earlier version of this chapter.  We would also like to thank
Mohammad Ali Maddah-Ali for bringing  \cite{NM13} to our attention  (see Remark~\ref{mate}).



\begin{thebibliography}{10}

\bibitem{MR3545930}
F.~Adiceam, V.~Beresnevich, J.~Levesley, S.~Velani, and E.~Zorin.
\newblock Diophantine approximation and applications in interference alignment.
\newblock {\em Adv. Math.}, 302:231--279, 2016.

\bibitem{MR3777412}
M. Aka, E. Breuillard, L. Rosenzweig  and N. de~Saxc\'e.
\newblock Diophantine approximation on matrices and {L}ie groups.
\newblock {\em Geom. Funct. Anal.}, 28(1):1--57, 2018.

\bibitem{MR3733884}
Jinpeng An, Victor Beresnevich, and Sanju Velani.
\newblock Badly approximable points on planar curves and winning.
\newblock {\em Adv. Math.}, 324:148--202, 2018.

\bibitem{BadzSchl}
Dzmitry Badziahin and Johannes Schleischitz.
\newblock An improved bound in {W}irsing's problem.
\newblock \url{https://arxiv.org/abs/1912.09013}, 2019.

\bibitem{MR3231023}
Dzmitry Badziahin and Sanju Velani.
\newblock Badly approximable points on planar curves and a problem of
  {D}avenport.
\newblock {\em Math. Ann.}, 359(3-4):969--1023, 2014.

\bibitem{MR1905790}
V.~Beresnevich.
\newblock A {G}roshev type theorem for convergence on manifolds.
\newblock {\em Acta Math. Hungar.}, 94(1-2):99--130, 2002.

\bibitem{MR1387861}
V.~Beresnevich and V. Bernik.
\newblock On a metrical theorem of {W}. {S}chmidt.
\newblock {\em Acta Arith.}, 75(3):219--233, 1996.

\bibitem{MR1944505}
V.~V. Beresnevich, V.~I. Bernik, D.~Y. Kleinbock, and G.~A. Margulis.
\newblock Metric {D}iophantine approximation: the {K}hintchine-{G}roshev
  theorem for nondegenerate manifolds.
\newblock {\em Mosc. Math. J.}, 2(2):203--225, 2002.
\newblock Dedicated to Yuri I. Manin on the occasion of his 65th birthday.

\bibitem{MR3425389}
Victor Beresnevich.
\newblock Badly approximable points on manifolds.
\newblock {\em Invent. Math.}, 202(3):1199--1240, 2015.

\bibitem{MR2184760}
Victor Beresnevich, Detta Dickinson, and Sanju Velani.
\newblock Measure theoretic laws for lim sup sets.
\newblock {\em Mem. Amer. Math. Soc.}, 179(846):x+91, 2006.


\bibitem{MR3346961}
Victor Beresnevich, Dmitry Kleinbock and Gregory Margulis.
\newblock Non-planarity and metric {D}iophantine approximation for systems of linear forms.
\newblock {\em J. Th\'{e}or. Nombres Bordeaux}, 27(1):1--31, 2015.


\bibitem{BNYV}
Victor Beresnevich, Erez Nesharim, Sanju Velani, and Lei Yang.
\newblock Schmidt's conjecture and badly approximable matrices.
\newblock In preparation.

\bibitem{BNY1}
Victor Beresnevich, Erez Nesharim, and Lei Yang.
\newblock Winning property of badly approximable points on curves.
\newblock \url{https://arxiv.org/abs/2005.02128}, 2020.

\bibitem{MR3618787}
Victor Beresnevich, Felipe Ram\'{\i}rez, and Sanju Velani.
\newblock Metric {D}iophantine approximation: aspects of recent work.
\newblock In {\em Dynamics and analytic number theory}, volume 437 of {\em
  London Math. Soc. Lecture Note Ser.}, pages 1--95. Cambridge Univ. Press,
  Cambridge, 2016.

\bibitem{MR2457266}
Victor Beresnevich and Sanju Velani.
\newblock A note on zero-one laws in metrical {D}iophantine approximation.
\newblock {\em Acta Arith.}, 133(4):363--374, 2008.

\bibitem{MR2576284}
Victor Beresnevich and Sanju Velani.
\newblock Classical metric {D}iophantine approximation revisited: the
  {K}hintchine-{G}roshev theorem.
\newblock {\em Int. Math. Res. Not. IMRN}, 2010(1):69--86, 2010.

\bibitem{MR1829381}
V.~Bernik, D.~Kleinbock, and G.~A. Margulis.
\newblock Khintchine-type theorems on manifolds: the convergence case for
  standard and multiplicative versions.
\newblock {\em Internat. Math. Res. Notices}, 2001(9):453--486, 2001.

\bibitem{MR1727177}
V.~I. Bernik and M.~M. Dodson.
\newblock {\em Metric {D}iophantine approximation on manifolds}, volume 137 of
  {\em Cambridge Tracts in Mathematics}.
\newblock Cambridge University Press, Cambridge, 1999.

\bibitem{MR2981929}
Ryan Broderick, Lior Fishman, Dmitry Kleinbock, Asaf Reich, and Barak Weiss.
\newblock The set of badly approximable vectors is strongly {$C^1$}
  incompressible.
\newblock {\em Math. Proc. Cambridge Philos. Soc.}, 153(2):319--339, 2012.

\bibitem{MR2451005}
Viveck~R. Cadambe and Syed~Ali Jafar.
\newblock Interference alignment and degrees of freedom of the {$K$}-user
  interference channel.
\newblock {\em IEEE Trans. Inform. Theory}, 54(8):3425--3441, 2008.

\bibitem{MR2753601}
Yitwah Cheung.
\newblock Hausdorff dimension of the set of singular pairs.
\newblock {\em Ann. of Math. (2)}, 173(1):127--167, 2011.

\bibitem{MR3544282}
Yitwah Cheung and Nicolas Chevallier.
\newblock Hausdorff dimension of singular vectors.
\newblock {\em Duke Math. J.}, 165(12):2273--2329, 2016.

\bibitem{MR1043706}
S.~G. Dani.
\newblock On badly approximable numbers, {S}chmidt games and bounded orbits of
  flows.
\newblock In {\em Number theory and dynamical systems ({Y}ork, 1987)}, volume
  134 of {\em London Math. Soc. Lecture Note Ser.}, pages 69--86. Cambridge
  Univ. Press, Cambridge, 1989.

\bibitem{dani1985divergent}
Shrikrishna~Gopal Dani.
\newblock Divergent trajectories of flows on homogeneous spaces and diophantine
  approximation.
\newblock {\em Journal f{\"u}r die reine und angewandte Mathematik},
  1985(359):55--89, 1985.

\bibitem{das2019variational}
Tushar Das, Lior Fishman, David Simmons, and Mariusz Urbański.
\newblock A variational principle in the parametric geometry of numbers.
\newblock arXiv:1901.06602, 2019.

\bibitem{MR166154}
H.~Davenport.
\newblock A note on {D}iophantine approximation. {II}.
\newblock {\em Mathematika}, 11:50--58, 1964.

\bibitem{MR279040}
H.~Davenport and W.~M. Schmidt.
\newblock Dirichlet's theorem on diophantine approximation. {II}.
\newblock {\em Acta Arith.}, 16:413--424, 1969/70.

\bibitem{MR0272722}
H.~Davenport and Wolfgang~M. Schmidt.
\newblock Dirichlet's theorem on diophantine approximation.
\newblock In {\em Symposia {M}athematica, {V}ol. {IV} ({INDAM}, {R}ome,
  1968/69)}, pages 113--132. Academic Press, London, 1970.

\bibitem{MR4859}
R.~J. Duffin and A.~C. Schaeffer.
\newblock Khintchine's problem in metric {D}iophantine approximation.
\newblock {\em Duke Math. J.}, 8:243--255, 1941.

\bibitem{MR1102677}
Kenneth Falconer.
\newblock {\em Fractal geometry}.
\newblock John Wiley \& Sons, Ltd., Chichester, 1990.
\newblock Mathematical foundations and applications.

\bibitem{MR3980277}
Arijit Ganguly and Anish Ghosh.
\newblock Quantitative {D}iophantine approximation on affine subspaces.
\newblock {\em Math. Z.}, 292(3-4):923--935, 2019.

\bibitem{MR2156655}
Anish Ghosh.
\newblock A {K}hintchine-type theorem for hyperplanes.
\newblock {\em J. London Math. Soc. (2)}, 72(2):293--304, 2005.

\bibitem{MR2805874}
Anish Ghosh.
\newblock Diophantine exponents and the {K}hintchine {G}roshev theorem.
\newblock {\em Monatsh. Math.}, 163(3):281--299, 2011.

\bibitem{MR2812651}
Anish Ghosh.
\newblock A {K}hintchine-{G}roshev theorem for affine hyperplanes.
\newblock {\em Int. J. Number Theory}, 7(4):1045--1064, 2011.

\bibitem{MR3310488}
Anish Ghosh and Robert Royals.
\newblock An extension of the {K}hinchin-{G}roshev theorem.
\newblock {\em Acta Arith.}, 167(1):1--17, 2015.

\bibitem{Groshev}
A.~Groshev.
\newblock A theorem on a system of linear forms.
\newblock {\em Dokl. Akad. Nauk SSSR}, 19:151–152, 1938.

\bibitem{JafarBook}
Syed~A Jafar.
\newblock {\em Interference alignment--{A} new look at signal dimensions in a
  communication network}.
\newblock Now Publishers, Inc., 2011.

\bibitem{MR2446746}
Syed~A. Jafar and Shlomo Shamai.
\newblock Degrees of freedom region of the {MIMO} {$X$} channel.
\newblock {\em IEEE Trans. Inform. Theory}, 54(1):151--170, 2008.

\bibitem{MR3736492}
S.~Kadyrov, D.~Kleinbock, E.~Lindenstrauss, and G.~A. Margulis.
\newblock Singular systems of linear forms and non-escape of mass in the space
  of lattices.
\newblock {\em J. Anal. Math.}, 133:253--277, 2017.

\bibitem{MR1512207}
A.~Khintchine.
\newblock Einige {S}\"{a}tze \"{u}ber {K}ettenbr\"{u}che, mit {A}nwendungen auf
  die {T}heorie der {D}iophantischen {A}pproximationen.
\newblock {\em Math. Ann.}, 92(1-2):115--125, 1924.

\bibitem{MR1982150}
D.~Kleinbock.
\newblock Extremal subspaces and their submanifolds.
\newblock {\em Geom. Funct. Anal.}, 13(2):437--466, 2003.

\bibitem{MR1652916}
D.~Y. Kleinbock and G.~A. Margulis.
\newblock Flows on homogeneous spaces and {D}iophantine approximation on
  manifolds.
\newblock {\em Ann. of Math. (2)}, 148(1):339--360, 1998.

\bibitem{MR2679461}
D.~Y. Kleinbock, G.~A. Margulis and J. Wang.
\newblock Metric {D}iophantine approximation for systems of linear forms via dynamics.
\newblock {\em Int. J. Number Theory}, 6(5):1139--1168, 2010.


\bibitem{MR1928528}
Dmitry Kleinbock, Nimish Shah, and Alexander Starkov.
\newblock Dynamics of subgroup actions on homogeneous spaces of {L}ie groups
  and applications to number theory.
\newblock In {\em Handbook of dynamical systems, {V}ol. 1{A}}, pages 813--930.
  North-Holland, Amsterdam, 2002.

\bibitem{MR3975500}
Dmitry Kleinbock and Nick Wadleigh.
\newblock An inhomogeneous {D}irichlet theorem via shrinking targets.
\newblock {\em Compos. Math.}, 155(7):1402--1423, 2019.

\bibitem{MR2366229}
Dmitry Kleinbock and Barak Weiss.
\newblock Dirichlet's theorem on {D}iophantine approximation and homogeneous
  flows.
\newblock {\em J. Mod. Dyn.}, 2(1):43--62, 2008.

\bibitem{Kouk-May}
Dimitris Koukoulopoulos and James Maynard.
\newblock On the {D}uffin-{S}chaeffer conjecture.
\newblock {\em Ann. of Math. (2)}, to appear.
\newblock arXiv:1907.04593, 2019.

\bibitem{MR2451007}
Mohammad~Ali Maddah-Ali, Abolfazl~Seyed Motahari, and Amir~Keyvan Khandani.
\newblock Communication over {MIMO} {X} channels: interference alignment,
  decomposition, and performance analysis.
\newblock {\em IEEE Trans. Inform. Theory}, 54(8):3457--3470, 2008.

\bibitem{LayeredMotahari}
Seyyed~Hassan Mahboubi, Abolfazl~Seyed Motahari, and Amir~Keyvan Khandani.
\newblock Layered interference alignment: achieving the total {DOF} of {MIMO}
  {X}-channels.
\newblock In {\em 2010 IEEE International Symposium on Information Theory},
  pages 355--359. IEEE, 2010.

\bibitem{MR1333890}
Pertti Mattila.
\newblock {\em Geometry of sets and measures in {E}uclidean spaces}, volume~44
  of {\em Cambridge Studies in Advanced Mathematics}.
\newblock Cambridge University Press, Cambridge, 1995.
\newblock Fractals and rectifiability.

\bibitem{MGK}
Abolfazl~Seyed Motahari, Shahab~Oveis Gharan, and Amir~Keyvan Khandani.
\newblock Real {I}nterference {A}lignment with {R}eal {N}umbers.
\newblock arXiv:0908.1208, 2009.

\bibitem{MR3245356}
Abolfazl~Seyed Motahari, Shahab Oveis-Gharan, Mohammad-Ali Maddah-Ali, and
  Amir~Keyvan Khandani.
\newblock Real interference alignment: exploiting the potential of single
  antenna systems.
\newblock {\em IEEE Trans. Inform. Theory}, 60(8):4799--4810, 2014.

\bibitem{NM13}
Urs Niesen, Mohammad Ali Maddah-Ali.
\newblock Interference Alignment: From Degrees of Freedom to Constant-Gap Capacity Approximations
\newblock {\em IEEE Transactions on Information Theory}, 59(8):4855--4888, 2013.

\bibitem{MR3215324}
Or~Ordentlich, Uri Erez, and Bobak Nazer.
\newblock The approximate sum capacity of the symmetric {G}aussian {$K$}-user
  interference channel.
\newblock {\em IEEE Trans. Inform. Theory}, 60(6):3450--3482, 2014.

\bibitem{MR1512000}
Oskar Perron.
\newblock \"{U}ber diophantische {A}pproximationen.
\newblock {\em Math. Ann.}, 83(1-2):77--84, 1921.

\bibitem{Schmidt-1980}
W.~M. Schmidt.
\newblock {\em Diophantine {Approximation}}.
\newblock Springer-Verlag, Berlin and New York, 1980.

\bibitem{MR118711}
Wolfgang Schmidt.
\newblock A metrical theorem in diophantine approximation.
\newblock {\em Canadian J. Math.}, 12:619--631, 1960.

\bibitem{MR195595}
Wolfgang~M. Schmidt.
\newblock On badly approximable numbers and certain games.
\newblock {\em Trans. Amer. Math. Soc.}, 123:178--199, 1966.

\bibitem{MR248090}
Wolfgang~M. Schmidt.
\newblock Badly approximable systems of linear forms.
\newblock {\em J. Number Theory}, 1:139--154, 1969.

\bibitem{schmidt1996diophantine}
Wolfgang~M Schmidt.
\newblock {\em Diophantine approximation}.
\newblock Springer Science \& Business Media, 1996.

\bibitem{MR2534098}
Nimish~A. Shah.
\newblock Equidistribution of expanding translates of curves and {D}irichlet's
  theorem on {D}iophantine approximation.
\newblock {\em Invent. Math.}, 177(3):509--532, 2009.

\bibitem{MR28549}
Claude~E. Shannon.
\newblock Communication in the presence of noise.
\newblock {\em Proc. I.R.E.}, 37:10--21, 1949.

\bibitem{MR548467}
Vladimir~G. Sprind\v{z}uk.
\newblock {\em Metric theory of {D}iophantine approximations}.
\newblock V. H. Winston \& Sons, Washington, D.C.; A Halsted Press Book, John
  Wiley \& Sons, New York-Toronto, Ont.-London, 1979.
\newblock Translated from the Russian and edited by Richard A. Silverman, With
  a foreword by Donald J. Newman, Scripta Series in Mathematics.

\bibitem{MR3990196}
Lei Yang.
\newblock Badly approximable points on manifolds and unipotent orbits in
  homogeneous spaces.
\newblock {\em Geom. Funct. Anal.}, 29(4):1194--1234, 2019.

\end{thebibliography}

\end{document}